\documentclass[reqno, 12pt]{amsart}
 
\usepackage[T1]{fontenc}
\usepackage[utf8]{inputenc}

\usepackage{a4wide}
\usepackage{xcolor}

\RequirePackage{color}
\definecolor{myred}{rgb}{0.75,0,0}
\definecolor{mygreen}{rgb}{0,0.5,0}
\definecolor{myblue}{rgb}{0,0,0.65}
\usepackage{cite}

\usepackage{amsfonts, amsthm, amsmath, amssymb,bm}

\numberwithin{equation}{section} 

\newtheorem{theorem}{Theorem}[section]
\newtheorem{lemma}[theorem]{Lemma}
\newtheorem{cor}[theorem]{Corollary}

\newtheorem{proposition}[theorem]{Proposition} 
\newtheorem{definition}[theorem]{Definition}
\newtheorem{conjecture}[theorem]{Conjecture}

\theoremstyle{definition}

\newcommand{\R}{\mathbb R}
\newcommand{\C}{\mathbb C}
\renewcommand{\phi}{\varphi}

\newcommand{\T}{\mathcal T}
\newcommand{\V}{\mathbb V}
\newcommand{\E}{\mathbb E}

\newcommand{\cM}{\mathcal{M}}
\newcommand{\cV}{\mathcal{V}}

\def\sums{\mathop{\sum \Bigl.^{*}}\limits}
\def\sumb{\mathop{\sum \Bigl.^{\flat}}\limits}

\newcommand{\ord}{{\rm ord}}

\renewcommand{\leq}{\leqslant}

\renewcommand{\geq}{\geqslant}

\renewcommand{\R}{\mathbb R}

\renewcommand{\d}{{\rm d}}
 \newcommand{\e}{{\rm e}}

\renewcommand{\S}{\mathcal S}

\newcommand{\U}{\mathcal U}

\newcommand{\bchi}{\boldsymbol{\chi}}

\newcommand{\bfsigma}{\boldsymbol{\sigma}}
\newcommand{\bfgamma}{\boldsymbol{\gamma}}

\DeclareMathOperator{\Mod}{mod}
\renewcommand{\bmod}[1]{\,(\Mod{#1})}
\renewcommand{\rho}{\varrho}
\def\sums{\mathop{\sum \Bigl.^{*}}\limits}

\renewcommand{\i}{\mathrm{i}}
\newcommand{\eps}{\varepsilon}

\theoremstyle{plain}
\newtheorem{X}{X}[section] 
\theoremstyle{remark}
\newtheorem{remark}[X]{Remark}

\begin{document}
\title[
]{Moments of moments of primes in arithmetic progressions}

\date{\today}
 
\author{R.\ de la Bret\`eche}
\address{
Institut de Math\'ematiques de Jussieu-Paris Rive Gauche\\
Universit\'e de Paris, Sorbonne Universit\'e, CNRS UMR 7586\\
Case Postale 7012\\
F-75251 Paris CEDEX 13\\ France}
\email{regis.delabreteche@imj-prg.fr}

\author{D. Fiorilli}
\address{CNRS, Universit\'e Paris-Saclay, Laboratoire de math\'ematiques d'Orsay, 91405, Orsay, France}
\email{daniel.fiorilli@universite-paris-saclay.fr}

\maketitle

\begin{abstract}
We establish unconditional $\Omega$-results for all weighted even moments of primes in arithmetic progressions.
We also study the moments of these moments and establish lower bounds under GRH. Finally, under GRH and LI we prove an asymptotic for all moments of the associated limiting distribution, which in turn indicates that our unconditional and GRH results are essentially best possible. Using our probabilistic results, we formulate a conjecture on the moments with a precise associated range of validity, which we believe is also best possible. This last conjecture implies a $q$-analogue of the Montgomery-Soundararajan conjecture on the Gaussian distribution of primes in short intervals. The ideas in our proofs include a novel application of positivity in the explicit formula and the combinatorics of arrays of characters which are fixed by certain involutions. 
\end{abstract}

\section{Introduction}

In addition to having a rich history spanning decades, the study of moments of arithmetical sequences in residue classes is experiencing accelerating and major recent progress. For instance, starting from the works of Liu~\cite{L93,L95}, Perelli~\cite{P95}, Friedlander and Goldston~\cite{FG96}, Hooley~\cite{H00,H03,H04} and culminating with the work of Harper-Soundararajan~\cite{HS17} and its generalization by Mastrostefano~\cite{M20}, we now have a general unconditional lower bound for the variance of primes and generalized divisor functions up to $x$ in arithmetic progressions modulo $q$, on average over $q\leq Q$ and in the range $x^{\frac 12+\eps}\leq Q\leq x$. These questions have their origin in the foundational work of Barban~\cite{B66}, Davenport-Halberstam~\cite{DH66}, Montgomery~\cite{M70} and Hooley~\cite{H75a}, as well as in the earlier work of Selberg~\cite{S43} and others on primes in short intervals. There is a myriad of other related works which would be relevant to mention here, including work on varations of the Montgomery-Hooley asymptotic formula~\cite{H75a,H74,H75b,H75c,H76,H77,H98,H00,H02,H03,H04,GV95,H13a,LPZ10,S09a,S10,S11}, relations with pair correlation and other statistics on zeros of $L$-functions~\cite{BKS16,CLLR14,L08}, spaced moduli~\cite{B14a,B14b,BF18,BW11,E01,MP05}, other arithmetic functions~\cite{B08,B18,BF16,C75,GMRR20,K95,LZ12,M06,M73,N15,N20,PV18,RS18,V98,W80,Y14,Y16}, as well as function field analogues \cite{HK19,KR14,KR16, RG17}.

In parallel, we have been experiencing intensive ongoing work on moments of $L$-functions in families, guided by the random matrix theory analogues~\cite{KS00a,KS00b} and the conjectures~\cite{CG98,CF99,CG01,CFKRS05}. In this direction we mention the fundamental work of Ramachandra~\cite{R78, R80} and Heath-Brown~\cite{HB81}, as well as the relatively recent Rudnick-Soundararajan breakthrough \cite{RS05,RS06}, which inspired much further work~\cite{AF12,RS13,CL13,A16,S18} on lower bounds. As for upper bounds, we mention Soundararajan's landmark result~\cite{S09} which was sharpened to the conjectured order of magnitude by Harper~\cite{H13b}, as well as some of the many other papers on the subject~\cite{K10,HB10,Y11,CIS12,RS15,M17,CL17,KY17}.

Much less is known about higher moments of arithmetical sequences in residue classes, with a notable family of exceptions. 
Indeed, in the influential work~\cite{FGKM14}, Fouvry, Ganguly, Kowalski and Michel have computed all moments of the classical divisor function in a certain range. This work has inspired further developments including generalizations to coefficients of GL($N$) automorphic forms~\cite{KR14b, LY16, X17a, ZW18, SZZ19}.  Moreover, very recently Nunes has established upper bounds on higher moments of squarefree integers~\cite{N20}.
As for the sequence $\Lambda(n)$ (the von Mangoldt function), the third moment has been computed by Hooley~\cite{H98}, as well as by Vaughan~\cite{V03b} with a major arcs approximation. Hooley has also formulated a conjecture on higher moments~\cite{H77} on average over $q\leq Q$ in the restricted range $x/(\log x)^A\leq Q = o(x/\log x) $. 
Other than this conjecture, we are not aware of any results on higher moments of $\Lambda(n)$ in arithmetic progressions.

In this paper, we put forward a new technique which will allow us, through positivity and involved combinatorics, to establish lower bounds on average for all weighted moments of $\Lambda(n)$ in arithmetic progressions. More precisely, we will establish lower bounds on all moments of moments. The conceptual idea of studying moments of moments has its roots in Cramer's result~\cite{C22} on the variance of $\psi(x)-x$, Hooley's bound~\cite{H76} on the $t$-average of the variance of $\psi(\e^t;q,a)$, Vaughan's upper bound~\cite{V01} on the higher $q$-moments of the variance of $\psi(x;q,a)$, the second author's work~\cite{F15b} on the moments of the limiting distribution of the variance, as well as the very recent work of Assiotis, Bailey and Keating~\cite{AK19,ABK19,BK19,BK20} on moments of moments in the context of families of $L$-functions and of characteristic polynomials of large random matrices. The approach in the present paper is to study the higher moments of $\psi(\e^t;q,a)$ via its $t$-moments for a given modulus $q$; in other words, we consider a double moment, over $a \bmod q$ and over $t\leq T$. 
As a side result, we are also able to study the fixed residue class $a=1$, once more on average over $t\leq T$. Note also that all our results apply to individual moduli $q$, 
rather than to the more usual average over $q$. 

As pointed out by Hooley~\cite{H02}, it is somewhat paradoxical that the GRH allows one to obtain an asymptotic for $\psi(x;q,a)$ for individual moduli $q\leq x^{\frac 12-\eps}$, while asymptotics for the variance of $\psi(x;q,a)$ require that $ x^{\frac 12+\eps} \leq q \leq x $. This last barrier has recently been broken over function fields in the influential work of Keating and Rudnick~\cite{KR14}. However, as far as we know, it still stands over $\mathbb Q$. Interestingly, with the current methods we were able to obtain lower bounds of the conjectured order of magnitude for moduli that are of size $x^{\delta}$ for some small enough $0<\delta<\tfrac 12$ (in fact, $x^{\delta}$ can be replaced a smaller function tending to infinity, no matter how slowly). 

The moments of $\Lambda(n)$ are strongly linked with the spacing properties of zeros of $L$-functions (see for instance~\cite{C22,CLLR14,GM87}), and \textit{a fortiori} 
to the Diophantine properties of imaginary parts of zeros of $L(s,\chi)$. An analogous relation was uncovered in the Monach--Montgomery determination of the exact order of magnitude of the error term in the prime number theorem~\cite[p.484]{MV07}, which depends on an effective version of the linear independence hypothesis on zeros of $\zeta(s)$. Having to assume such a hypothesis is a major drawback in the application of the explicit formula to this context. Nevertheless, in the current work we were able to overcome this hypothesis
through an application of positivity. 
As a result, we obtain optimal lower bounds that are independent of the Diophantine properties of zeros of $L$-functions. Moreover, by combining these results with an approach pioneered by Littlewood~\cite{L14}, we  obtain unconditional and conjecturally optimal $\Omega$-results on all even moments.

Before we state our results, we need to introduce some notation. For $\eta \in \mathcal L^1(\mathbb R)$, we recall the usual definition of the Fourier transform
$$\widehat \eta(\xi) := \int_{\R} \eta(t)\e^{-2\pi i \xi t}  \d t. $$
As in the work of Fouvry, Ganguly, Kowalski and Michel~\cite{FGKM14}, we will count $\Lambda(n)$ with a weight function. More precisely, for any fixed $\delta>0$, we define 
$\S_\delta\subset \mathcal L^1(\mathbb R)$ to be the set of all non-trivial differentiable even $ \eta :\R \rightarrow \R$ such that for all $t\in \R$,
$$\eta (t) ,  \eta'(t) \ll \e^{-(\frac 12 +\delta) |t|}, $$
and moreover for all $\xi \in \R$, we have that\footnote{The upper bound on $\widehat \eta(\xi)$ is a quite mild condition given the differentiability of $\eta$; going through the proof of the Riemann-Lebesgue lemma 
we see for instance that a stronger bound holds as soon as $\eta'$ is monotonous. (A stronger bound holds if $\eta$ is twice differentiable.) As for the positivity condition, we can take for example $\eta = \eta_1 \star \eta_1$ for some smooth and rapidly decaying $\eta_1$.}
 \begin{equation}
 \label{condiavecdelta}
 0 \leq \widehat \eta(\xi) \ll (|\xi|+1)^{-1} (\log(|\xi|+2))^{-2-\delta}. 
 \end{equation}
 For $a,q\in \mathbb N$, we define
$$ \psi_{\eta}(x;q,a):=\sum_{\substack{n\geq 1 \\ n\equiv a \bmod q}} \frac{\Lambda(n)}{n^{\frac 12}} \eta(\log (n/x)); \qquad \psi_{\eta}(x,\chi _{0,q}):=\sum_{\substack{n\geq 1 \\ (n,q)=1}} \frac{\Lambda(n)}{n^{\frac 12}} \eta(\log (n/x)). $$
The effect of this choice of test function is the concentration of the mass on those $n$ for which $n\asymp x $.  As an example of element of $\S_\delta$, one can take  
$\eta_K(t) = \e^{-K |t|}$ with 
$K\geq\tfrac 12+\delta$, for which $\widehat \eta(\xi) = 2K / (K^2+(2\pi \xi)^2)$.
Note also that the classical prime counting functions correspond to the choice  $\eta_0(t):= \e^{ \frac t2}1_{t\leq 0}\notin \S_{\delta}$, for which $\psi_{\eta_0}(x;q,a)=\psi(x;q,a)/\sqrt{x}$ and $\widehat {\eta_0}(\xi) = (\tfrac 12- 2\pi i\xi)^{-1}$.

The object of study in this paper is the moment\footnote{Note that the moments are usually not divided out by $\phi(q)$; this normalization should be kept in mind while comparing our results to those in the literature. Moreover, since $\psi_{\eta}(x,\chi_{0,q})$ is of order $x^{\frac 12}$, in a sense we have also normalized by $x^{\frac n2}$.}
\begin{equation}
M_n(x,q;\eta)  = \frac{1}{\phi(q)}\sum_{\substack{a \bmod q\\ (a,q)=1}} \Big(\psi_{\eta}(x;q,a)-\frac {\psi_{\eta}(x,\chi_{0,q})}{\phi(q)}\Big)^n.
\end{equation}
We define the constants $\alpha(h):=\widehat h(0)$ and
\begin{equation} \beta_q(h):=
-\widehat h(0)\Big(\log(8\pi) +\gamma_0+  \sum_{p \mid q } \frac{\log p}{p-1}\Big)+\frac{1 }{2}\int_0^\infty 
\frac{\e^{-  \frac x2} }{1-\e^{-  x}}
\big(2\widehat h(0)-\widehat h(x)-\widehat h(-x)\big)\d x,
\label{equation definition alpha beta}
\end{equation}
where $\gamma_0$ is the Euler-Mascheroni constant.
Here are our main unconditional results. 
\begin{theorem}
\label{thunconditional n=2}
 Let $\delta>0,$ $\eta\in \S_\delta$, fix $\eps>0$, and let $g:\mathbb R_{\geq 1} \rightarrow \R_{\geq 3}$ be any increasing function tending to infinity for which $g(u) \leq \e^{u}$. 
 For a positive proportion of moduli $q$, there exists an associated value $x_q$ such that 
$  g(c_1(\eps,\delta,\eta)\log x_q) \leq q \leq  g(c_2(\eps,\delta,\eta)\log x_q) $, where the $c_j(\eps,\delta,\eta)>0$ are constants, and
\begin{equation}  \label{eq:unconditionalbound}
M_{2}(x_q,q;\eta) \geq (1-\eps) \frac{\alpha(\widehat \eta^2) \log q}{\phi(q)}. 
\end{equation}
If moreover $g(u) \leq  u^{\frac 1M}$ for some $M\geq 4\eps^{-1}$, then we can choose
$x_q$ in such a way that $g((\log x_q)^{1-\frac 3{M\eps}}) \leq q \leq  g(\log x_q) $ and
\begin{equation}
     M_{2}(x_q,q;\eta) \geq   \frac{\alpha(\widehat \eta^2)\log q +\beta_{q}(\widehat \eta^2)}{\phi(q)}+ 
\frac{ 1}{q^{   \frac 32 +\eps}}. 
\label{equation omega result M_2}
\end{equation}

\end{theorem}
\begin{remark}
\begin{enumerate}
    \item The expression $ (\alpha(\widehat \eta^2)\log q  +\beta_{q}(\widehat \eta^2))/\phi(q ) $ is the expected main term for $ M_{2}(x ,q ;\eta)$ (see Conjecture~\ref{conjecturemoments} below).
    
    \item One can obtain an $\Omega$-result for $ M_{2}(x_q,q;\eta) - (\alpha(\widehat \eta^2)\log q +\beta_{q}(\widehat \eta^2))/\phi(q)$ which is slightly stronger than~\eqref{equation omega result M_2}, but with and undetermined  sign. More precisely, the term $1/q^{\frac 32+\eps}$ can be replaced with a constant times $\log q/\phi(q)^{\frac 32}$.
    
    \item Goldston and Vaughan~\cite{GV95} have shown\footnote{Note that in \cite[Theorem 1.1]{GV95}, $-cxQ$ should be replaced by $+cxQ$.} that under GRH and in the range $x^{\frac 57+\eps}\leq Q \leq x^{1-\eps}$, 
      $$\frac 1{Q^2} \sum_{q\leq Q} \Big( \sum_{\substack{a\bmod q \\ (a,q)=1}} \frac 1x \Big( \sum_{\substack{n\leq x \\ n\equiv a \bmod q}}\Lambda(n) - \frac {x}{\phi(q)}\Big)^2 -\log q -\beta_{q}(|\widehat\eta_0|^2) \Big) = \Omega_{\pm}( x^{-\frac 34}Q^{-\frac 14}),  $$
   where $\beta_{q}(|\widehat\eta_0|^2)= -(\gamma_0+\log(2\pi)+\sum_{p\mid q} \frac{\log p}{p-1})$. 
\end{enumerate}

\end{remark}
In a certain range, we have similar results for all even moments.
Here and throughout, we will denote the $n$-th moment of the Gaussian by
\begin{equation}
    \mu_n:= \begin{cases}
    \frac{(2m)!}{2^m m!} & \text{ if } n=2m \text{ for some } m\in \mathbb N,\\
    0 &\text{otherwise.}
    \end{cases}
\end{equation}

\begin{theorem} 
\label{thunconditional n>2}
Let $\delta>0$, $\eta\in \S_\delta$, and fix $m\geq 2$. 
For any $B\in \mathbb R_{ > 0}$, there exists a real number $\lambda \in [\frac 1{m-1+B},\frac 1B]$ with the following property. For any fixed $\eps >0$ there exists a sequence of moduli $\{q_i\}_{i\geq 1}$ and associated values $\{x_i\}_{i\geq 1}$ such that 
$q_i= (\log x_i)^{\lambda+o(1)} $ and
\begin{equation}\label{eq:unconditionalbound m>2}
M_{2m}(x_i,q_i;\eta) \geq (1-\eps)\mu_{2m}\frac{(\alpha(\widehat \eta^2)\log q_i)^m}{\phi(q_i)^m}.  
\end{equation}
If moreover $g:\mathbb R_{\geq 1} \rightarrow \R_{\geq 3}$ is an increasing function such that $ g(x) \leq x^{\frac 1{mM}}$ with $M\geq 2\eps^{-1}$, then for a positive proportion of $q$ there exists 
$x_q$ such that $g((\log x_q)^{1-\frac{1}{\eps M}}) \leq q \leq  g(\log x_q) $ and
\begin{equation}
    M_{2m}(x_i,q_i;\eta) \geq   \mu_{2m}\Big( \frac{\alpha(\widehat \eta^2)\log q_i +\beta_{q_i}(\widehat \eta^2)}{\phi(q_i)}\Big)^m+ \frac{ 1}{q_i^{  m+\frac 12 +\eps }}. 
    \label{equation more precise omega result high moments}
\end{equation}
\end{theorem}

\begin{remark}
\begin{enumerate}
    \item Under the additional condition $ q\geq (\log_2 x)^{1+\eps}$, we believe that
    Theorems~\ref{thunconditional n=2} and~\ref{thunconditional n>2} 
    are essentially best possible (see Conjecture~\ref{conjecturemoments}). However, in the range $ q\leq \eps \log_2 x$ we believe that $M_{2m}(x,q;\eta)$ can be larger (see Remark~\ref{remark very small q}).

\item
Under GRH, we will show in Section~\ref{section unconditional applications} that analogues of Theorems~\ref{thunconditional n=2} and~\ref{thunconditional n>2} are valid for \emph{all} moduli $q$. For instance, it follows from Lemma~\ref{lemma raw oscillations higher moments} that for every large enough $q$, there exists an associated value $x_q$ such that $\log x_q\asymp_{\eps,\eta} \log q$ and
$$
M_{2}(x_q,q;\eta) \geq (1-\eps)\frac{\alpha(\widehat \eta^2) \log q}{\phi(q)}.  
$$
Moreover, one can obtain a version of Theorem~\ref{thunconditional n>2} (still unconditional) which holds for a positive proportion of moduli $q$ in the range $ \log q /\log\log x \asymp_m 1 $.

\item Under GRH, one can also obtain $\Omega$-results for the odd moments $M_{2m+1}(x,q;\eta)$ (using this time even moments, that is the lower bound on $\cV_{2r,2m+1}(T,q;\eta,\Phi)$ in Theorem~\ref{thmomentscentres2rT} below). 

\end{enumerate}
\end{remark}

Theorems~\ref{thunconditional n=2} and~\ref{thunconditional n>2} are applications of our estimates for higher moments of moments, which we now describe.
We define the expected value
\begin{equation}
    m_n(q;\eta):= \lim_{T\rightarrow \infty} \frac 1T \int_0^T M_n(\e^t,q;\eta) \d t, 
  \label{equation definition mean of M_n}
\end{equation}
  whenever\footnote{We will see in Lemma~\ref{lemma moments of limiting distribution as limit of V(T)} that GRH implies the existence of this limit.} the limit exists. 
We consider the set $\U\subset \mathcal L^1(\R)$ of non-trivial even integrable functions $\Phi:\R \rightarrow \R$ such that $\widehat \Phi\geq 0$.  As a consequence,
we have that $\Phi(0)>0$. For $s,n\in \mathbb N$ such that~\eqref{equation definition mean of M_n} exists, we define\footnote{It is crucial here that we subtract $m_n(q;\eta)$ rather than the empirical mean $\frac 1{T} \int_0^TM_n(\e^t,q;\eta)$; in the latter case the resulting combinatorial expression is not well suited for an application of positivity of $\widehat \eta$, $\widehat \Phi$.}
\begin{equation}
     \label{defV2rnT}
\cV_{s,n}(T,q;\eta,\Phi):=\frac {1}{T \int_0^{\infty} \Phi} \int_{0}^{\infty } \Phi\Big( \frac tT \Big) \Big(M_n(\e^t,q;\eta)-m_n(q;\eta)\Big)^{s}\d t.
\end{equation}
We also define the 
quantities\footnote{ As we will see in~\eqref{equation asymptotic high even moments LI}, $V_n(q;\eta)$ is the (conjectural) asymptotic variance of $M_n(\e^t,q,\eta)$.}
\begin{equation}
\label{defKn} \nu_n:=n!^2 \sum_{\substack{(k,m)\\ k\geq 2,m\geq 0\\    n=k+2m}}   \frac 1{k!2^{2m} m!^2};\hspace{2cm}  V_n(q;\eta):=\frac{\nu_n (\alpha(\widehat \eta^2)\log q)^{  n}}{\phi(q)^{  n+1}}.
 \end{equation}

We are ready to state our main theorem.

\goodbreak
\begin{theorem}
\label{thmomentscentres2rT} Assume GRH, and let $\delta>0$, $\eta\in \S_\delta, \Phi\in \U$, $q\geq 3$, $T\geq 1$, $ n\geq 2,$ $k,r\geq 1$. Then, the limit~\eqref{equation definition mean of M_n} exists. Moreover, if $q$ is large enough in terms of $\delta$ and $\eta$, then in the ranges $n\leq \log q/\log_2 q$ and $r\leq \log q$ the moment of moment $\cV_{s,n}(T,q;\eta,\Phi)$ satisfies the lower bounds
\begin{equation*}
    \cV_{2r,n}(T,q;\eta,\Phi)\geq \mu_{2r}V_n(q;\eta)^r \Big( \Big( 1+O_{\delta,\eta}\Big( n \frac{\log_2 q}{\log q}\Big)\Big)^r+O_{\delta,\eta}\Big( \frac r{\log q}\Big)\Big)
+ O_\Phi\Big( \frac{(C_{\delta,\eta}\log q)^{2rn} }{ T\phi(q)^{2r}}\Big);
\end{equation*}
\begin{multline*}
    (-1)^n\cV_{2r+1,n}(T,q;\eta,\Phi)\geq  \frac{(2r+1)!}{2^{r} (r-1)!}\frac{\vartheta_{n,2r+1}}{\phi(q)^{\frac 12}} V_n(q;\eta)^{\frac{2r+1}2}\times \\
  \Big( \Big( 1+O_{\delta,\eta}\Big( n \frac{\log_2 q}{\log q}\Big)\Big)^r+O_{\delta,\eta}\Big( \frac r{\log q}\Big)\Big)+  O_\Phi\Big( \frac{(C_{\delta,\eta}\log q)^{(2r+1)n} }{ T\phi(q)^{2r+1}}\Big).
\end{multline*}
Here,
$\vartheta_{n,2r+1}=0$ if $n$ is odd, and otherwise $\vartheta_{n,2r+1}:=\nu_n^{-\frac 32}(\nu_n'+\nu_n''), $ where $\nu_n'$ and $\nu_n''$ are constants of the same nature as $\nu_n$ which are defined in~\eqref{defKn'} and~\eqref{defKn''}, respectively. Moreover, the constant $C_{\delta,\eta}>0$ depends only on $\delta, \eta$.
Finally, the mean of the $n$-th moment $m_{n}(q;\eta)$ satisfies the 
bounds $m_{2k+1}(q;\eta) \leq 0$ and 
$$m_{2k}(q;\eta) \geq  \mu_{2k}\Big( \frac{\alpha(\widehat \eta^2) \log q +\beta_q(\widehat \eta^2) }{\phi(q)}\Big)^k  
+O\Big(   \mu_{2k} \frac{(C_{\delta,\eta}\log q)^{k}}{\phi(q)^{k+1}} \Big).$$
\end{theorem}

\begin{remark}
\begin{enumerate}
    \item In Proposition~\ref{proposition moments a=1}, we also obtain lower bounds for the moments of $\psi_\eta(\e^t;q,1)-\psi_\eta(\e^t,\chi_{0,q})/\phi(q)$; in other words, we may fix the individual residue class $1\bmod q$. 

\item  The conditions $\widehat\eta,\widehat \Phi \geq 0$ are crucial here (recall the definitions of $\S_\delta$ and $\mathcal U$). If they were to be dropped, then as discussed earlier, in order to obtain a lower bound 
 of the strength of Theorem~\ref{thmomentscentres2rT} using the techniques of this paper, one would need to assume an \emph{effective} version of the linear independence hypothesis on Dirichlet $L$-functions. More precisely, it would be sufficient to assume that if $\gamma_1,\dots ,\gamma_n$ are imaginary parts of zeros of height at most $S\geq 1$ of $L(s,\chi_1),\dots L(s,\chi_n) $ respectively with $\chi_j \neq \chi_{0,q}$, then either $|\gamma_1+\dots +\gamma_n| \geq \exp(-(qS)^{1+\eps})$, or $n=2m$ and for each $j$ there exists $\ell\neq j$ for which $\chi_j=\overline{\chi_\ell}$ and $\gamma_j=-\gamma_{\ell}$. This is the $q$-analogue of the Montgomery--Monach conjecture~\cite[(15.24)]{MV07} on zeros of $\zeta(s)$. However, this approach would impose the extra condition $q\leq (\log T)^{1-\eps}$ (in other words, when translating back to $M_{n}(x,q;\eta)$, we would need $q\leq (\log_2 x)^{1-\eps}$). 
\end{enumerate}

\end{remark} 

As another application of our results, we will establish $\Omega$-results for the usual prime counting functions (without smooth weights).

\begin{cor}
\label{corollary montgome}
Assume GRH, and fix $g:\mathbb R_{\geq 1} \rightarrow \mathbb R_{\geq 3}$ an increasing function tending to infinity for which $g(u)\leq \e^{ u}$.
For each large enough modulus $q$ there exists
$a \bmod q $ with $(a,q)=1$ 
and values $x_q$
for which $ g(c_1\log x_q) \leq q \leq  g(c_2\log x_q)$, where $0<c_1<c_2\leq \frac 12$ are absolute, and 
$$ \Bigg| \sum_{\substack{ n\leq x_q \\ n\equiv a\bmod {q}}} \Lambda(n)-\frac{1}{\phi(q)} \sum_{\substack{ n\leq x_q \\ (n,q)=1}} \Lambda(n)\Bigg| \gg \frac{(x_q\log q)^{\frac 12}}{\phi(q)^{\frac 12}}.   $$
Unconditionally, the same holds for a positive proportion of moduli $q$.
\end{cor}

Our method also allows to isolate the individual arithmetic progression $1\bmod q$. We will show in Proposition~\ref{proposition moments a=1} that under GRH, for any $\eta \in \S_\delta$ and $\Phi \in \U$, and in the range $m\leq \log q/(2\log_2 q)$, 
\begin{multline*} \frac 1{T\int_0^{\infty} \Phi}\int_{\mathbb R}  \Phi\Big( \frac tT\Big) \Big( \psi_\eta(\e^t;q,1) - \frac{\psi_\eta(\e^t,\chi_{0,q})}{\phi(q)}\Big)^{2m} \d t \\ \geq \mu_{2m}\Big( \frac{ \alpha(\widehat \eta^2)\log q+\beta_q(\widehat \eta^2) }{\phi(q)}\Big)^m +O_\Phi\Big(\mu_{2m} \frac{ (C_{\delta,\eta}\log q)^{m}}{\phi(q)^{m+1 }}+ \frac{(K_{\delta,\eta}\log q)^{2m}}{T}\Big),
\end{multline*}
where $C_{\delta,\eta},K_{\delta,\eta}>0$ are constants.
We deduce the following.
\begin{cor}
\label{corollary a=1}
Assume GRH, and fix $\eps>0$ small enough and $M> 1+\eps$. 
 Let $g:\mathbb R_{\geq 1} \rightarrow \R_{\geq 3}$ be an increasing function tending to infinity such that $g(u) \leq  u^{\frac 1M}$. Then, for each large enough modulus $q$ there exists an associated value $x_q$ such that $g((\log x_q)^{1-\frac{1}M-2\eps}) \leq q \leq  g(\log x_q) $ and
\begin{equation}
\Bigg| \sum_{\substack{ n\leq x_q\\  n\equiv 1 \bmod q  } } \Lambda(n) -\frac 1{\phi(q)}\sum_{\substack{ n\leq x_q \\ (n,q)=1  } }\Lambda(n) \Bigg| \gg \frac{(x_q\log q)^{\frac 12}}{\phi(q)^{\frac 12}}.
\end{equation}

\end{cor}

In the next theorem we will show under GRH that $M_n(\e^t,q;\eta) $ has a limiting distribution as $t\rightarrow \infty$ (which we will denote by $H_n(q;\eta)$), and estimate the corresponding moments. 
\begin{definition}
\label{definition limiting distribution}
 For $n\geq 1$, we say that 
$E: \mathbb R_{\geq 0} \rightarrow \mathbb R^{n}$ has limiting distribution $\mu$ if $\mu$ is a probability distribution on $\mathbb R$ such that for any bounded continuous $f:\mathbb R^{n} \rightarrow \mathbb R$ we have that
\begin{equation}
    \label{equation definition limiting distribution}
 \lim_{T\rightarrow \infty} \frac 1T \int_0^T f(E(t)) \d t = \int_{\mathbb R^{n}} f \d \mu. 
 \end{equation}
\end{definition}
\noindent We will establish lower bounds for all centered moments 
\begin{equation}
\label{defMtn}
\mathbb V_{s,n}( q;\eta ):=\mathbb E\big[(H_{n}(q;\eta ) -\mathbb E[ H_{n}(q;\eta ) ])^{s}\big].
\end{equation}
Under the additional hypothesis LI below, we will show that these bounds are sharp.
\medskip
\smallskip

\noindent \textbf{Hypothesis LI}: 
For any $q\geq 3$, the multiset $$Z(q):=\Big\{ \Im m(\rho_{\chi}) \geq 0 : L(\rho_{\chi},\chi)=0,\quad \chi\bmod q , \quad \Re e(\rho_\chi)\geq \tfrac 12\Big\}$$ is linearly independent over $\mathbb Q$. 

\begin{remark}
LI implies that the zeros of $L(s,\chi)$ are simple, and that $L(\frac 12,\chi) \neq 0$.
\end{remark}

\begin{theorem}
\label{thmomentscentres2r} Assume GRH, and let $\delta>0$, $\eta\in\S_\delta$, $q\geq 3$,  $ n\geq 2$, $r,m\geq 1$. 
Then, $M_n(\e^t,q;\eta) $ posesses a limiting distribution of as $t\rightarrow \infty$, which will be denoted by $H_n(q;\eta)$.
If moreover $q$ is large enough in terms of $\delta$ and $\eta$, $n\leq \log q/\log_2 q$ and $r\leq \log q$, then the variance defined in~\eqref{defMtn} satisfies the lower bounds
\begin{equation}\label{inegmoment2r}
    \mathbb V_{2r,n}( q;\eta )\geq \mu_{2r}V_n(q;\eta)^r\Big( \Big( 1+O_{\delta,\eta}\Big( n \frac{\log_2 q}{\log q}\Big)\Big)^r+O_{\delta,\eta}\Big( \frac r{\log q}\Big)\Big);
\end{equation}
 \begin{equation}      
    \label{inegmoment2r+1}
     \begin{split}
  (-1)^n  \mathbb V_{2r+1,n}(q;\eta )\geq \frac{(2r+1)!}{2^r (r-1)!}&
\frac{\vartheta_{n,2r+1}}{\phi(q)^{\frac 12}}  V_n(q;\eta)^{\frac{2r+1}2}  \\ &\times\Big( \Big( 1+O_{\delta,\eta}\Big( n \frac{\log_2 q}{\log q}\Big)\Big)^r+O_{\delta,\eta}\Big( \frac r{\log q}\Big)\Big).
\end{split}
 \end{equation}
Assuming in addition that LI holds, we have that  $\mathbb E[H_{2m+1}(q;\eta) ]=0$ and
\begin{equation}
\mathbb E[H_{2m}(q;\eta) ]=  \mu_{2m}\Big( \frac{\alpha(\widehat \eta^2) \log q+\beta_q(\widehat \eta^2)}{\phi(q)}\Big)^m  +O_m\Big(\frac{ (\log q)^{m }}{\phi(q)^{m+1}}\Big), 
\label{equation asymptotic mth moment LI}
\end{equation}
and for higher moments we have the asymptotic estimates
\begin{align}
\label{equation asymptotic high even moments LI}
\mathbb V_{2r,n}( q;\eta )&= \mu_{2r}V_n(q;\eta)^r\Big(1+O_{n,r}\Big(\frac{\log_2q}{\log q}\Big)\Big);
\\
     \mathbb V_{2r+1,n}(q;\eta )&=(-1)^n\frac{(2r+1)!}{2^r (r-1)!}
\frac{\vartheta_{n,2r+1}}{\phi(q)^{\frac 12}} V_n(q;\eta)^{\frac{2r+1}2}\Big(1+O_{n,r}\Big(\frac{\log_2q}{\log q}\Big)\Big).
\label{equation asymptotic high odd moments LI}
\end{align}  
\end{theorem}

\begin{remark}
In the specific case $m=1$, one can obtain an upper bound of the correct order of magnitude in~\eqref{equation asymptotic mth moment LI} assuming GRH only, by using Hooley's arguments~\cite[Theorem 1]{H76}. However, for higher moments, applying the same approach would yield an upper bound which is  $= 1/\phi(q)^{1+o(1)}$.
\end{remark}

In other words, as $q \rightarrow \infty$, $H_n(q;\eta)$ has variance asymptotic to $V_n(q;\eta)$,
and the normalized random variable $(H_n(q;\eta) - \E[H_n(q;\eta)])/\V[H_n(q;\eta) ]^{\frac 12}$ has Gaussian moments up to a certain point. This does not persist for very large moments (in terms of $q$); $H_n(q;\eta)$ is bounded by $(K_{\delta,\eta}\log q)^n/\phi(q)$ for some $K_{\delta,\eta}>0$, and thus $ \mathbb V_{s,n}(q;\eta ) \ll (2K_{\delta,\eta}\log q)^{ns}/\phi(q)^s$. The presence of the constant $\nu_n$ (rather than the variance of the $n$-th power of the Gaussian) in the variance $V_n(q;\eta)$ can be explained by the fact that the values of $\psi_{\eta} (\e^t,q;a)-\psi_{\eta}(\e^t,\chi_{0,q})/\phi(q) $ are not independent as $a$ runs over the invertible residues modulo $q$. Indeed, a calculation similar to~\cite[Lemma 2.1]{L12} shows that the covariances are non-zero.

We deduce the following conjecture which generalizes and refines Hooley's conjecture~\cite{H77}.

\begin{conjecture} 
\label{conjecturemoments}
Let $\varepsilon>0$, $n\geq 1$, $\delta>0$ and $\eta \in \S_\delta$.
In the range $(\log_2 x)^{1+\varepsilon } \leq q \leq x^{1-\eps}$ and for fixed values of $n$ we have the asymptotic 
\begin{equation}
    M_{n}(x,q;\eta) = (\mu_{n}+o_{q,x\rightarrow\infty}(1)) \Big( \frac{\alpha( \eta^2 )\log q}{\phi(q)}\Big)^{\frac n2}.
    \label{equation asymptotic conjecture}
\end{equation}
 
\end{conjecture} 

\begin{remark}
\label{remark very small q}
We also conjecture that~\eqref{equation asymptotic conjecture} holds in the range $(\log_2 x)^{1+\varepsilon } \leq q \leq x^{1-\eps}$ with the function $\eta_0(t) = \e^{\frac t2}1_{t\leq 0}$, that is with the classical weight (since then $\eta_0(\log(n/x))/n^{\frac 12} =x^{-\frac 12} 1_{n\leq x} $). This generalizes~\cite[Conjecture 1.1]{F15b}.
Moreover, the lower bound in the range $(\log_2 x)^{1+\eps} \leq q \leq x^{1-\eps}$ is essentially best possible. Indeed, by~\cite[Theorem 1.2]{FM20} and H\"older's inequality, for $  \eps \log_2 x / \log_3 x \leq q \leq \eps \log_2 x$, we have that the moments $M_{2m}(x,q;\eta_0)$ can be as large as $(\frac{\log q}{\phi(q)})^m (\frac{\log_2 x} q)^m$. Moreover,
one can show that~\eqref{equation asymptotic conjecture} does not hold for $q$ very close to $x$ (indeed in this range $M_{2m}(x,q;\eta)$ can be as large as $(\log q)^{2m-1} /\phi(q)^m$).
\end{remark}

Conjecture~\ref{conjecturemoments} suggests the following subsequent conjecture about the distribution of primes in arithmetic progressions.

\begin{conjecture}
\label{conjecture distribution} Let $\varepsilon>0$.
In the range $ (\log_2 x)^{1+\eps} \leq q\leq x^{1-\eps}$ and for fixed $V\in \mathbb R$,
$$  \frac 1{\phi(q)} \Big|\Big\{ a\bmod q : \psi(x;q,a)-\frac{\psi(x,\chi_{0,q})}{\phi(q)} > V \frac{(x\log q)^{\frac 12}}{\phi(q)^{\frac 12}}  \Big\}\Big| = \int_V^{\infty} \e^{-\frac {u^2}2} \frac{\d u}{\sqrt{2\pi}}+o_{q,x\rightarrow\infty}(1). $$
\end{conjecture}

Conjectures~\ref{conjecturemoments} and~\ref{conjecture distribution} are $q$-analogues of the conjectures of Montgomery and Soundararajan on primes in short intervals~\cite[Conjectures 1,2]{MS04} (see also the recent article~\cite{BF20}).

\section*{Acknowledgements}

The work of the second author was supported at the University of Ottawa by an NSERC discovery grant.

\section{Weil explicit formula} 
\label{section explicit formula}

The goal of this short section is to express $M_n(x,q;\eta)$ as a sum over zeros of Dirichlet $L$-functions.
For $\eta\in \S_\delta$, we will study the moments of $\psi_{\eta}(x;q,a)$ via the explicit formula for 
$$ \psi_{\eta}(x,\chi):= \sum_{n=1}^{\infty}\frac{\Lambda(n)\chi(n)}{n^{\frac 12}} \eta(\log (n/x)),$$
where $\chi \bmod q$ is a Dirichlet character of conductor $q_\chi\geq 3$. By orthogonality,
we have the identity 
\begin{equation}\label{eqMn}\begin{split}
M_n(x,q;\eta)
&= \frac {1}{\phi(q)^{n}} \sum_{\substack{\chi_1,\ldots,\chi_n\neq \chi_{0,q} \\
\chi_1 \cdots \chi_n= \chi_{0,q}}} \psi_{\eta}(x,\chi_1)\cdots \psi_{\eta}(x,\chi_n).
\end{split}\end{equation}
 We recall that $\eta$ is even and real-valued, and hence for all $\xi\in \mathbb R$, $\widehat{\eta}(-\xi) = \overline{\widehat{\eta}(\xi)}=\widehat{\eta}(\xi)$.
 We begin with the following bound on "ramified primes".
\begin{lemma} Let  $\delta>0$ and  $\eta\in\S_\delta$. Let $\chi$ be a non-principal character  modulo $q$ with associated primitive character $\chi^*$. 
For any $t\geq 0$ and $q\geq 3$, we have the bound
$$\psi_{\eta}(\e^t,\chi) -\psi_{\eta}(\e^t,\chi^*)  \ll_{\delta,\eta} \e^{-\frac t2}  \log q.  $$ 
\label{lemma ramified primes}
\end{lemma}  

\begin{proof}
By the triangle inequality,
\begin{align*}\big|\psi_{\eta}(\e^t,\chi) -\psi_{\eta}(\e^t,\chi^*)\big|   \leq  &\sum_{(n,q)>1} \frac{\Lambda(n)}{n^{\frac 12}}|\eta( t-\log  n  )|\cr&= \sum_{p\mid q}    \log p  
\sum_{k\geq 1}\frac{1}{p^{\frac{k}2}}|\eta( t-k\log p )| .\end{align*}
As $\eta(u)\ll \e^{-(\frac 12+\delta) |u|}$, we split the sum over $k$ into two parts, according to whether  $k\log p \geq  t$ or not.  A straightforward calculation shows that both of these sums are $\ll \e^{-\frac t2}$. The claimed bound follows.
\end{proof} 
 
 We are ready to apply the explicit formula.
\begin{lemma}
\label{lemma:explicit formula}
Let $q\geq 3$, let $\chi \bmod q$ be a non-principal Dirichlet character, and let $\delta>0$ and $\eta \in\S_\delta$. Then for $t\geq 0$ we have the formula
\begin{equation}
\psi_{\eta}(\e^t ,\chi)=-
\sum_{\rho_{\chi}} \e^{  (\rho_\chi-\frac 12)  t} \widehat \eta\Big(\frac{\rho_\chi-\frac 12}{2\pi i}\Big)  +O_{\delta,\eta}(\e^{-\frac t2}\log q),
\label{equation explicit formula psi chi}
\end{equation}
where $\rho_\chi$
runs over the non-trivial zeros of $L(s,\chi)$.
\end{lemma}
\begin{proof}
Let $\chi^*$ be the primitive character modulo $q_\chi$ which induces $\chi$, and define as usual $a(\chi) = 0$ when $\chi(-1) = 1$ and $a(\chi) =1$ when $\chi(-1) = -1$.
We will evaluate the sum over zeros on the right hand side of~\eqref{equation explicit formula psi chi}. We can replace $L(s,\chi)$ by $L(s,\chi^*)$, since these have the same non-trivial zeros. We apply Weil's explicit formula for $L(s,\chi^*)$. 
 Taking $F(x):=\eta(2\pi x+t)$ in  \cite[Theorem 12.13]{MV07}, we find that $\widehat F(\frac{\rho_\chi-\frac 12}i)=\frac{1}{2\pi}\e^{(\rho_\chi-\frac 12) t} \widehat \eta(\frac{\rho_\chi-\frac 12}{2\pi i} )$, and it follows that
\begin{equation}
    \begin{split} 
\sum_{\rho_{\chi}} \e^{  (\rho_\chi-\frac 12)  t} \widehat \eta\Big(\frac{\rho_\chi-\frac 12}{2\pi i}\Big)  =&  \Big( \log \frac {q_\chi}{\pi} +\frac{\Gamma'}{\Gamma} \Big( \frac 14+\frac{a(\chi)}2\Big)\Big)\eta(t)
\cr & -\sum_{n=1}^{\infty} \frac{\Lambda(n)}{n^{\frac 12}}  \Big( \chi^*(n) \eta(t-\log n)+\overline{\chi^*}(n) \eta(t+\log n)\Big) 
\cr &
+\int_0^{\infty} \frac{\e^{-(\frac 12+a(\chi)) x}}{1-\e^{-2 x}} \big( 2\eta(t) -\eta( x+t)-\eta( -x+t)\big)\d x.\end{split}
\label{equation explicit formula 1}
\end{equation}
Since $\eta\in\S_\delta$, the first and third terms are clearly $\ll_\eta \e^{-\frac t2} \log q $. For the term involving $\overline{\chi^*}(n)$, 
the triangle inequality yields that
$$
\sum_{n=1}^{\infty}\frac{\Lambda(n)}{n^{\frac 12}}\overline{\chi^*}(n) \eta(t+\log n) \ll_{\delta,\eta} \e^{-\frac t2}.$$
 The proof is completed by applying Lemma~\ref{lemma ramified primes}.
\end{proof}
\goodbreak

Throughout the paper, we will use the Riemann-von Mangoldt formula (see for instance~\cite[Corollary 14.7]{MV07})
 \begin{equation}\label{estNTchi}N(T,\chi):= \#\big\{ \rho_\chi : |\Im m(\rho_{\chi})|\leq T \big\} =\frac{T}{ \pi} \log\Big(\frac{q_\chi T}{2\pi \e}\Big)+O\big(\log(q_\chi (T+2))\big) \hspace{.5cm} (T\geq 1).
 \end{equation}

The following slightly weaker version of GRH will appear naturally in this paper.

\medskip
\noindent\textbf{Hypothesis GRH$_{\widehat \eta}$.} \textit{For all $q\geq 3$, if $\chi$ is any primitive character to the modulus $q$ and $\rho_{\chi}$ is a non-trivial zero of $L(s,\chi)$, then either $\Re e(\rho_\chi)=\tfrac 12$, or $\widehat \eta\big(\frac{\rho_\chi-\frac 12}{2\pi i}\big)=0$.}
 \medskip
 
 We observe that Hypothesis GRH$_{\widehat \eta}$ is independent of the Riemann Hypothesis for $\zeta(s)$ (since we have excluded the principal character).

As a corollary of Lemma~\ref{lemma:explicit formula}, we obtain the following uniform bounds.

\begin{cor}
\label{corollary sup bound}
Let $q\geq 3$ and $n\geq 2,$ let $\chi \bmod q$ be a non-principal Dirichlet character, and let $\delta>0$ and $\eta \in\S_\delta$. Then, under GRH$_{\widehat \eta}$, for $t\geq 0$ and for $q$ large enough we have the bounds
\begin{equation}
      |\psi_{\eta}(\e^t,\chi)| \ll_{\delta,\eta}  \log q; 
     \label{leqlogq}
\end{equation}
\begin{equation}
     M_n(\e^t,q;\eta)\ll \frac{(C_{\delta,\eta}
     \log q)^n}{\phi(q)}, 
     \label{equation pointwise bound M_n}
\end{equation}
where $C_{\delta,\eta}$ is a positive constant depending only on $\delta$ and $\eta$.
\end{cor}
\begin{proof}
The bound~\eqref{leqlogq} is obtained by applying~\eqref{condiavecdelta} to the 
sum over zeros in Lemma~\ref{lemma:explicit formula}, followed by a summation by parts using~\eqref{estNTchi}.
The bound on $M_n(\e^t,q;\eta)$ follows from inserting this bound into~\eqref{eqMn}.
\end{proof}

\section{Conductors and convergent sums over zeros}
\label{section convergent sums}

In this section, we fix $\delta>0$ and consider the set $\T_\delta$ of non-trivial measurable functions $h: \R \rightarrow \R$ having the following properties.
We require that $\xi\mapsto \xi h(\xi)$ is integrable, and that, for all $\xi \in \R$, we have the bounds 
 $$0 \leq h(\xi)\ll (1+|\xi|)^{-1}\big(\log(2+|\xi|)\big)^{-2-2\delta }.$$
Moreover, for all $t\in \R$, we have that\footnote{The integrability of $\xi\mapsto \xi h(\xi)$ implies that $\widehat h$ is differentiable (see \cite[p. 430]{KF89}).}
 $$\widehat h(t),\widehat h'(t)\ll \e^{-(\frac 12+\frac \delta 2)|t|}.  $$
Note that if $\eta \in \S_\delta$ is non-trivial, then $h_\eta:= \widehat \eta^2 
\in \T_{\delta/2}$. 
We extended $h$ to $\{s\in \mathbb C:  |\Im m(z)| \leq  \frac {1}{4\pi } \}$ by writing
\begin{equation}
     h(s):= \int_{\mathbb R} \e^{2\pi i s \xi} \widehat h(\xi) \d \xi.  
     \label{equation h Fourier}
\end{equation}
Our goal is to compute
various sums involving the quantity
\begin{equation}
b (\chi;h):=
\sum_{\rho_\chi}  {h\Big(\frac{\rho_{\chi}-\frac 12}{2\pi i} \Big) }
\label{equation definition b(chi,h)}
\end{equation}
with $h\in \T_\delta$, where $\chi$ is a non-principal character modulo $q\geq 3$, and where the sum is running over the non-trivial zeros $\rho_\chi$ of $L(s,\chi)$. 
To do so, we will first need to bound the sum
\begin{equation}
    b_2(\chi;h):=  -\sum_{n=1}^\infty \frac{\Lambda(n)}{ n^{\frac 12}}\Big(\chi^*(n)\widehat h( \log n)+\overline{\chi^*}(n)\widehat h(-\log n)\Big),
    \label{equation definition b_2}
\end{equation}
 where $\chi^*$ is the primitive character modulo $q_\chi$ inducing $\chi$.

 \begin{lemma}
\label{lemma truncation b_2}
Let  $\delta>0$ and $h\in \T_{\delta/2}$. For $q\geq 3$ and $\chi\bmod q$ non-principal, we have the estimate
\begin{equation}
    \frac 1{\phi(q)}\sum_{\substack{ \chi \bmod q \\ \chi \neq \chi_{0,q}}} b_2(\chi;h) \ll_{\delta,h}  \frac{\log q}{\phi(q)}. \end{equation}
\end{lemma}

\begin{proof}
Firstly, since  $h$ is real-valued, 
$$  b_2(\chi;h) = -2\Re e\Big(\sum_{n=1}^\infty\frac{\Lambda(n)}{ n^{\frac 12}}\chi^*(n)\widehat h( \log n)\Big).$$
We   use the same method as
~\cite[Lemma 3.2]{F15b}. 
To do so, we apply~\cite[Proposition 3.4]{FM13}, and deduce that 
\begin{align*}
 \sum_{\substack{ \chi \bmod q \\ \chi \neq \chi_{0,q}}}\Big( b_2(\chi;h) +2\Re e\Big( \sum_{n= 1}^\infty \frac{\Lambda(n)}{ n^{\frac 12}}\chi(n)\widehat h(\log n)\Big)\Big) &\ll_h \sum_{\substack{ p^\nu \parallel q  \\ p^e \equiv 1 \bmod {q/p^\nu}  \\ \nu,e\geq 1}} \frac{\log p}{p^{e (1+\frac \delta 2)}} \phi\Big ( \frac q{p^\nu }\Big)\\& \ll  \log q, 
 \end{align*}
 where the second inequality follows from an argument similar to that in the proof of~\cite[Lemma 3.2]{F15b}.
Hence, we can apply the orthogonality relations and obtain the estimate
$$
     \sum_{\substack{ \chi \bmod q \\ \chi \neq \chi_{0,q}}} b_2(\chi;h) = - 2 \Re e \Big(\Big( \phi(q)\sum_{\substack{n\geq 1 \\ n\equiv 1 \bmod q }} -\sum_{n\geq 1} \Big) \frac{\Lambda(n)}{ n^{\frac 12}}\widehat h(\log n)\Big)+O_h(\log q).
$$ 
 The proof is finished by noting that 
 $$ \sum_{\substack{n\geq 1 \\ n\equiv 1 \bmod q }}\frac{\Lambda(n)}{ n^{1+\frac \delta 2}} \ll \sum_{j=1}^{\infty} \frac{\log (1+qj)}{(1+qj)^{1+\frac \delta 2}} \ll_{\delta} \frac{\log q}{q^{1+\frac \delta 2}}. $$
\end{proof}

We are now ready to estimate the average of $b(\chi;h).$ The proof will involve the first moment of $\log q_{\chi}$, which was computed in~\cite{FM13}.

\begin{proposition}
\label{proposition first moment b(chi,h)} 
Let  $\delta>0$, $h\in \T_{\delta/2}$ and 
 $q\geq 3$.
 We have the average estimate
\begin{equation}\label{2emeest}
\frac 1{\phi(q)} \sum_{ \substack{ \chi \bmod q \\ \chi \neq \chi_{0,q}}} b(\chi;h)= \alpha(h) \log q +\beta_q(h)+O_{\delta,h}\Big(\frac{\log q }{\phi(q)}\Big),
\end{equation} 
where $\alpha(h)$ and $\beta_q(h)$ are defined in~\eqref{equation definition alpha beta}. Moreover, we have the pointwise estimate
\begin{equation} 
  b(\chi;h) = \alpha(h)\log q_\chi +O_{\delta,h}(1).
  \label{equation pointwise b(chi) }
\end{equation}
\end{proposition}

\begin{proof}

We generalize~\cite[Lemma 3.2]{F15b} by
applying the explicit formula~\cite[Theorem 12.13]{MV07} with $F(x):=2\pi \widehat h(-2\pi x). $ For $\chi \bmod q$ non-principal, this yields that
\begin{equation}\label{bsumbj}b(\chi;h)=b_1(\chi;h)+b_2(\chi;h)+b_3(\chi;h),\end{equation}
where $b_2(\chi;h)$ is defined in~\eqref{equation definition b_2} and
\begin{equation}\label{defbj}\begin{split}b_1(\chi;h)&:=\Big(\log\Big(\frac{q_\chi}{\pi}\Big)+\frac{\Gamma'(\tfrac14+\tfrac12a(\chi))}{\Gamma (\tfrac14+\tfrac12a(\chi))}\Big)\widehat h(0);\cr
b_3(\chi;h)&:=
 \int_0^\infty \frac{\e^{-(\frac 12+a(\chi)) x}}{1-\e^{-2 x}}\big(2\widehat h(0)-\widehat h( x)-\widehat h(- x)\big)\d x,\end{split}\end{equation}
with $a(\chi)$ defined as in the proof of Lemma~\ref{lemma:explicit formula}. 
We first prove~\eqref{equation pointwise b(chi) }. We clearly have that $b_1(\chi;h) = \alpha(h) \log q_\chi +O_h(1)$. As for $b_3(\chi;h)$, after recalling that $h\in \T_{\delta/2}$ implies that $\widehat h$ is differentiable,
we deduce that 
$b_3(\chi;h)\ll \sup_{t\in [0,1]}|\widehat h'(t)|+ O_h(1)\ll_h 1 $. Finally, 
$$ b_2(\chi;h) \ll_h \sum_{n\geq 1} \frac{\Lambda(n)}{n^{1+\frac \delta 2}} \ll  \delta^{-1}.$$ 
The claimed estimate follows.

We now move to~\eqref{2emeest}. The formulas  \cite[(C.15), (C.16)]{MV07}
imply that \begin{equation}b_1(\chi;h)=\alpha(h)\Big(\log\Big(\frac{q_\chi}{8\pi}\Big) -\gamma_0-  \frac \pi 2   \chi(-1)\Big).
\label{equation expression for b_1(chi)}
\end{equation}
Now, by~\cite[Proposition 3.3]{FM13} we have that
\begin{equation}
\label{equation first moment log conductor}
 \frac 1{\phi(q)} \sum_{\substack{\chi \bmod q\\ \chi\neq \chi_{0,q} }} \log q_\chi = \log q -\sum_{p \mid q } \frac{\log p}{p-1},
 \end{equation}
and hence
\begin{equation}\label{estsumb1}\frac{1}{\phi(q)} \sum_{\substack{\chi \bmod q\\ \chi\neq \chi_{0,q} }}b_1(\chi;h)
=\alpha(h)\Big(\log\Big(\frac{q }{8\pi}\Big) -\sum_{p \mid q } \frac{\log p}{p-1}-\gamma_0 \Big)  +O_h\Big(\frac{1}{\phi(q)} \Big). \end{equation}

The computation of the average of $b_2(\chi;h)$ was done in Lemma~\ref{lemma truncation b_2}, and that of $b_3(\chi;h)$ is straightforward. 

\end{proof}

\begin{remark}

In a subsequent article we will show how the techniques of this section allow to obtain an estimate for $$b_{2n}(\chi):= \sum_{\rho_\chi}\frac{1}{|\rho_\chi|^{2n}}.$$
\end{remark}

We will also need to estimate incomplete sums over zeros. To handle these, we will establish a pair correlation result. We first need to estimate the following sum over characters.

\begin{lemma}
\label{lemma orthogonality second moment}
Let $m,n,q\in \mathbb N$ with $q\geq 3$, let $\chi\bmod q$ be a fixed character, and define
$$S_q(m,n) := \sum_{\substack{\chi_1,\chi_2 \bmod q \\ \chi_1\chi_2=\chi   }} \chi_1^*(m) \chi_2^*(n).$$
Then, for any primes $p_1,p_2$ and integers $e_1,e_2\geq 1$ and $\nu_1,\nu_2\geq 0$ such that $p_1^{\nu_1},p_2^{\nu_2} \parallel q, $ we have the formula
$$  S_q(p_1^{e_1},p_2^{e_2})= \begin{cases}
 \chi(p_1^{e_1}) \phi(q)1_{p_1^{e_1}\equiv p_2^{e_2} \bmod q}   & \text{ if } (p_1p_2,q)=1 ,\\
 \chi(p_2^{e_2}) \phi(q/p_1^{\nu_1}) 1_{p_1^{e_1} \equiv p_2^{e_2} \bmod {q/ p_1^{\nu_1}}}  & \text{ if } p_1\mid q, p_2\nmid q ,\\
  \chi(p_1^{e_1}) \phi(q/p_2^{\nu_2}) 1_{p_1^{e_1} \equiv p_2^{e_2} \bmod {q/ p_2^{\nu_2}}}  & \text{ if } p_1\nmid q, p_2\mid q ,\\
1_{{\rm cond}(\chi)\mid q/p_1^{\nu_1}} \chi(p_1^{e_1}+q/p_1^{\nu_1}) \phi(q/p_1^{\nu_1}) 1_{p_1^{e_1}\equiv p_1^{e_2} \bmod {q/p_1^{\nu_1}}}  &\text{ if } p_1\mid q,  p_2=p_1,\\
\big\{\chi^{(1)}(p_2^{e_2}+q/p_2^{\nu_2}) \chi^{(2)}(p_1^{e_1})  \chi^{(3)}(p_2^{e_2}) \phi(q/(p_1^{\nu_1}p_2^{\nu_2}))\times \\ 1_{p_1^{e_1}\equiv p_2^{e_2} \bmod {q/(p_1^{\nu_1}p_2^{\nu_2})}}\big\} & \text{ if }p_1,p_2\mid q, p_1\neq p_2,
\end{cases} $$
where by definition ${\rm cond}(\chi)$ is the conductor of $\chi$, and for $p_1\neq p_2$, $\chi =\chi^{(1)}\chi^{(2)}\chi^{(3)}$ with $\chi^{(1)} \bmod {p_1^{\nu_1}}$, $\chi^{(2)} \bmod {p_2^{\nu_2}}$ and $\chi^{(3)} \bmod {q/(p_1^{\nu_1}p_2^{\nu_2})}$.
\end{lemma}

\begin{proof}
The easiest case is when $ (p_1p_2,q)=1,$ in which we have that
\begin{equation*}
S_q(p_1^{e_1},p_2^{e_2})
= \sum_{\substack{\chi_1 \bmod q }}
 \chi_1(p_1^{e_1}) ( \chi \overline{\chi_1})(p_2^{e_2}) =  \chi(p_2^{e_2})\phi(q)1_{p_1^{e_1}\equiv p_2^{e_2} \bmod q}.
\end{equation*}
If $p_1\mid q$ but $p_2\nmid q$, then
$$S_q(p_1^{e_1},p_2^{e_2})=\sum_{\substack{\chi_1,\chi_2 \bmod q  \\ \chi_1\chi_2=\chi  }} \chi_1^*(p_1^{e_1}) \chi_2(p_2^{e_2}) = \sum_{\substack{\chi_1 \bmod q    }} \chi_1^*(p_1^{e_1}) ( \chi \overline{\chi_1})(p_2^{e_2}),$$
which by~\cite[Proposition 3.4]{FM13} is equal to
$ \chi(p_2^{e_2}) \phi(q/p_1^{\nu_1}) 1_{p_1^{e_1} \equiv p_2^{e_2} \bmod {q/ p_1^{\nu_1}}}.$

Let us now assume that $p_1, p_2 \mid q$. First observe that
$$ S_q(p_1^{e_1},p_2^{e_2})
=\sum_{\substack{\chi_1 \bmod {q/p_1^{\nu_1}}\\\chi_2 \bmod {q/p_2^{\nu_2}} \\ \forall x: (x,q)=1, \chi_1\chi_2(x)=\chi(x)   }} \chi_1(p_1^{e_1}) \chi_2(p_2^{e_2}).$$
We claim that given $\chi \bmod q$ and $\chi_1 \bmod {q/p_1^{\nu_1}}$, the condition $\forall x: (x,q)=1, \chi_1\chi_2(x)=\chi(x)$ uniquely determines $\chi_2 \bmod {q/p_2^{\nu_2}}.$ Indeed, we have that for all $f\geq 1$, $\chi_2(p_2^f) =\chi_2(p_2^f+ q/p_2^{\nu_2}), $ and $(p_2^f+ q/p_2^{\nu_2},q)=1.$ Moreover, if  ${\rm cond}  (\chi \overline{\chi_1})\mid  q/p_2^{\nu_2}$, then such a $\chi_2$ always exists. As a result, $S_q(p_1^{e_1},p_2^{e_2})$ is equal to
$$\sum_{\substack{\chi_1 \bmod {q/p_1^{\nu_1}} \\ {\rm cond}(\chi \overline{\chi_1})\mid  q/p_2^{\nu_2}}} \chi_1(p_1^{e_1}) (\chi \overline{\chi_1})(p_2^{e_2}+q/p_2^{\nu_2})= \chi(p_2^{e_2}+q/p_2^{\nu_2})\!\!\sum_{\substack{\chi_1 \bmod {q/p_1^{\nu_1}}  \\ {\rm cond}(\chi \overline{\chi_1})\mid  q/p_2^{\nu_2}}}\!\!\!\! \chi_1(p_1^{e_1})  \overline{\chi_1}(p_2^{e_2}+q/p_2^{\nu_2}) .$$

Let us now assume that $p_1\neq p_2$, in which case $\chi\bmod q $ can be decomposed as $\chi=\chi^{(1)}\chi^{(2)}\chi^{(3)}$, with $\chi^{(1)} \bmod {p_1^{\nu_1}}$, $\chi^{(2)} \bmod {p_2^{\nu_2}}$ and $\chi^{(3)} \bmod {q/(p_1^{\nu_1}p_2^{\nu_2})}$. Similarly, we can decompose any $\chi_1\bmod {q/p_1^{\nu_1}}$ as $\chi_1=\chi_1^{(2)}\chi_1^{(3)}$, with $\chi_1^{(2)} \bmod {p_2^{\nu_2}}$ and $\chi_1^{(3)} \bmod {q/(p_1^{\nu_1}p_2^{\nu_2})}$. The condition ${\rm cond}(\chi \overline{\chi_1})\mid  q/p_2^{\nu_2}$ is equivalent to saying that $\chi^{(2)}\overline{\chi_1^{(2)}} = \chi_{0,p_2^{\nu_2}} $. In other words, 
 $$S_q(p_1^{e_1},p_2^{e_2})=\chi(p_2^{e_2}+q/p_2^{\nu_2}) \chi^{(2)}(p_1^{e_1})  \overline{\chi^{(2)}}(p_2^{e_2}+q/p_2^{\nu_2})\!\!\!\sum_{\substack{\chi_1 \bmod {q/(p_1^{\nu_1}p_2^{\nu_2})}  }} \!\!\!\!\chi_1(p_1^{e_1})  \overline{\chi_1}(p_2^{e_2}+q/p_2^{\nu_2}) ,$$
 which by the orthogonality relations is equal to
$$ \chi(p_2^{e_2}+q/p_2^{\nu_2}) \chi^{(2)}(p_1^{e_1})  \overline{\chi^{(2)}}(p_2^{e_2}+q/p_2^{\nu_2}) \phi(q/(p_1^{\nu_1}p_2^{\nu_2})) 1_{p_1^{e_1}\equiv p_2^{e_2}+q/p_2^{\nu_2} \bmod {q/(p_1^{\nu_1}p_2^{\nu_2})}}. $$

The final case to consider is $p_1=p_2 \mid q$, in which 
the condition ${\rm cond}(\chi \overline{\chi_1})\mid  q/p_1^{\nu_1}$ is equivalent to ${\rm cond}(\chi)\mid q/p_1^{\nu_1}$. The sum $S_q(p_1^{e_1},p_2^{e_2})$ is then equal to
\begin{multline*}
    1_{{\rm cond}(\chi)\mid q/p_1^{\nu_1}}\chi(p_1^{e_2}+q/p_1^{\nu_1}) \sum_{\substack{\chi_1 \bmod {q/p_1^{\nu_1}}  }}\!\!\!\! \chi_1(p_1^{e_1})  \overline{\chi_1}(p_1^{e_2}) \\ =1_{{\rm cond}(\chi)\mid q/p_1^{\nu_1}} \chi(p_1^{e_2}+q/p_1^{\nu_1}) \phi(q/p_1^{\nu_1}) 1_{p_1^{e_1}\equiv p_1^{e_2} \bmod {q/p_1^{\nu_1}}}, 
\end{multline*}
 by the orthogonality relations.
\end{proof}

Here is the pair correlation result we will need. We will assume GRH$_{h}$ for brevity, however it is very possible that this assumption can be removed.
\begin{proposition}
\label{proposition pair correlation}
Let $q\geq 3$, and fix $\chi \bmod q$ and $z\in\R$. Let $h\in\T_{\delta/2}$, and assume GRH$_{h}$. Let $\Phi \in \mathcal L^1(\mathbb R)$ be even and supported in $[-1,1]$.
Then in the range $  1\leq  L\leq  \log q-\log_2 q$ we have the bound
$$  P(h;\chi,z):= \!\! \sum_{\substack{\chi_1,\chi_2 \neq \chi_{0,q} \\ \chi_1\chi_2 = \chi}} \sum_{\substack{\gamma_1,\gamma_2 }}   h\Big(\frac{\gamma_{1}}{2\pi } \Big) h\Big(\frac{\gamma_{2}}{2\pi } \Big)\widehat\Phi\Big(\frac {L}{2\pi} (z-\gamma_1-\gamma_2)\Big)  \ll_{\delta,h} \phi(q)  (\log q)\Big( 1+\frac{\log q}L\Big),$$
where $\rho_j=\frac 12+i\gamma_j$ runs through the zeros of $L(s,\chi_j)$. The implied constant is independent of $\chi$ and $z$.
\end{proposition}

\begin{proof}
We apply~\eqref{equation explicit formula 1} with $\eta:=\widehat h$; we obtain the identity 
$$ b(\chi;h, t  ):=\sum_{\gamma_{\chi}} h\Big( \frac{\rho_\chi-\frac 12}{2\pi i} \Big)\e^{(\rho_\chi-\frac 12)  t  } =b_1(\chi;h, t  )+b_2(\chi;h, t  )+b_3(\chi;h, t  )+b_4(\chi;h, t  ),
$$ 
 where 
\begin{align*}
     b_1(\chi;h, t  ):=& \Big( \log \frac{q_{\chi}}{\pi} +\frac{\Gamma'}{\Gamma}\Big(\frac 14+\frac{a(\chi)}2 \Big)\Big) \widehat h( t  );\cr 
  b_2(\chi;h, t  ):=&-\sum_{n\geq 1} \frac{\Lambda(n)}{n^{\frac 12}} \chi^*(n) \widehat h( t  -\log n); 
  \cr b_3(\chi;h, t  ):=&\int_0^\infty \frac{\e^{-(\frac 12+a(\chi))y}}{1-\e^{-2y}} \Big(2\widehat h( t  ) - \widehat h( t  +y)-\widehat h( t  -y) \big) \d y;\cr b_4(\chi;h, t  ):=&-\sum_{n\geq 1} \frac{\Lambda(n)}{n^{\frac 12}} \overline{\chi^*}(n) \widehat h( t  +\log n). 
\end{align*}
We have the bounds $b_1(\chi;h, t   ) \ll |\widehat h( t   )|\log q \leq  \widehat h(0)\log q$; $b_3(\chi;h, t   ),b_4(\chi;h, t   ) \ll_h 1 ;$ $b(\chi;h, t   )\ll_h\log q$, and as a result, $b_2(\chi;h, t   )\ll_{h} \log q$ as well.
We deduce that
\begin{align*}
P(h;\chi,z)=&\int_{\mathbb R} \Phi(\xi) \e^{- i L z \xi} \sum_{\substack{\chi_1,\chi_2 \neq \chi_{0,q} \\ \chi_1\chi_2 = \chi}} \sum_{\substack{\gamma_1,\gamma_2}}   h\Big(\frac{\gamma_{1}}{2\pi } \Big) h\Big(\frac{\gamma_{2}}{2\pi } \Big) \e^{ i L \gamma_1 \xi}\e^{ i L \gamma_2 \xi}\d \xi\\
 = &\int_{\mathbb R} \Phi(\xi) \e^{- i L z \xi} \sum_{\substack{\chi_1,\chi_2 \neq \chi_{0,q} \\ \chi_1\chi_2 = \chi}} b_2(\chi_2,h,L\xi)\sum_{\substack{\gamma_1}}   h\Big(\frac{\gamma_{1}}{2\pi } \Big) \e^{ i L \gamma_1 \xi}\d \xi\\
&+O_{\delta,h}\Big(\phi(q) (\log q)^2 \int_{\mathbb R} |\widehat h(L\xi)| \d \xi + \phi(q) \log q\Big).
\end{align*}
Therefore, we have that
$$ P(h;\chi,z)= \int_{\mathbb R} \Phi(\xi) \e^{- i L z \xi}\!\!\!\! \sum_{\substack{\chi_1,\chi_2 \neq \chi_{0,q} \\ \chi_1\chi_2 = \chi}}\!\!\!\!\! b_2(\chi_1,h,L\xi) b_2(\chi_2,h,L\xi)\d \xi\\
+O_{\delta,h}\Big(\phi(q) \frac{(\log q)^2}L +\phi(q)\log q \Big).$$
The main term is equal to
\begin{equation}
P_M(h,z,L):= \sum_{\substack{m,n\geq 1  }} \frac{\Lambda(m)\Lambda(n)}{(mn)^{\frac 12}} I_{m,n}(h,z,L)\sum_{\substack{\chi_1,\chi_2 \bmod q \\ \chi_1,\chi_2 \neq \chi_{0,q} \\ \chi_1\chi_2 = \chi}}  \chi_1^*(m)\chi_2^*(n), 
 \label{equation before applying lemma orthogonality}
\end{equation}
where 
$$ I_{m,n}(h,z,L):=\int_{\mathbb R} \Phi(\xi) \e^{- i L z \xi}  
\widehat h(\log m -L\xi)\widehat h(\log n-L\xi)\d \xi. $$
We note that $ I_{m,n}(h,z,L)\ll L^{-1}$, and moreover for $m\leq n$,
\begin{equation}\label{majImn}
I_{m,n}(h,z,L)\ll \begin{cases}
(mn)^{-\frac 12-\delta} L^{-1}\e^{L(1+2\delta)} & \text{ if }   \e^L \leq m\leq n ,\\
\big( \frac mn\big)^{\frac 12+\delta} \big(  \frac 1L + 1-\frac{\log m}L\big) &\text{ if } m\leq \e^L \leq  n,
\\ \big( \frac mn\big)^{\frac 12+\delta} \big(  \frac 1L + \frac{\log (n/m)}L\big) &\text{ if } m\leq n\leq \e^L.
\end{cases}
\end{equation}
In particular, if  $m\leq n  $ and $m\leq \e^L$ we have the bound $I_{m,n}(h,z,L)\ll ( \frac mn)^{\frac 12+\delta}$.

We now apply Lemma~\ref{lemma orthogonality second moment}. The contribution to $P_M(h;\chi,z)$ of those $n,m$ for which $(mn,q)=1$ is given by
\begin{align*}
 P_{(mn,q)=1}:&=  \sum_{\substack{m,n\geq 1  \\ (mn,q)=1}} \frac{\Lambda(m)\Lambda(n)}{(mn)^{\frac 12}}  I_{m,n}(h,z,L) \big( \chi(m)\phi(q) 1_{n\equiv m\bmod q}-\chi(m)-\chi(n) + 1\big)
 \\ &\ll J+\phi(q) L  + \e^{L},
  \end{align*}
where the contribution of the non-diagonal terms in the part of the sum involving the factor $\chi(m)\phi(q) 1_{n\equiv m\bmod q}$ is bounded by a constant times
$$ J := \phi(q)  \sum_{\substack{m,k\geq 1  }} \frac{\Lambda(m)\Lambda(m+kq)}{(m(m+kq))^{\frac 12}} I_{m,m+kq}(h,z,L)\ll \phi(q)\frac{\e^L (\log q+L)}{q}. $$
In other words, we have shown that $P_{(mn,q)=1}$ is an admissible error term. 

For the remaining terms in $P_M(h,z,L)$ we first note that we may drop the conditions $\chi_1,\chi_2\neq \chi_{0,q}$ in~\eqref{equation before applying lemma orthogonality} at the cost of the error term $O(\e^L)$; we will denote the resulting sum by $Q(h;\chi,z)$. By Lemma~\ref{lemma orthogonality second moment}, the contribution to $Q(h;\chi,z)$ of those $n,m$ for which either $(m,q)>1$ and $(n,q)=1$, or $ (n,q)>1$ and $(m,q)=1$, is then 
$$ \ll  \sum_{\substack{n\geq 1  \\ (n,q)=1}}   \sum_{\substack{p^{\nu}\parallel q \\ \nu,e\geq 1 }}\frac{\Lambda(n)\log p}{(np^e)^{\frac 12}} I_{n,p^e}(h,z,L)   \phi(q/p^{\nu}) 
 \ll \phi(q)\log_2q.$$
 As
 for the contribution of those $m,n$ for which $m=p_1^{e_1}, n=p_2^{e_2}$ with $p_1\neq p_2$ and $p_1,p_2\mid q$, it is 
  $$ \ll \sum_{\substack{p_1^{\nu_1}\parallel q \\ \nu_1,e_1\geq 1 }} \sum_{\substack{p_2^{\nu_2}\parallel q \\ \nu_2,e_2\geq 1  \\ p_2^{e_2} \geq p_1^{e_1}}}\frac{\log p_1\log p_2}{(p_1^{e_1}p_2^{e_2})^{\frac 12}} I_{p_1^{e_1},p_2^{e_2}}(h,z,L)  \phi(q/(p_1^{\nu_1}p_2^{\nu_2})) 
 \ll_\eps \phi(q) \log_2 q .$$
 To show this last bound, we used the inequalities
$$\sum_{\substack{\nu_1\geq 1\\ \nu_2\geq 1  }}
\phi(q/(p_1^{\nu_1}p_2^{\nu_2})) \leq \sum_{\substack{\nu_1\geq 1\\ \nu_2\geq 1  }} 
\frac{\phi(q)}{\phi(p_1^{\nu_1}p_2^{\nu_2})}
\ll 
\frac{\phi(q)}{p_1p_2}  $$   and applied the bounds \eqref{majImn}.

Finally, the contribution of those $n,m$ for which $m=p^{e_1}, n=p^{e_2}$ with $p^\nu\mid q$ and $\nu\geq 1$ is 
\begin{align*}
  \ll \sum_{\substack{p^{\nu}\parallel q \\ \nu,e_1,e_2\geq 1 }} \frac{(\log p)^2}{p^{\frac {e_1+e_2}2}} I_{p^{e_1},p^{e_2}}(h,z,L)  \phi(q/p^{\nu}) 
 &\ll\frac 1L \sum_{\substack{p^\nu\parallel q \\ \nu\geq 1 }}\frac{(\log p)^2}{p}\phi(q/p^{\nu}) 
 \ll \phi(q).
\end{align*} 
 The proof is finished.
\end{proof}

The next step is to estimate the variance of $\log q_\chi$. This was done in \cite[Lemma 3.3]{F15b}, which we refine and rectify here. 
 
\begin{lemma}
\label{lemmavarlogq}
For $q\geq 3$, we have the estimate 
$$\frac{1}{\phi(q)} \sum_{\substack{\chi \bmod q\\ \chi\neq \chi_{0,q} }} \Big(\log q_{\chi} -\log q +\sum_{p \mid q } \frac{\log p}{p-1} \Big)^2   =\sum_{p\mid q }  \frac{(\log p)^2}{p} +O(1) \ll (\log_2q)^2 .  $$
Moreover, for $h\in \T_{\delta/2}$ we have that
$$ \frac 1{\phi(q)}  \sum_{\substack{ \chi \bmod q \\ \chi \neq \chi_{0,q}}} \big(b(\chi;h) - \alpha(h) \log q  \big)^2 \ll  (\alpha(h)\log_2 q)^2 +O_{\delta,h}(1).$$
\end{lemma}

\begin{proof} 
Performing a computation analogous\footnote{The first display on page 4434 should have a factor $2$ in front of the sum in parentheses on the right hand side.} to that in \cite[Lemma 3.3]{F15b}, 
we see that
\begin{align*}\frac{1}{\phi(q)}& \sum_{\substack{\chi \bmod q\\ \chi\neq \chi_{0,q} }} (\log q_\chi-\log q)^2\cr&=\sum_{p^r\parallel q} \frac{p(\log p)^2}{(p-1)^2}\Big(1+\frac{1}{p}- \frac{2}{p^r}\Big)
+2\sum_{\substack{p_1,p_2\mid  q\\ p_1<p_2}}\frac{( \log p_1)(\log p_2)}{(p_1-1) (p_2-1)}+O\Big(\frac{(\log q)^2}{\phi(q)}\Big)
\cr&=\Big(\sum_{p \mid q } \frac{\log p}{p-1}\Big)^2+\sum_{p^r\parallel q} \frac{p(\log p)^2}{(p-1)^2}\Big(1- \frac{2}{p^r}\Big)+O\Big(\frac{(\log q)^2}{\phi(q)}\Big).
\end{align*} 
We conclude that 
\begin{align*}\frac{1}{\phi(q)} \sum_{\substack{\chi \bmod q\\ \chi\neq \chi_{0,q} }} \Big(\log q_\chi-\log q+\sum_{p \mid q } \frac{\log p}{p-1}\Big)^2&=\sum_{p^r\parallel q} \frac{p(\log p)^2}{(p-1)^2}\Big(1
 -\frac{2}{p^r}\Big)+O\Big(\frac{(\log q)^2}{\phi(q)}\Big) \cr& =\sum_{p\mid q }  \frac{(\log p)^2}{p } +O(1) \ll (\log_2q)^2,
\end{align*}  
and the first claimed estimate follows. For the second, we combine the last bound with~\eqref{equation pointwise b(chi) }.
\end{proof}

We now establish the final lemma of this section, which will be useful in our computation of higher moments.

\begin{lemma}
\label{lemme moyenne multiple b(chi)} 
Let $\delta>0$, $h\in \T_{\delta/2}$, and
let $k\geq 2$. If $q$ is large enough in terms of $\delta$ and $h$ and $\chi$ is a character modulo $q$, then in the range $k\leq \log q/\log_2q $ we have the estimate 
\begin{equation}\label{sumb1bk}\sum_{\substack{\chi_1,\ldots,\chi_k  \bmod q\\ \chi_1,\ldots,\chi_k \neq \chi_{0,q} \\ \chi_1 \cdots \chi_k = \chi }} b(\chi_1;h) \cdots b(\chi_k;h)= \phi(q)^{k-1}(\alpha(h)\log q)^{k}\Big( 1+O \Big(k \frac {\log_2 q  }{ \log q } \Big) \Big).  \end{equation}
Moreover, if we assume GRH$_h$, then uniformly for $z\in \R$ we have the bound
\begin{equation}\sum_{\substack{\chi_1,\ldots,\chi_k  \bmod q\\ \chi_1,\ldots,\chi_k \neq \chi_{0,q} \\ \chi_1 \cdots \chi_k = \chi }}\sum_{\substack{\rho_{\chi_1}, \dots, \rho_{\chi_k} \\ \gamma_{\chi_1} + \dots + \gamma_{\chi_k} = z}}  {h\Big(\frac{\gamma_{\chi_1}}{2\pi } \Big) } \cdots {h\Big(\frac{\gamma_{\chi_k}}{2\pi } \Big) }  \ll_{\delta,h}  \phi(q)^{k-1}(\alpha(h)\log q)^{k-1}, 
\label{equation average sum over zero one removed}
\end{equation}
where we used the shorthand $\gamma_{\chi_j} := (\rho_{\chi_j}-\frac 12)/i$.
\end{lemma}

\begin{proof} 

We first establish~\eqref{equation average sum over zero one removed}. The left hand side of this equation is
\begin{multline*}
  \leq\sum_{\substack{\chi_3,\ldots,\chi_k  \bmod q\\ \chi_3,\ldots,\chi_k \neq \chi_{0,q}  }}\sum_{\substack{\rho_{\chi_3}, \dots, \rho_{\chi_k} }}   h\Big(\frac{\gamma_{3}}{2\pi }\Big)\cdots  h\Big(\frac{\gamma_{k}}{2\pi }\Big)\times  \\  \sup_{{ z'}\in \mathbb R} \sup_{\chi \bmod q} \sum_{\substack{\chi_1,\chi_2 \neq \chi_{0,q} \\ \chi_1\chi_2 = \chi}} \sum_{\substack{\gamma_1,\gamma_2 }}   h\Big(\frac{\gamma_{1}}{2\pi } \Big) h\Big(\frac{\gamma_{2}}{2\pi } \Big) \widehat\Phi_0\Big(\frac {\log q}{3\pi} ({ z'}-\gamma_1-\gamma_2)\Big)\Big),    
\end{multline*}
where $\Phi_0(x):=\max(0,1-|x|)$.
 Here, we replaced the sum over $\chi=(\chi_3\cdots\chi_k)^{-1}$ with $z'=z-\gamma_3-\ldots-\gamma_k$ by the supremum over $z'$ and $\chi $, and deleted the condition $\gamma_1+\gamma_2=z'$ by adding the factor involving $\widehat\Phi_0$. 

The claimed bound follows from Proposition~\ref{proposition pair correlation} and~\eqref{equation pointwise b(chi) }.

We now move to~\eqref{sumb1bk}. By expanding the product $ a_1 \cdots a_k=(a_1-b_1+b_1)(a_2-b_2+b_2)\cdots (a_k-b_k+b_k)   $ we deduce the identity
$$ a_1 \cdots a_k - b_1\cdots b_k = \sum_{\substack{(\eps_1,\dots ,\eps_k)\subset \{0,1 \}^k \\ (\eps_1,\dots ,\eps_k) \neq (0,\dots,0)}} c(\eps_1)\cdots c(\eps_k), $$
 where 
 $$ c(\eps_j):= \begin{cases} a_j-b_j &\text{ if } \eps_j = 1 \\
 b_j &\text{ otherwise}.
 \end{cases} $$
We observe that for $\ell\geq 2$,
\begin{equation}\label{def+estCkq}
    C_{\ell}(q;\chi):=\big|\big\{(\chi_1,\ldots,\chi_{\ell}) \,:\, \chi_1 \cdots \chi_{\ell} = \chi, \, \chi_j\neq \chi_{0,q}\big\}\big|=(\phi(q)-1)^{\ell-1}+O(\phi(q)^{\ell-2}),
\end{equation}
and hence by setting $a_j:=b(\chi_j;h)$ and $b_j:=\alpha(h)\log q$ it follows that 
\begin{multline*}
   \sum_{\substack{\chi_1,\ldots,\chi_k  \bmod q\\ \chi_1,\ldots,\chi_k \neq \chi_{0,q} \\ \chi_1 \cdots \chi_k = \chi }} b(\chi_1;h) \cdots b(\chi_k;h)- C_k(q;\chi)(\alpha(h)\log q)^{k}\\ = \sum_{\varnothing \neq J\subset \{1,\dots,k\} }  (\alpha(h)\log q)^{k-|J|} \sum_{\substack{\chi_1,\ldots,\chi_k  \bmod q\\ \chi_1,\ldots,\chi_k \neq \chi_{0,q} \\ \chi_1 \cdots \chi_k = \chi }} \prod_{j\in J}(b(\chi_j;h)-\alpha(h)\log q).
    \end{multline*}
Now,
noting that $C_1(q;\chi) \in \{0,1 \}$, the triangle and Cauchy-Schwarz inequalities and the esimtate~\eqref{def+estCkq} imply that for any fixed $J\subsetneq \{ 1,\dots,k\}$,
\begin{align*}
 \Big|\sum_{\substack{\chi_1,\ldots,\chi_k  \bmod q\\ \chi_1,\ldots,\chi_k \neq \chi_{0,q} \\ \chi_1 \cdots \chi_k = \chi }} \prod_{j\in J}(b(\chi_j;h)-\alpha(h)\log q) \Big| &\ll
 (\phi(q)-1)^{k-|J|-1}
\prod_{j\in J}\Big( \sum_{ \chi_j\neq \chi_{0,q}  }\!\!\! |b(\chi_j;h)-\alpha(h)\log q| \Big)   \\
 & \leq \phi(q)^{k-\frac {|J|}2-1} \Big(\sum_{ \substack{\chi \neq \chi_{0,q}}} |b(\chi;h)-\alpha(h)\log q|^2\Big)^{\frac {|J|}2} \\
 & \ll \phi(q)^{k-1} \big( \alpha(h)\log_2 q+c_{\delta,h}\big)^{|J|},
 \end{align*}
 by Lemma~\ref{lemmavarlogq}, where 
 $c_{\delta,h}>0$ depends on $\delta$ and $h$. When $J= \{ 1,\dots,k\}$, this is $\ll \phi(q)^{k-1} ( \alpha(h)\log_2 q+c_{\delta,h})^{k-1} \cdot  (\alpha(h)\log q+c_{\delta,h})$, by the estimate~\eqref{equation pointwise b(chi) } and the trivial bound $\log q_{\chi} \leq \log q$.
 Consequently, we have the bound 
\begin{align*}
  \sum_{\substack{\chi_1,\ldots,\chi_k  \bmod q\\ \chi_1,\ldots,\chi_k \neq \chi_{0,q} \\ \chi_1 \cdots \chi_k = \chi }} & b(\chi_1;h) \cdots b(\chi_k;h)- C_k(q)(\alpha(h)\log q)^{k} \\ & \ll \phi(q)^{k-1}\sum_{j=1}^{k-1} \binom kj \big(\alpha(h)\log q\big)^{k-j}  (\alpha(h)\log_2 q+c_{\delta,h})^{j} \\
  &\hspace{3cm}\qquad+ \phi(q)^{k-1} \big( \alpha(h)\log_2 q+c_{\delta,h}\big)^{k-1}  (\alpha(h)\log q+c_{\delta,h}) \\
  & = \phi(q)^{k-1} (\alpha(h) \log q)^k \Big(\Big(1+\frac{\log_2 q+c_{\delta,h} \alpha(h)^{-1}}{\log q} \Big)^k-1+O \Big( \Big( 2\frac{\log_2 q}{\log q} \Big)^{k-1}\Big)\Big)
  \\ 
& 
\ll  k \frac{\log_2 q+c_{\delta,h} \alpha(h)^{-1}}{\log q}  \phi(q)^{k-1} (\alpha(h) \log q)^k \Big(1+\frac{\log_2 q+c_{\delta,h} \alpha(h)^{-1}}{\log q} \Big)^{k-1}.
  \end{align*}
 The claim follows.
\end{proof}

  \section{The probabilistic model}

The two goals of this section is to show that under GRH$_{\widehat \eta}$ (see the beginning of Section~\ref{section explicit formula} for the definition of this hypothesis), the function $M_n(\e^t,q;\eta)$ has a limiting distribution, and moreover we can compute its moments. We begin with the following lemma.

\label{section limiting distributions}
\begin{lemma}
For $n,m\in \mathbb N$, 
let $\bm E=( u_{\mu,j},v_{\mu,j})_{\substack{ 1\leq \mu \leq m \\ 1\leq j\leq n}}$ be an array of functions $u_{\mu,j},v_{\mu,j}:\R_{\geq 0} \rightarrow \R $ such that $\bm E$ has a limiting distribution in $\mathbb R^{2mn}$. Then, the function 
$$ F(t):= \Re e\Big(\sum_{\mu=1}^m \prod_{j=1}^n \big(u_{\mu,j}(t)+iv_{\mu,j}(t)\big)\Big) $$
has a limiting distribution in $\mathbb R$.
\label{lemme distribution limite de somme et produit}
\end{lemma}

\begin{proof}
This follows from the definition of limiting distribution. Defining
$G: \R^{2mn} \rightarrow \R$ by $G((x_{\mu,j},y_{\mu,j})_{\mu,j}):= \Re e(\sum_{\mu=1}^m\prod_{j=1}^n (x_{\mu,j}+iy_{\mu,j}))$, one can check that the measure defined by
$$ \mu_F(f):= \mu_{\bm E}( f\circ G ) $$
is a limiting distribution for $F$.  In other words, $\mu_F$ is the pushforward measure $G_* \mu_{\bm E}$. 
\end{proof}
We can now show the existence of the limiting distribution. 
\begin{lemma}
Let  $\delta>0$, $\eta\in \S_{\delta}$, and assume GRH$_{\widehat \eta}$. Then for any $n\geq 2$ and $q\geq 3$, $M_{n}(\e^t,q;\eta)$ has a limiting distribution.  
\end{lemma} 

\begin{proof}
In the expression~\eqref{eqMn}, we order the set of vectors of characters  $(\chi_1,\dots, \chi_n) $ modulo $q$ such that $\chi_j\neq \chi_{0,q}$, $\chi_1\dots \chi_n  = \chi_{0,q}$, as $\{(\psi_{\mu,1},\dots, \psi_{\mu,n}) : 1\leq \mu\leq C_n(q)
\}$,
where $C_n(q)$ is defined in \eqref{def+estCkq}.
It follows that
$$ 
M_n(\e^t,q;\eta) 
= \sum_{\mu=1}^m \prod_{j=1}^n (u_{\mu,j} (t) +iv_{\mu,j} (t) ),
$$
where $m=C_n(q)$, $u_{\mu,j} (t):=\Re e(\psi_{\eta}(\e^t,\psi_{\mu,j}))$, $v_{\mu,j} (t):=\Im m(\psi_{\eta}(\e^t,\psi_{\mu,j}))$. By Lemma~\ref{lemma:explicit formula}, we have that 
$$ u_{\mu,j}(t) = -\Re e\Big(\sumb_{\gamma_{\psi_{\mu,j}} } \e^{ i t \gamma_{\psi_{\mu,j}} }  \widehat \eta\Big(\frac{\gamma_{\psi_{\mu,j}}}{2\pi} \Big)\Big) +O(\e^{-\frac t2} \log q);$$
$$v_{\mu,j}(t) = -\Re e\Big(i\sumb_{\gamma_{\psi_{\mu,j}} } \e^{ i t \gamma_{\psi_{\mu,j}} }  \widehat \eta\Big(\frac{\gamma_{\psi_{\mu,j}}}{2\pi} \Big)\Big) +O(\e^{-\frac t2} \log q),  $$ 
where $\rho_{\psi_{\mu,j}}= \frac 12+i \gamma_{\psi_{\mu,j}}$ is running over the non-trivial zeros $\rho_{\psi_{\mu,j}}$ of $L(s,\psi_{\mu,j})$, and where the $\flat$ over the sum indicates that we are summing over zeros $\rho_{\psi_{\mu,j}}$ for which $\widehat \eta ( (\rho_{\psi_{\mu,j}}-\frac 12)/(2\pi i) ) \neq 0$. In particular, since we are assuming GRH$_{\widehat \eta}$, we are only summing over zeros on the critical line; in other words, $\gamma_{\psi_{\mu,j}}\in\mathbb R$.
We now apply~\cite[Theorem 1.4]{ANS14} (see also~\cite{D20}), whose conditions are satisfied thanks to the fact that $\eta\in \S_\delta$ and the Riemann-von Mangoldt formula~\eqref{estNTchi}. It follows that $\bm E(t):=( u_{\mu,j}(t),v_{\mu,j}(t))_{\substack{ 1\leq \mu \leq m \\ 1\leq j\leq n}}$ admits a limiting distribution on $\mathbb R^{2mn}$, and the same is true for $M_n(\e^t,q;\eta)$ by Lemma~\ref{lemme distribution limite de somme et produit}.
\end{proof}  
Our next goal is to give an expression for the moments of $H_n(q;\eta)$.

\begin{lemma}
\label{lemma moments of limiting distribution as limit of V(T)}
  Let $\delta>0$ and $\eta \in \S_\delta$, and assume $GRH_{\widehat\eta}$. Let $s\in \mathbb N$, $n\geq 2,$ $q\geq 3$, and let $H_n(q;\eta)$ be the random variable associated to the limiting distribution of $M_n(\e^t,q;\eta)$. 
Then we have the formulas
$$ \E[ H_n(q;\eta)^s] = \lim_{T\rightarrow \infty}  \frac {1}T \int_{0}^{T} M_n(\e^t,q;\eta)^s\d t; $$
$$ \E\big[ (H_n(q;\eta)-\mathbb E[H_n(q;\eta)])^s\big] = \lim_{T\rightarrow \infty}  \frac {1}T \int_{0}^{T} \big(M_n(\e^t,q;\eta)-\mathbb E[H_n(q;\eta)]\big)^s\d t. $$ 
\label{lemma formula for moments of limiting distribution}
\end{lemma}

\begin{proof}
The proof is direct since $M(\e^t,q;\eta)$ is bounded by Corollary~\ref{corollary sup bound}, and thus $H_n(q;\eta)$ is compactly supported. Hence, with $B(q;\eta):= \sup_{t\geq 0}|M(\e^t,q;\eta)| $, one can take 
$$f(x)= \begin{cases}
x^s 
&\text{ if } |x|\leq B(q;\eta); \\
x^s (B(q;\eta)+1-|x|) &\text{ if } B(q;\eta)<|x|\leq B(q;\eta) +1 ;\\ 
0  &\text{ if } |x|> B(q;\eta) +1
\end{cases}$$ in~\eqref{equation definition limiting distribution}, and deduce that 
$$ \lim_{T\rightarrow \infty} \frac 1T \int_{ t\leq T} M(\e^t,q;\eta)^s \d t = \int_{|x|\leq B(q;\eta)} x^s \d \mu_{H_n(q;\eta)}(x)=  \E[H_n(q;\eta)^s]. $$
The proof is similar for centered moments.
\end{proof}

  In order to give an explicit description of the moments of $H_n(q;\eta)$, we apply the explicit formula in Lemma~\ref{lemma:explicit formula}. The following notation will simplify the resulting expression.
  For $s,n\in \mathbb N$, we introduce the set of arrays of characters
  \begin{equation}\label{defXsn}
      X_{s,n} :=\left\{  \bchi=(\chi_{\mu,j})_{\substack{1\leq \mu\leq s\\ 1\leq j\leq n}} :\,  { 
   \chi_{\mu,j}  \neq \chi_{0,q} \, (1\leq j\leq n),\,\, \atop  \chi_{\mu,1}\cdots \chi_{\mu,n} =\chi_{0,q}\,\, ( 1\leq \mu\leq s)}
      \right\}
  \end{equation}
For $\bchi\in X_{s,n} $, we define the set of associated arrays of non-trivial zeros
\begin{equation}
\Gamma(\bchi):=\Big\{  \bfgamma=( \gamma_{\chi_{\mu,j}})_{\substack{1\leq \mu\leq s\\ 1\leq j\leq n}}  \,:\quad L(\tfrac 12+i\gamma_{\chi_{\mu,j}},\chi_{\mu,j})=0, |\Im m(\gamma_{\chi_{\mu,j}})|< \tfrac 12\Big\},
\label{equation definition Gamma(chi)}
\end{equation}
as well as the subset 
$$\Gamma_0(\bchi):=\Big\{ \bfgamma \in \Gamma(\bchi)\,:\quad 
\sum_{\substack{1\leq \mu\leq s\\ 1\leq j\leq n}} \gamma_{\chi_{\mu,j}}=0\Big\}.$$
For $\bm \gamma \in \Gamma(\bchi)$, we will use the notations
\begin{equation}\label{defetabfgamma}
    \widehat \eta(\bm \gamma):= \prod_{\substack{1\leq \mu\leq s\\ 1\leq j\leq n}} \widehat \eta\Big( \frac{\gamma_{\chi_{\mu,j}}}{2\pi }\Big); \qquad \Sigma_{\bm \gamma} :=\frac 1{2\pi}\sum_{\substack{1\leq \mu\leq s\\ 1\leq j\leq n}} \gamma_{\chi_{\mu,j}} . 
\end{equation}
Note that under GRH$_{\widehat \eta}$, for each $\bm \gamma \in \Gamma(\bm \chi)$ we have that either $\bm \gamma \in \R^{sn}$, or $\widehat \eta(\bm \gamma)=0$.

  \begin{lemma}
\label{lemma explicit formula for moments of nu n}
Let  $\delta>0$, $\eta \in \S_\delta$, $\Phi\in \mathcal L^1(\mathbb R)$, $\Phi$ even, and assume GRH$_{\widehat \eta}$. For $s  \in \mathbb N$, $n\geq 2,$  $T\geq 1 $ and $q\geq 3$, we have the formula  
\begin{multline*}
\frac {1}{T} \int_0^\infty \Phi\Big(\frac tT\Big)   M_n(\e^t,q;\eta)^s 
\d t 
=
\frac {(-1)^{sn}}{2\phi(q)^{sn}}\!\!\! \sum_{\substack{\bchi\in X_{s,n}}}  \sum_{\substack{ \bm \gamma \in \Gamma(\bchi) }} \widehat \eta(\bm \gamma) \widehat \Phi\big(T\Sigma_{\bm \gamma }\big)+O_\Phi\Big( \frac{( K_{\delta,\eta}\log q)^{sn}}{T\phi(q)^s} \Big),
\end{multline*}
where $K_{\delta,\eta}>0$ is a constant. 
In particular,
  \begin{align*}
\mathbb E[H_n(q;\eta)^s] &=    \frac{(-1)^{sn}} {\phi(q)^{sn}}\sum_{\substack{\bchi\in X_{s,n}}}  \sum_{\substack{{\bm \gamma}\in \Gamma_0(\bchi)}}\widehat \eta(\bm \gamma).
\end{align*}  
\end{lemma}

\begin{proof}
By Lemma~\ref{lemma:explicit formula}, we have that, for any non-principal character $\chi$ of modulus $q$ and $t\geq 0$,

\begin{equation}
 \psi_{\eta}(\e^t,\chi) = -\sum_{\gamma_{\chi}} \e^{ i\gamma_{\chi}  t} \widehat \eta\Big(\frac{\gamma_{\chi}}{2\pi}\Big)+ O(\e^{-\frac t2} \log q), 
 \label{equation explicit formula before conjugation}
\end{equation}
where the sum is taken over the zeros $\rho_{\chi}=\frac 12+i\gamma_{\chi}$ of $L(s,\chi)$. In particular, under GRH$_{\widehat \eta}$, either $\gamma_{\chi}\in \mathbb R$ or $ \widehat \eta(\frac{\gamma_{\chi}}{2\pi})=0$. 
Taking complex conjugates, it follows that 
\begin{equation}
 \psi_{\eta}(\e^t,\overline{\chi}) = -\sum_{\gamma_{\chi}} \e^{-  i\gamma_{\chi}  t} \widehat \eta\Big(\frac{\gamma_{\chi}}{2\pi}\Big)+ O(\e^{-\frac t2} \log q). 
 \label{equation explicit formual conjuguate}
\end{equation}
 We see that 
\begin{align*}
 \frac {(-1)^{sn}}{\phi(q)^{sn}} & \sum_{\substack{ \bchi\in X_{s,n}}}   \sum_{\substack{ \bm \gamma \in \Gamma(\bchi)}} \widehat \eta(\bm \gamma) \widehat \Phi(T\Sigma_{\bm \gamma})  
 = \frac {1}{T} \int_\R \Phi\Big(\frac tT\Big) \frac {(-1)^{sn}}{\phi(q)^{sn}}     \sum_{\substack{\bchi\in X_{s,n}}} \sum_{\substack{ \bm\gamma\in \Gamma(\bchi)  }} \widehat \eta(\bm \gamma)   \e^{-2\pi  i t\Sigma_{\bm \gamma}}\d t\\
& = \frac {1}{T} \int_\R \Phi\Big(\frac tT\Big)  \Big(\frac {(-1)^{n}}{\phi(q)^{n}}      \sum_{\substack{ \bchi\in X_{1,n}}} \sum_{\substack{  \gamma_{\chi_1} ,\dots ,\gamma_{\chi_n} }} \widehat \eta\Big( \frac{\gamma_{\chi_1}}{2\pi }\Big) \cdots \widehat \eta\Big( \frac{\gamma_{\chi_n}}{2\pi }\Big)  \e^{- i t (\gamma_{\chi_1}+\dots +\gamma_{\chi_n})}\Big)^s\d t.
\end{align*}
We split the integral as $ \int_\mathbb R =  \int_0^{\infty}+\int_{-\infty}^0$. In the first of these, an application of~\eqref{equation explicit formual conjuguate} and~\eqref{leqlogq} shows that it is
\begin{align*}
&=\frac {1}{T} \int_0^\infty \Phi\Big(\frac tT\Big)   \Big(\frac 1{\phi(q)^{n}}\sum_{\substack{ \bchi\in X_{1,n}}}\psi_\eta(\e^{t},\overline{\chi_1})\cdots \psi_\eta(\e^{t},\overline{\chi_n}) \Big)^s \d t
 +O_\Phi\Big( \frac{( K_{\delta,\eta}\log q)^{sn}}{T\phi(q)^s} \Big) \\
&=\frac {1}{T} \int_0^\infty \Phi\Big(\frac tT\Big)   M_n(\e^t,q;\eta)^s 
\d t  +O_\Phi\Big( \frac{( K_{\delta,\eta}\log q)^{sn}}{T\phi(q)^s} \Big),
\end{align*}
by~\eqref{eqMn} and since $M_n(\e^t,q;\eta)$ is real.
Making the change of variables $u=-t$ and applying~\eqref{equation explicit formula before conjugation}, we see that the same holds for the integral between $-\infty$ and $0$, and the claimed estimate follows.

As for the claimed identity, we specialize to $ \Phi=\Phi_0=1_{[-1,1]} $. Taking $T\rightarrow \infty$ and applying Lemma~\ref{lemma formula for moments of limiting distribution}, we obtain that
$$ \mathbb E[G_n(q;\eta)^s]= \lim_{T\rightarrow \infty}\frac {(-1)^{sn}}{2\phi(q)^{sn}} \sum_{\substack{\bchi\in X_{s,n}}}  \sum_{\substack{ \bm \gamma \in \Gamma(\bchi)}} \widehat \eta(\bm \gamma) \widehat \Phi_0(T\Sigma_{\bm \gamma}).   $$
Now, $\widehat \Phi_0(\xi) = \sin(2\pi \xi)/\pi\xi$, and by dominated convergence (note that for any $\eta\in \S_\delta$, the sum over $\bm \gamma$ converges absolutely) we can interchange the limit and the sum over $\bm \gamma$. The claimed identity follows.
\end{proof}

\section{Expected value of moments}
\label{section mean of first moment} 
 
This section serves as a warm-up for the following one. We recall that $\U\subset \mathcal L^1(\R)$ is the set of non-trivial even integrable functions $\Phi:\R \rightarrow \R$ such that $\widehat \Phi\geq 0$.  
Our goal is to estimate the first moment of the $n$-th moment 
 \begin{equation} \cM_{1,n}(T,q;\eta,\Phi) :=\frac {1}{T \int_0^\infty \Phi} \int_{0}^{\infty} \Phi\Big( \frac tT \Big) M_n(\e^t,q;\eta)\d t.
\label{equation definition M_1,n} 
 \end{equation}
We begin with the following combinatorial lemma.
 \begin{lemma}
 Let $m\geq 1$,  $\delta>0$, $\eta \in \S_{\delta}$, and assume GRH$_{\widehat{\eta}}$. Then, for $q\geq 3$,  we have the lower bound
 \begin{equation}\begin{split} 
\frac{1}{\phi(q)^{2m}}& \sum_{\substack{\bchi\in X_{1,2m}}} \sum_{\substack{\gamma_{\chi_1} , \dots, \gamma_{\chi_{2m}} \\ \gamma_{\chi_1} + \dots + \gamma_{\chi_{2m}}=0}} 
\widehat\eta\Big(\frac{\gamma_{\chi_1}}{2\pi}\Big)\cdots \widehat\eta\Big(\frac{\gamma_{\chi_{2m}}}{2\pi}\Big)  \\& \geq  \mu_{2m}\Big( \frac{ \alpha(\widehat \eta^2)\log q+\beta_q(\widehat \eta^2) }{\phi(q)}\Big)^m   +O\Big(\mu_{2m}\frac{ (C_{\delta,\eta}\log q)^{m}}{\phi(q)^{m+1 }} \Big),
\end{split}\label{equation lemma lower bound sum zeros}
   \end{equation}
where the $\rho_{\chi_j}=\frac 12+i \gamma_{\chi_j}$ are running over the non-trivial zeros of $L(s,\chi_j)$, and $C_{\delta,\eta}>0$ is a constant. 
 \label{lemma lower bound sum over zeros}
 \end{lemma}

\begin{proof}
By positivity of $\widehat\eta(\frac{\gamma_{\chi_j}}{2\pi})$ (GRH$_{\widehat \eta}$ is crucial in this step), we can restrict the sum over $(\gamma_{\chi_i})_{1\leq i\leq n}$ to those ordinates of zeros for which $\forall i, \exists j\neq i$ s.t. $\gamma_{\chi_i}=-\gamma_{\chi_j}$ and $\overline{\chi_j}=\chi_i$. 
We can further restrict the sum such that the associated pairs $\{ i_k,j_k\} $ are disjoint, and such that $\chi_{i_k}\notin\{ \chi_\ell, \overline{\chi_\ell}\quad \forall \ell\neq i_k,j_k\} $. We may also take $i_k,j_k$ such that for any $k$ the index $i_k$ satisfies 
$$i_k=\min \{ i_\ell, j_\ell \quad (\ell\geq k)\}.$$ In particular $i_1=1$. The number of such 
 $ i_1,j_1,\ldots  i_k, j_k, \ldots,  i_m, j_m $ is
$$=(2m-1)(2m-3)\cdots  1 =\mu_{2m}.$$
For a given choice of sets $\{ i_k,j_k\} $, the contribution to the left hand side of~\eqref{equation lemma lower bound sum zeros} is
$$\geq\frac{ 1}{\phi(q)^{2m}} \sum_{\chi_{i_1}\neq \chi_{0,q}}b^+(\chi_{i_1};\widehat \eta^2) 
\sum_{\chi_{i_2}\notin\{ \chi_{0,q},\chi_{i_1},\overline{\chi_{i_1}}\}}\!\!\! b^+(\chi_{i_2};\widehat \eta^2) \ldots \!\!\!\!\!\!\!\!\!\!\!\!\!\!\!\!
\sum_{\chi_{i_m}\notin\{ \chi_{0,q},\chi_{i_1},\ldots, \chi_{i_{m-1},
\overline{\chi_{i_1}},\ldots, \overline{\chi_{i_{m-1}}}}\}}\!\!\!\!\!\!\!\!\!\!\!\!\!\!\!\!\!\!\!\!\!\!\!\! b^+(\chi_{i_m};\widehat \eta^2),$$
where denoting by $\sum^*$ a sum without multiplicities,
\begin{equation}\label{inegb+b}b^+(\chi;\widehat \eta^2):= \sums_{\gamma_{\chi}}\big(\ord_{s=\frac12+i\gamma_{\chi} }L(s,\chi)\big)^2 \Big|\widehat \eta\Big(\frac{\gamma_{\chi}}{2\pi}\Big)\Big|^2\geq b(\chi;\widehat \eta^2). 
\end{equation}
After applying this lower bound, 
we can forget the conditions $$\chi_{i_k}\notin\{ \chi_{i_1 },\ldots, \chi_{i_{k-1}}, \overline{\chi_{i_1}},\ldots, \overline{\chi_{i_{k}}}\}$$
by adding an error term bounded by
$\ll m^2 \phi(q)^{-m-1}(K_{\delta,\eta}\log q)^{m}.$ This yields that the left hand side of~\eqref{equation lemma lower bound sum zeros} is
$$  \geq     \mu_{2m}\frac{ 1}{\phi(q)^{2m}} \Big(\sum_{\substack{ \chi \bmod q \\ \chi \neq \chi_{0,q}}} b(\chi;\widehat \eta^2)\Big)^{m}+O\Big(  m^2\mu_{2m}\frac{(K_{\delta,\eta}\log q)^{m}}{\phi(q)^{m+1}} \Big),$$
which by Proposition~\ref{proposition first moment b(chi,h)} is
\begin{align*} \geq   &  \mu_{2m}\Big( \frac{ \alpha(\widehat \eta^2)\log q+\beta_q(\widehat \eta^2) }{\phi(q)}\Big)^m   +O\Big(\mu_{2m}\frac{ (2K_{\delta,\eta}\log q)^{m}}{\phi(q)^{m+1 }} \Big).
\end{align*}  
The claimed lower bound follows. 
\end{proof}

We deduce the following inequalities on the mean of moments $\cM_{1,2m}(T,q;\eta,\Phi)$, which is defined in \eqref{equation definition M_1,n}. 
\begin{proposition} Let $m\geq 1$, $\Phi\in \U$,  $\delta>0$, $\eta\in \S_\delta$, and assume GRH$_{\widehat \eta}$. For any $T\geq 1$ and $q\geq 3$, we have that
$$\cM_{1,2m}(T,q;\eta,\Phi)\geq   \mu_{2m}\Big( \frac{ \alpha(\widehat \eta^2)\log q+\beta_q(\widehat \eta^2) }{\phi(q)}\Big)^m +O\Big(\mu_{2m} \frac{ (C_{\delta,\eta}\log q)^{m}}{\phi(q)^{m+1 }}+ \frac{(K_{\delta,\eta}\log q)^{2m}}{ T\phi(q)}\Big);$$
$$ 
\cM_{1,2m+1}(T,q;\eta,\Phi)\leq O\Big( \frac{(K_{\delta,\eta}\log q)^{2m+1}}{ T\phi(q)}\Big),
$$
where $K_{\delta,\eta},C_{\delta,\eta}>0$ are constants. 
\label{proposition first moment of moments}
\end{proposition}

\begin{proof}
We apply positivity in Lemma~\ref{lemma explicit formula for moments of nu n}, noting that  $\widehat \eta , \widehat \Phi \geq 0$. We obtain the lower bound
$$(-1)^n\cM_{1,n}(T,q;\eta,\Phi)\geq \frac {1}{\phi(q)^{n}}\!\!\! \sum_{\substack{\bchi\in X_{1,n}}} \sum_{\substack{\gamma_{\chi_1} , \dots, \gamma_{\chi_{n}} \\ \gamma_{\chi_1} + \dots + \gamma_{\chi_{n}}=0}} \!\!\!
\widehat \eta\Big(\frac{\gamma_{\chi_1}}{2\pi}\Big)\cdots \widehat\eta\Big(\frac{\gamma_{\chi_{n}}}{2\pi}\Big)
+O\Big( \frac{(K_{\delta,\eta}\log q)^{n} }{ T\phi(q)}\Big). 
$$
If $n=2m+1$, then the claimed bound follows at once from positivity of $\widehat \eta(\frac{\gamma_{\chi_j}}{2\pi})$ (hypothesis GRH$_{\widehat \eta}$ is crucial here). If $n=2m$, then we apply Lemma~\ref{lemma lower bound sum over zeros} and obtain the desired result.
\end{proof}

We are ready to prove the second part of Theorem~\ref{thmomentscentres2rT}.
\begin{proof}[Proof of Theorem~\ref{thmomentscentres2rT}, second part]
The claimed bounds follow from positivity in Lemma~\ref{lemma explicit formula for moments of nu n} (recall also Lemma~\ref{lemma moments of limiting distribution as limit of V(T)}), and from Proposition~\ref{proposition first moment of moments}.
\end{proof}

 \section{Higher moments of moments}
\label{section higher moments of moments}
The goal of this section is to prove Theorem~\ref{thmomentscentres2rT}.
We first apply the inclusion-exclusion principle in order to evaluate the main terms which we will obtain later. For $s\in \mathbb N$, $ \bm \sigma = (\sigma_1,\dots,\sigma_s)\in \R^s,$  $T\in \R_{>0}$ and $\Phi\in \mathcal U$,
we define the quantities
\begin{equation}\label{defDelta2r} \Delta_{s}(\bfsigma ):= \sum_{I\subset\{ 1,\ldots, s\}}(-1)^{s-|I|}  \delta_0\Big(   \sum_{\mu \in I}\sigma_\mu  \Big) 
\prod_{\mu \notin I}\delta_0(   \sigma_\mu  ) ;
\end{equation}
  \begin{equation}\label{defDelta2rT} \Delta_{s}(\bfsigma;\Phi,T):= \frac 1{\widehat \Phi(0)}\sum_{I\subset\{ 1,\ldots, s\}}(-1)^{s-|I|}  \widehat\Phi\Big(  T\sum_{\mu \in I}\sigma_\mu  \Big) 
\prod_{\mu \notin I}\delta_0(   \sigma_\mu  ) ,
\end{equation}
where 
$$ \delta_0(x):= \begin{cases}
1 &\text{ if } x=0; \\
0& \text{ otherwise.}
\end{cases}$$
Note that 
$$\lim_{T\to+\infty}\Delta_{s}(\bfsigma;\Phi,T)=
\Delta_{s}(\bfsigma ).$$

\begin{lemma}\label{inegDelta} Let $s\in \mathbb N$, $ \bfsigma \in \R^s,$  $T\in \R_{>0}$ and $\Phi\in \mathcal U$. The function $\Delta_{s}(\bfsigma)$ is the indicator function of the 
 $\bfsigma$ such that $ \sum_{\mu=1}^{s}\sigma_\mu =0$ and $\sigma_\mu\neq 0$ for all $1\leq \mu\leq s$. Moreover, 
 
 \begin{equation}
      \Delta_{s}(\bfsigma;\Phi, T)\geq 
\Delta_{s}(\bfsigma ).
\label{equation lower bound Delta T}
 \end{equation}
\end{lemma}
\begin{proof}
    
We observe that 
$$\widehat \Phi(0) \Delta_{s}(\bfsigma;\Phi,T)=\sum_{\substack{I\subset\{ 1,\ldots, s\}: \\ \forall \mu \notin I,\sigma_\mu = 0 }}(-1)^{s-|I|}\widehat \Phi\Big(\sum_{\mu =1}^s \sigma_\mu\Big)
=\widehat \Phi\Big(\sum_{\mu =1}^s \sigma_\mu\Big)\prod_{\mu=1}^s \big(1-\delta_0(\sigma_\mu)\big).
$$
We then deduce that \begin{equation}
    \label{Deltas}\Delta_{s}(\bfsigma;\Phi,T) =
\begin{cases} \frac{1}{\widehat \Phi(0)} \widehat\Phi\Big(  T\sum_{\mu=1} ^ s \sigma_\mu \Big) &\text{ if } \sigma_{\mu} \neq 0 \text{ for all } \mu,\\ 
 0 & \text{ otherwise.} 
 \end{cases}
\end{equation}
 The claimed bound follows.  
\end{proof}

We recall the definition~\eqref{defV2rnT} of $\cV_{s,n}(T,q;\eta,\Phi) $, that of $\Gamma(\bm \chi)$ in~\eqref{equation definition Gamma(chi)},  
and moreover for $\bm \gamma \in \Gamma(\bm \chi)$ we define
\begin{equation}
    \bm \sigma_{\bm \gamma} :=\Big(\frac 1{2\pi}\sum_{j=1}^n \gamma_{\chi_{\mu,j}}\Big)_{1\leq \mu\leq s} \in \C^s.\label{defsigmagamma}
\end{equation}
We also recall the definition~\eqref{defXsn}. Our next goal is to express $\cV_{s,n}(T,q;\eta,\Phi)$ as a sum over zeros of Dirichlet $L$-functions.

\goodbreak
\begin{lemma}
\label{lemma explicit formula higher moments of moments}
Assume GRH$_{\widehat \eta}$, and let  $\delta>0$, $\eta \in \S_\delta$, $\Phi\in \mathcal L^1(\R),$ $\Phi$ even, $s \geq 1$ and $n\geq 2$. For $T\geq 1 $ and $q\geq 3$,  we have the formula 
\begin{equation}
      \cV_{s,n}(T,q;\eta,\Phi) = \frac {(-1)^{sn}}{\phi(q)^{sn}} 
\sum_{\substack{ \bchi\in X_{s,n}}}  \sum_{\substack{ \bm \gamma \in \Gamma(\bchi) }} \widehat \eta(\bm \gamma) \Delta_{s}(\bm \sigma_{\bm\gamma};\Phi,T) +O_\Phi\Big( \frac{( K_{\delta,\eta}\log q)^{sn}}{T\phi(q)^s} \Big),
\label{equation lemma lower bound centered moments}
\end{equation} 
where   
$\widehat \eta(\bm \gamma)$ is defined in~\eqref{defetabfgamma}, $\Gamma(\bchi)$ is defined in~\eqref{equation definition Gamma(chi)} and $K_{\delta,\eta}>0$ is a constant.   
\end{lemma}

\begin{proof} 
The proof follows the lines of that of Lemma~\ref{lemma explicit formula for moments of nu n}. 
We see that
\begin{align*}
&L_{s,n}(T,q;\eta,\Phi):=\frac {(-1)^{sn}}{\phi(q)^{sn}}\!\!\! \sum_{\substack{ \bchi\in X_{s,n}}}  \sum_{\substack{ \bm \gamma \in \Gamma(\bchi) }} \widehat \eta(\bm \gamma) \Delta_{s}(\bm \sigma_{\bm\gamma};\Phi,T)  \\
& = \frac {1}{T\widehat \Phi(0)} \int_\R \Phi\Big(\frac tT\Big) \frac {(-1)^{sn}}{\phi(q)^{sn}} \sum_{I\subset\{ 1,\ldots, s\}}(-1)^{s-|I|}     \sum_{\substack{ \bchi\in X_{s,n}}} \sum_{\substack{ \bm \gamma     \in \Gamma(\bchi) \\\forall \mu \notin I,\,\,  \sum_{j=1}^{n} \gamma_{\chi_{\mu,j}}=0}} \widehat \eta(\bm \gamma)   \e^{- i t\sum_{\gamma \in \bm\gamma} \gamma}\d t\\
& = \frac {1}{T\widehat \Phi(0)} \int_\R \Phi\Big(\frac tT\Big) \frac {(-1)^{sn}}{\phi(q)^{sn}} \sum_{I\subset\{ 1,\ldots, s\}}(-1)^{s-|I|}    \!\!\! \!\!\!\!\!  \sum_{\substack{ \chi_{\mu,1},\dots,\chi_{\mu,n} \neq \chi_{0,q} \\ \chi_{\mu,1}\cdots \chi_{\mu,n} =\chi_{0,q}\\ \mu \in I}}  \sum_{\substack{ \bm \gamma = (\gamma_{\chi_{\mu,j}})_{ \substack{\mu \in I \\ j\leq n}}  }} \widehat \eta(\bm \gamma)   \e^{- i t\sum_{\gamma \in \bm\gamma} \gamma} \\
&
\hspace{8cm}\times  \!\!\! \sum_{\substack{ \chi_{\mu,1},\dots,\chi_{\mu,n} \neq \chi_{0,q} \\ \chi_{\mu,1}\cdots \chi_{\mu,n} =\chi_{0,q}\\ \mu \notin I}} \sum_{\substack{ \bm \gamma = (\gamma_{\chi_{\mu,j}})_{ \substack{\mu \notin I \\ j\leq n}}  \\\forall \mu \notin I ,\, \sum_{j=1}^{n} \gamma_{\chi_{\mu,j}}=0}} \widehat \eta(\bm \gamma)
\d t.
\end{align*}
Now, the sums over $\chi_{\mu,j}$ depend on $|I|$ rather than on $I$ itself, and thus we deduce that
\begin{align*}
L_{s,n}(T,q;\eta,\Phi)&= \frac {1}{T\widehat \Phi(0)} \int_\R \Phi\Big(\frac tT\Big) \frac {(-1)^{sn}}{\phi(q)^{sn}} \sum_{k=0}^s(-1)^{s-k}  \binom sk  \Big(\!\!\! \sum_{\substack{ \bchi\in X_{1,n}}}   \prod_{j=1}^n  \sum_{\substack{ \gamma_{\chi_j} }}\widehat\eta\Big(\frac{\gamma_{\chi_j}}{2\pi} \Big)   \e^{- i t \gamma_{\chi_j}} \Big)^k \\
&
\hspace{4cm}\times \Big(  \sum_{\substack{ \bchi\in X_{1,n}}} \sum_{\substack{\gamma_{\chi_1} , \dots, \gamma_{\chi_{n}} \\ \gamma_{\chi_1} + \dots + \gamma_{\chi_{n}}=0}} \prod_{j=1}^n \widehat\eta\Big(\frac{\gamma_{\chi_j}}{2\pi} \Big)\Big)^{s-k}
\d t.
\end{align*}
By Lemmas~\ref{lemma moments of limiting distribution as limit of V(T)} and~\ref{lemma explicit formula for moments of nu n}, the second term in parentheses is equal to $ (-1)^n\phi(q)^nm_n(q;\eta)$. We split the integral as $ \int_\mathbb R =  \int_0^{\infty}+\int_{-\infty}^0$. In the first of these, an application of~\eqref{equation explicit formual conjuguate} and~\eqref{leqlogq} shows that it is
\begin{align*}
&=\frac {1}{T\widehat \Phi(0)} \int_0^\infty \Phi\Big(\frac tT\Big) \frac 1{\phi(q)^{sn}} \sum_{k=0}^s(-1)^{s-k}  \binom sk \Big(\sum_{\substack{ \bchi\in X_{1,n}}}\psi_\eta(\e^{t},\overline{\chi_1})\cdots \psi_\eta(\e^{t},\overline{\chi_n}) \Big)^k \times 
 \\ &\hspace{4cm} \big(\phi(q)^nm_n(q;\eta)\big)^{s-k}\d t
 +O_\Phi\Big( \frac{( K_{\delta,\eta}\log q)^{sn}}{T\phi(q)^s} \Big) \\
&=\frac {1}{T\widehat \Phi(0)} \int_0^\infty \Phi\Big(\frac tT\Big)  \sum_{k=0}^s(-1)^{s-k} \binom sk  M_n(\e^t,q;\eta) ^k m_n(q;\eta)^{s-k}
\d t +O_\Phi\Big( \frac{( K_{\delta,\eta}\log q)^{sn}}{T\phi(q)^s} \Big),
\end{align*}
by~\eqref{eqMn}.
Once more, the same holds for the integral between $-\infty$ and $0$. An application of the binomial theorem achieves the proof.
\end{proof}

We will establish a lower bound for $\cV_{s,n}(T,q;\eta,\Phi)$ through positivity of $\widehat \eta$ and $\widehat \Phi$. This will involve the following combinatorial object.
\begin{definition}
For $r \in \mathbb N$ and  $n\geq 2,$ we define $F_{2r,n}$ to be the set of involutions $\pi: \{ 1,\dots,2r\} \times \{ 1,\dots, n\} \rightarrow \{ 1,\dots,2r\} \times \{ 1,\dots, n\}  $ having no fixed point and with the following two properties. First, for each fixed $1\leq \mu \leq 2r   $, there exists a unique $ 1\leq \nu \leq  2r $, $\nu \neq \mu$ such that $J_{\mu,\nu}(\pi)\neq 0$, where
\begin{equation}
J_{\mu,\nu}(\pi):=      (\{\mu\} \times [1,n]) \cap \pi (\{ \nu\}\times [1,n]) 
\label{equation definition Jmunu}
\end{equation}
is set of of elements on the "row" $\mu$ which are sent to some element on the "row" $\nu$ through~$\pi$. Second, we require that for each $1\leq \mu\leq 2r$,
$$ \sum_{\substack{1\leq \nu\leq 2r\\ \nu\neq \mu}}|J_{\mu,\nu}(\pi)| \geq 2.$$ 
Given an involution $\pi\in F_{2r,n}$ and an array $\bchi=(\chi_{\mu,j})_{\substack{ 1\leq \mu \leq 2r \\ 1\leq j\leq n}}$ of characters modulo $q$, we will write $\pi(\bchi):= (\chi_{\pi(\mu,j)})_{\substack{ 1\leq \mu \leq 2r \\ 1\leq j\leq n}}$.
\label{definition F}
\end{definition}
One can identify $F_{2r,n}$ with the set $\mathcal F_{2r,n}$ of pairs $\{(\mu,j),(\nu,i)\}$ with $1\leq \mu,\nu\leq 2r$, $1\leq j,i\leq n$,  $(\mu,j)\neq(\nu,i)$, and such that for each fixed $\mu$, 
$$ \big| \big\{ \nu \in [1,2r]\smallsetminus\{ \mu\} \mid \exists j,i :\quad   \{(\mu,j),(\nu,i)\} \in  \mathcal F_{2r,n} \big\}\big| = 1,  $$
$$ \sum_{\substack{1\leq \nu \leq 2r \\ \nu \neq \mu} } \big| \big\{ j \in [1,n] \mid \exists i : \quad  \{(\mu,j),(\nu,i)\} \in  \mathcal F_{2r,n} \big\}\big| \geq 2. $$
We note that for any $\pi \in F_{2r,n}$ we have that $|J_{\nu,\mu}(\pi)|=|J_{\mu,\nu}(\pi)|$, and moreover for any fixed $1\leq \mu \leq 2r$,
\begin{equation}\label{equation conditionsurk}|J_{\mu,\mu}(\pi)|\equiv 0 \bmod 2, \qquad\sum_{\nu=1}^{2r}|J_{\mu,\nu}(\pi)|=n.
\end{equation}

We begin with the following technical lemma, whose proof gives an idea on how we will count characters.

\begin{lemma}
\label{lemma removal star}
Let $ r\in \mathbb N$, $n\geq 2,$ $q\geq 3$ and let $\pi \in F_{2r,n}$ be such that $ J_{\mu,\nu}(\pi) 
\neq \varnothing $
if either $\mu=\nu$ or $\{ \mu,\nu\} = \{ 2u-1,2u \}$ for some $1\leq u\leq r$, and $J_{\mu,\nu}(\pi)=\varnothing $ otherwise. Let also $1\leq u \leq r$, and let $(\mu_0,j_0), (\mu_1,j_1)$ be such that $1\leq \mu_1\leq \mu_0, \mu_0\in \{2u-1,2u\}, 1\leq j_0,j_1\leq n$, and $\pi(\mu_0,j_0) \neq (\mu_1,j_1)$. Then, given a fixed array of characters $ (\chi_{\mu,j})_{ \substack{ 1\leq \mu \leq 2u-2 \\ 1\leq j \leq n}}$
we have the bound
\label{lemma bound non distinct}
$$\sum_{\substack{ \chi_{2u-1,1},\dots,\chi_{2u-1,n}, \chi_{2u,1},\dots,\chi_{2u,n} \\ \Theta_{\pi}(2u)=\chi_{0,q}\\ 
 \forall j,\,\, \chi_{\pi(\mu,j)}  =\overline{ \chi_{\mu,j}}(\mu=2u-1,2u) \\ 
  \chi_{\mu,j} \in \{\overline{\chi_{\nu,i}} |  i\leq n, \nu\leq \mu ,(\i,\nu)\notin\{ (j,\mu),\pi(j,\mu)\} \}
 \chi_{\mu_0,j_0}= \overline{ \chi_{\mu_1,j_1}}  
} }
1 \ll   \phi(q)^{n-2},$$
where $\Theta_{\pi}(2u):=\prod_{\substack{ 1\leq  j\leq n \\  
j\notin J_{2u,2u}}} \chi_{2u,j} $.
The implied constant is absolute.
\end{lemma}

\begin{proof} The second relation of \eqref{equation conditionsurk} applied for $\mu=2u$ and $\mu=2u-1$ gives
$$|J_{2u,2u}|+|J_{2u,2u-1}| +|J_{2u-1,2u-1}|+ |J_{2u-1,2u}|
=|J_{2u,2u}|+2|J_{2u,2u-1}| +|J_{2u-1,2u-1}|=2n.$$ 
The pairs $(\mu_0,j_0),(\mu_1,j_1)$ 
are elements of either   $J_{2u-1,2u}$, $J_{2u,2u-1}$, $J_{2u,2u}$ or $J_{2u-1,2u-1}$. 
First we will treat the case where $(\mu_0,j_0),(\mu_1,j_1)\in J_{2u-1,2u-1}$, and thus $(\mu_0,j_0)=(2u-1,j_0)$, $ (\mu_1,j_1)=(2u-1,j_1)$ and $\pi(\mu_0,j_0)=(2u-1,j_0')$, $ \pi(\mu_1,j_1)=(2u-1,j_1')$ with $j_0,j_0',j_1,j_1'$ all distinct. We have that  
 $$\sum_{\substack{ \chi_{\mu,j}; (\mu,j)\in J_{2u-1,2u-1}\cup J_{2u,2u}\\ \chi_{2u-1,j_0} =\overline{ \chi_{2u-1,j_1}} 
} }   
1   \leq \phi(q)^{\frac{|J_{2u,2u}|}2+\frac{|J_{2u-1,2u-1}|}2-1}.$$
 Indeed, the character $\chi_{2u-1,j_0}$ is determined by the value of $\chi_{2u-1,j_1}$ through the relation $\chi_{2u-1,j_0} =\overline{ \chi_{2u-1,j_1}} 
.$
Moreover, we claim that 
$$\sum_{\substack{ \chi_{\mu,j}; (\mu,j)\in J_{2u,2u-1}  \\  
 \forall j,\,\, \chi_{\pi(\mu,j)}  =\overline{ \chi_{\mu,j}}(\mu=2u-1,2u)    
\\ \prod_{   
j\in J_{2u,2u-1}} \chi_{2u,j}=\chi_{0,q} } }  
1   \leq \phi(q)^{|J_{2u-1,2u}|-{1}}.$$
Indeed, in this sum we may only sum over the characters $\chi_{\mu,j}$ for $(\mu,j) \in J_{2u-1,2u} $, as these determine the values of the characters for $(\mu,j) \in J_{2u,2u-1} $ thanks to the relation $\chi_{\pi(\mu,j)}  =\overline{ \chi_{\mu,j}}$.
Finally, one of $\chi_{\mu,j}$ for $(\mu,j) \in J_{2u,2u-1} $ is determined by the other characters $\chi_{\mu,i}$ for $(\mu,i) \in J_{2u,2u-1} $ thanks to the relation $\prod_{ 
j\in J_{2u-1,2u }  }  \chi_{2u-1,j}=\chi_{0,q} $.  This implies the claimed upper bound. 
The bound for the case $(\mu_0,j_0)\in J_{2u-1,2u-1},(\mu_1,j_1)\in J_{2u ,2u }$ follows along the same lines.\par Secondly, 
we treat the case where $(\mu_0,j_0),(\mu_1,j_1)\in J_{2u-1,2u}$ (in particular this implies that $\mu_0=\mu_1=2u-1$); the proof for the other cases is similar. 
If $\pi(2u-1,j_0)=(2u,j_0')$ and $\pi(2u-1,j_1)=(2u,j_1')$, then the sum we are interested in is actually equal to
$$ \sum_{\substack{ \chi_{2u-1,1},\dots,\chi_{2u-1,n}, \chi_{2u,1},\dots,\chi_{2u,n} \\ \Theta_{\pi'}(2u)=\chi_{0,q}\\ 
 \forall j,\,\, \chi_{\pi'(\mu,j)}  =\overline{ \chi_{\mu,j}}(\mu=2u-1,2u) \\ 
 \chi_{2u-1,j_0}= \overline{ \chi_{2u,j_0'}}  
} }
1, $$
where $\pi'$ is equal to $\pi$ except for the values $ \pi'(2u-1,j_0)=(2u-1,j_1)$ and $ \pi'(2u,j_0')=(2u,j_1'). $
Thus, we have reduced the problem to the case treated in the first part of the proof, and the result follows.
\end{proof}

We continue to estimate sums over characters in the following lemma.
\begin{lemma}
Let $ r\in \mathbb N$, $n\geq 2,$ $q\geq 3$ and let $\pi \in F_{2r,n}$. We have the estimate
\label{lemma sum over characters}
\begin{equation}\label{equation C_{2r,n}}
C_{2r,n}(q;\pi):=\sum_{\substack{\bchi\in X_{2r,n}
\\   \pi(\bchi) =\overline{ \bchi}\\  \chi_{\mu,j} =\overline{\chi_{\nu,i}} \Rightarrow (\mu,j)\in \{ (\nu,i),\pi(\nu,i)\}    } }    1 =  \phi(q)^{(n-1)r} \Big(1+ O\Big(   \frac {n^2r^2}{\phi(q)}\Big) \Big).
\end{equation}
If moreover $q$ is large enough in terms of $\delta$ and $\eta$, then in the ranges $n\leq \log q/\log_2q $ and $r\leq c\phi(q)^{\frac 12}/n$ with $c>0$ a small enough absolute constant,
we have the estimate
\begin{equation}
\!\!\!\sum_{\substack{\bchi \in X_{2r,n} \\ \pi(\bchi)=\bchi \\ \chi_{\mu,j} = \overline{\chi_{\nu,i}} \Rightarrow (\mu,j) \in \{ (\nu,i),\pi(\nu,i) \} }} \prod_{\substack{ 1\leq \mu \leq 2r \\ 1\leq j \leq n }}b(\chi_{\mu,j};\widehat \eta^2)^{\frac 12}  = \phi(q)^{(n-1)r}(\alpha(\widehat \eta^2)\log q)^{nr}   \Big( 1+O \Big( n \frac {\log_2 q}{ \log q } \Big) \Big)^r.
\label{equation average b(chi) F_2r}
   \end{equation}
  Under the additional assumption of GRH$_{\widehat \eta}$, if $1\leq \mu_0 \neq \nu_0 \leq 2r$ are such that $J_{\mu_0,\nu_0}(\pi) \neq \varnothing$, then we have the bound
    \begin{multline}
\!\!\!\!\!    \sum_{\substack{\bchi \in X_{2r,n} \\ \pi(\bchi)=\bchi \\ \chi_{\mu,j} = \overline{\chi_{\nu,i}} \Rightarrow (\mu,j) \in \{ (\nu,i),\pi(\nu,i) \} }}  \!\!\!\!\!\sup_{\substack{z\in \R }}\Big\{ \!\!\!\!\! \!\!
\sums_{\substack{\rho_{\chi_{\mu_0,1}}, \dots, \rho_{\chi_{\mu_0,n}}\\ \rho_{\chi_{\nu_0,1}}, \dots, \rho_{\chi_{\nu_0,n}}  \\ \forall j,  \,\,\,  \gamma_{\chi_{\mu,j}} +\gamma_{\chi_{\pi(\mu,j)} }=0 \,\,\, (\mu \in \{ \mu_0,\nu_0\})\\ \gamma_{\chi_{\mu_0,1}} + \dots + \gamma_{\chi_{\mu_0,k}} = z}} \!\!\!\!\!\!\!\!\!\!\!\!  \prod_{\substack{ 1\leq j\leq n \\ \mu \in \{ \mu_0,\nu_0\} }}\Big(\widehat \eta\Big(\frac{\gamma _{\chi_{\mu,j}}}{2\pi } \Big) (\ord_{s=\rho_{\chi_{\mu,j}}}L(s,\chi_{\mu,j}))^{\frac 12}  \Big) \Big\} \times 
\\ \prod_{\substack{ 1\leq \mu \leq 2r \\ 1\leq j \leq n \\ \mu\notin \{\mu_0,\nu_0\} }}b(\chi_{\mu,j};\widehat \eta^2)^{\frac 12} 
\ll_{\delta,\eta}  \phi(q)^{(n-1)r}(\alpha(\widehat \eta^2)\log q)^{nr-1}.
\label{equation average b(chi) with one zero removed}
\end{multline}
\end{lemma}

\begin{proof}
We begin with~\eqref{equation C_{2r,n}}. Fix  $\pi \in F_{2r,s}$. In the sum defining $C_{2r,n}(q;\pi)$ and by definition of $F_{2r,n}$, we can change the indexing in the variable $\mu$ in such a way that 
$ J_{\mu,\nu}(\pi) 
\neq \varnothing $ (recall~\eqref{equation definition Jmunu}) if $\{ \mu,\nu\} \in \{ 2u-1,2u \}$ for some $1\leq u\leq r$, and $J_{\mu,\nu}(\pi)=\varnothing $ otherwise.

We distinguish two different sets of indices:
\begin{align*}
      J_\pi&:= \bigcup_{\mu=1}^{2r}J_{\mu,\mu}(\pi) = \big\{ (\mu,j) : \pi(\mu,j) \in \{\mu\}\times [1,n] \big\}; \cr I_\pi&:= \bigcup_{1\leq \mu\neq \nu\leq 2r} J_{\mu,\nu}(\pi)= \big\{ (\mu,j) : \pi(\mu,j) \notin \{\mu\}\times [1,n] \big\},  
\end{align*}
and we recall the relations~\eqref{equation conditionsurk}.
For each fixed $1\leq u\leq r$, the conditions $\chi_{2u-1,1}\cdots \chi_{2u-1,n} =\chi_{2u,1}\cdots \chi_{2u,n} =\chi_{0,q}$ are equivalent to $\Theta_{\pi}(2u):=\prod_{\substack{ 1\leq  j\leq n \\ (j,2u)\in I_\pi}} \chi_{2u,j} = \chi_{0,q}$. Note also that the condition $\chi_{\mu,j} =\overline{\chi_{\nu,i}} \Rightarrow (\mu,j)\in \{ (\nu,i),\pi(\nu,i)\}  $ implies that for $(\nu,i)\notin \{ (\mu,j),\pi(\mu,j) \}$, $\chi_{\nu,i} \neq \chi_{\mu,j}$.

The choice of $\pi$ implies that we can write  
$\pi(\bchi)=(\pi_2(\bchi_{2}),\pi_4(\bchi_{4}),\ldots,\pi_{2r}(\bchi_{2r})) $ 
where
$\bchi=(\bchi_2,\bchi_4,\ldots,\bchi_{2r}) $ and 
$\bchi_{2u}=(\chi_{\mu,j})_{\substack{\mu\in \{ 2u-1, 2u\}\\ 1\leq j\leq n}}$. 
Then, we can write 
$$
C_{2r,n}(q;\pi)=\sum_{\substack{ \bchi_2\in X_{2,n}\\ \Theta_{\pi}(2)=\chi_{0,q} \\ \pi_{2}(\bchi_{2} ) =\overline{ \bchi_{2}}
} }\sum_{\substack{
\bchi_4\in X_{2,n}\\ \Theta_{\pi}(4)=\chi_{0,q} \\ \pi_{4}(\bchi_{4} ) =\overline{ \bchi_{4}}
} } \cdots   \underset{ \substack{\bchi_{2r}\in X_{2,n}\\ \Theta_{\pi}(2r)=\chi_{0,q} \\ \pi_{2r}(\bchi_{2r} ) =\overline{ \bchi_{2r}}} }{{\sum}^*} 1 
, $$
where the star means that in each of the sums we have the extra condition 
\begin{equation}\label{condition mu j}
 \chi_{\mu,j} \notin \big\{\overline{\chi_{\nu,i}} |  i\leq n, \nu\leq \mu ,( i,\nu)\notin\{ (\mu,j),\pi(\mu,j)\} \big\}\hspace{.5cm}  (\mu =  2u-1,2u;  1 \leq j \leq n).
\end{equation}

We first treat $C_{2r,n}(q;\pi)$ as if there were no star on the sums; we will come back to the original sums later. For any fixed $1\leq u\leq r$, we have that 
$$\sum_{\substack{ \bchi_{2u}\in X_{2,n}
\\ \Theta_{\pi}(2u)=\chi_{0,q}\\ \pi_{2u}(\bchi_{2u})=\overline{\bchi_{2u}}
}}  1 
= \sum_{\substack{ 
\chi_{\mu,j} \neq \chi_{0,q}\\ \mu \in \{2u-1,2u\} ,  (\mu,j) \in J_\pi\\ \forall (\mu,j)\in J_{\pi},\,\,  \chi_{\pi(\mu,j)}  =\overline{ \chi_{\mu,j}} (\mu =2u-1,2u)  
}}  1 \cdot  \sum_{\substack{ \chi_{2u,j} \neq \chi_{0,q}\\  j \in J_{2u,2u-1} \\ \Theta_{\pi}(2u)=\chi_{0,q}  } }  1. $$
The second factor on the right hand side is equal to 
$$ (\phi(q)-1)^{|J_{2u,2u-1}|-1}\Big( 1+O\Big(\frac 1{\phi(q)}\Big)\Big),$$
since $|J_{2u,2u-1}|\geq 2$. As for the first, it is equal to the square of
$$   \sum_{\substack{ \chi_{2u,j} \neq \chi_{0,q}\\ j\in J_{2u,2u}\\  \forall j\in J_{2u,2u}, \,\,\chi_{\pi(2u,j)} =\overline{ \chi_{2u,j}}    }}  1 = (\phi(q)-1)^{\frac12|J_{ 2u,2u}| }. $$
The analogue of the estimate~\eqref{equation C_{2r,n}}, with no star on the sums, follows. We now indicate how to remove the condition~\eqref{condition mu j} using Lemma~\ref{lemma removal star}. 
This is done by summing over the $\leq (2rn)^2$ possibilities for the pair of pairs $(\mu_0,j_0),(\mu_1,j_1)$ for which $\chi_{\mu_0,j_0} =\overline{\chi_{\mu_1,j_1}}$ and $\pi(\mu_0,j_0)\neq (\mu_1,j_1)$; in each of these sums we can apply Lemma~\ref{lemma removal star}. The claimed bound follows.

We now move to~\eqref{equation average b(chi) F_2r}.
Following the steps above, we may once more assume that $ J_{\mu,\nu}(\pi) 
\neq \varnothing $ (recall~\eqref{equation definition Jmunu}) if $\{ \mu,\nu\} \in \{ 2u-1,2u \}$ for some $1\leq u\leq r$, and $J_{\mu,\nu}(\pi)=\varnothing $ otherwise; this implies that the left hand side of~\eqref{equation average b(chi) F_2r} is equal to $$L_{2r,n}(q;\pi):=\sum_{\substack{ \bchi_2\in X_{2,n}\\ \Theta_{\pi}(2)=\chi_{0,q} \\ \pi_{2}(\bchi_{2} ) =\overline{ \bchi_{2}}
} }\sum_{\substack{  
\bchi_4\in X_{2,n}\\ \Theta_{\pi}(4)=\chi_{0,q} \\ \pi_{4}(\bchi_{4} ) =\overline{ \bchi_{4}}
} } \cdots   \underset{ \substack{\bchi_{2r}\in X_{2,n}\\ \Theta_{\pi}(2r)=\chi_{0,q} \\ \pi_{2r}(\bchi_{2r} ) =\overline{ \bchi_{2r}}} }{{\sum}^*} \prod_{\substack{ 1\leq \mu \leq 2r \\ 1\leq j \leq n }}b(\chi_{\mu,j};\widehat \eta^2)^{\frac 12}.  
$$
As before, we may remove the star by introducing an admissible error term; this is possible thanks to  Proposition~\ref{proposition first moment b(chi,h)}
in which we use the bound $q_\chi \leq q$. Now, for any fixed $1\leq u\leq r$, keeping in mind that $ J_{2u,2u-1}\neq \varnothing$, we notice that by Lemma~\ref{lemme moyenne multiple b(chi)}, 
$$ \sum_{\substack{ \chi_{2u,j} \neq \chi_{0,q}\\  j \in J_{2u,2u-1} \\ \Theta_{\pi}(2u)=\chi_{0,q}  } }   \prod_{j\in J_{2u,2u-1}}b(\chi_{2u,j};\widehat \eta^2) = \phi(q)^{|J_{2u,2u-1}|-1}(\alpha(\widehat \eta^2)\log q)^{|J_{2u,2u-1}|}\Big( 1+O \Big(n \frac {\log_2 q  }{ \log q } \Big) \Big). $$
The argument is similar for the factors involving $J_{2u,2u}$ and $J_{2u-1,2u-1}$, and this concludes the proof. The proof of~\eqref{equation average b(chi) with one zero removed} is similar, thanks to Lemma~\ref{lemme moyenne multiple b(chi)}.
 \end{proof}

We are ready to obtain our first lower bound on $\cV_{2r,n}(T,q;\eta,\Phi)$ in terms of $|F_{2r,n}|$.

\begin{lemma}
\label{lemma lower bound centered moments even s}
 Let $\delta>0$, $\eta \in \S_\delta$, $\Phi\in \U,$  $r \geq 1$ and $n\geq 2$, and assume GRH$_{\widehat \eta}$. For $T\geq 1$, $ q$ large enough in terms of $\delta$ and $\eta$, and in the ranges $n\leq \log q/\log_2 q$ and $r\leq \log q$, we have the lower bound
   \begin{equation}\begin{split}
  \cV_{2r,n}(T,q;\eta,\Phi) \geq \frac{V_n(q;\eta)^r}{\nu_n^r} |F_{2r,n}| \Big( \Big( 1+O_{\delta,\eta}\Big( n \frac{\log_2 q}{\log q}\Big)\Big)^r+O_{\delta,\eta}\Big( \frac r{\log q}\Big)\Big)
\\ +O_\Phi\Big( \frac{(K_{\delta,\eta}\log q)^{2rn}}{T\phi(q)^{2r}} \Big),
\end{split}
\label{equation lemma lower bound centered moments even s GRH}
\end{equation}
where $K_{\delta,\eta}>0$ is a constant.
\end{lemma}
\begin{proof}
We recall that we can view the $\pi\in F_{2r,n}$ as involutions on the set of arrays of characters $\bchi=(\chi_{\mu,j})_{\substack{1\leq \mu \leq 2r \\ 1\leq j\leq n}}$ by setting  $\pi(\bchi) = (\chi_{\pi(\mu,j)})_{\substack{1\leq \mu \leq 2r \\ 1\leq j\leq n}} $.

By Lemma~\ref{inegDelta} and by positivity of $\widehat \eta$ and of $\Delta_{2r}(\bm \sigma_{\bm\gamma};\Phi,T)$,
we can restrict the sum on the right hand side of~\eqref{equation lemma lower bound centered moments} as follows:
\begin{equation}
    \begin{split}
      &  \cV_{2r,n}(T,q;\eta,\Phi) 
      \\  &\geq
\frac 1{\phi(q)^{2rn}}\sum_{\substack{ \bchi\in X_{2r,n}\\  \forall (\mu,j),\exists!(\nu,i)\neq (\mu,j): \chi_{\mu,j} =\overline{\chi_{\nu,i}} }}  \sum_{\substack{ \bm \gamma \in \Gamma(\bchi)}} \widehat \eta(\bm \gamma) \Delta_{2r}(\bm\sigma_{\bm\gamma};\Phi,T)+O_\Phi\Big( \frac{( K_{\delta,\eta}\log q)^{sn}}{T\phi(q)^s} \Big)
\\ & \geq\frac 1{\phi(q)^{2rn}}\sum_{\pi \in F_{2r,n}}
\!\!\!\!\!\!\!\!\!\!\!\!\sum_{\substack{\bchi\in X_{2r,n}\\   \pi(\bchi ) =\overline{\bchi} \\
\chi_{\mu,j} =\overline{\chi_{\nu,i}} \Rightarrow (\mu,j)\in\{ (\nu,i),\pi(\nu,i)\} }}  \sum_{\substack{ \bm \gamma \in \Gamma(\bchi) : \\  
\forall \mu,j, \,\,\,  \gamma_{\chi_{\mu,j}} +\gamma_{\chi_{\pi(\mu,j)} }=0 \\  \forall \mu, \,\,\,\sum_{j=1}^n \gamma_{\chi_{\mu,j}} \neq 0}}\!\!\! \widehat \eta(\bm \gamma)+O_\Phi\Big( \frac{( K_{\delta,\eta}\log q)^{sn}}{T\phi(q)^s} \Big).
 \end{split}\label{equation lemma restriction sum over zeros }
\end{equation}
Now, we may bound the inner sum as follows: 
$$\sum_{\substack{ \bm \gamma \in \Gamma(\bchi) : \\  
\forall \mu,j, \,\,\,  \gamma_{\chi_{\mu,j}} +\gamma_{\chi_{\pi(\mu,j)} }=0 \\  \forall \mu, \,\,\,\sum_{j=1}^n \gamma_{\chi_{\mu,j}} \neq 0}} \widehat \eta(\bm \gamma) \geq \sums_{\substack{ \bm \gamma \in \Gamma(\bchi) : \\  
\forall \mu,j, \,\,\,  \gamma_{\chi_{\mu,j}} +\gamma_{\chi_{\pi(\mu,j)} }=0 \\  \forall \mu, \,\,\,\sum_{j=1}^n \gamma_{\chi_{\mu,j}} \neq 0}} \widehat \eta(\bm \gamma) \prod_{\substack{ 1\leq \mu \leq 2r \\ 1\leq j\leq n }} \big(\ord_{s=\frac 12+i\gamma_{\chi_{\mu,j}}}L(s,\chi_{\mu,j})\big)^{\frac 12}. $$
Applying Lemma~\ref{lemma sum over characters}, we can remove the conditions $\sum_{j=1}^n \gamma_{\chi_{\mu,j}} \neq 0$ at the cost of an admissible error term. The resulting sum is handled with~\eqref{equation average b(chi) F_2r}, and the claimed estimate follows.
\end{proof}

The final step is to compute $|F_{2r,n}|$.

\begin{lemma}
\label{lemma combinatorics even s}
For $r \in \mathbb N$ and $n\geq 2$, we have the formula
$$ |F_{2r,n}| =\mu_{2r}
\nu_n^r
,$$ 
where $F_{2r,n}$ was defined in Definition~\ref{definition F} and $\nu_n$ in \eqref{defKn}.
\end{lemma}

\begin{proof}
We can view the elements of $F_{2r,n}$ as pairs $ \{(\mu,j), (\nu,i)\}$, where $1\leq \mu,\nu\leq 2r$, $1\leq j,i\leq n$, and $\pi(\mu,j)= (\nu,i)$. 
Clearly, 
$ |F_{2r,n}| = \mu_{2r} |G_{2r,n}|, $ where $G_{2r,n}$ is the set of involutions $\pi: \{ 1,\dots,2r\} \times \{ 1,\dots, n\} \rightarrow \{ 1,\dots,2r\} \times \{ 1,\dots , n\}  $ having no fixed point and such that (recall~\eqref{equation definition Jmunu})
$$ \{  (\mu,\nu) \in [1,2r]^2 : 
  J_{\mu,\nu}(\pi)\neq \varnothing \} = \{ (2u,2u-1),(2u-1,2u): 1\leq u \leq r \}, $$
  and $|J_{2u,2u-1}| \geq 2 $.
For $\pi \in G_{2r,n}$ and $1\leq u \leq r$, we let  $  k_u(\pi):= |J_{2u-1,2u-1}(\pi) |$ (recall~\eqref{equation definition Jmunu}),
and $\bm{k}(\pi) = (k_u(\pi))_{1\leq u\leq r}$. We recall~\eqref{equation conditionsurk}; in particular, $\bm \ell(\pi):= \bm{k}(\pi)/2 \in \mathbb (\mathbb Z_{\geq 0})^r $. For any given $\bm \ell\in \{0,\dots \lceil \frac {n-2}2\rceil\}^r$, there are exactly $\prod_{1\leq u\leq r}\binom{n}{2 \ell_u}^2 (n-2 \ell_u)!\big( \frac{(2\ell_u)!}{2^{\ell_u}\ell_u!} \big)^2 $ involutions $\pi\in G_{2r,n}$ for which $\bm \ell(\pi)=\bm \ell$. The claim follows.
\end{proof}

  We now work with odd values $s=2r+1$. The combinatorics are more complicated here, and we need to change the definition of the set $F_{2r,n}$ accordingly.
  \begin{definition}
  \label{definition F odd s}
For  $r \in \mathbb N$ and $n\geq 2$, we define   $F_{2r+1,n}$ to be the set of involutions $$\pi: \{ 1,\dots,2r+1\} \times \{ 1,\dots , n\} \rightarrow \{ 1,\dots,2r+1\} \times \{ 1,\dots , n\}  $$ having no fixed point and such that there exists a set $N_{\pi} \subset \{ 1,\dots,2r+1\} $ of cardinality $3$ for which for each fixed $ \mu \in S_\pi:=\{1,\dots,2r+1\} \smallsetminus N_\pi,   $ there exists a unique $ 1\leq \nu \leq  2r+1 $, $\nu \neq \mu$ such that $  J_{\mu,\nu}\neq \varnothing$, and moreover $|\{( \mu,\nu) :  \mu, \nu\in N_\pi, \mu < \nu, J_{\mu,\nu}(\pi)\neq \varnothing\}| \in \{2,3\} $. In other words, $\pi$ is such that there is a set of "rows" $S_\pi$ that are paired two by two through $\pi$, and its complement $N_\pi$ is such that its rows are either paired in a $3$-cycle, or for which there is one line which is paired with the other two. Moreover, we require that for each $1\leq \mu \leq 2r+1 $,
$$  \sum_{\substack{ 1\leq \nu  \leq 2r+1 \\ \nu \neq \mu}} |J_{\mu,\nu}| \geq 2.$$
  \end{definition}

Note that the relations~\eqref{equation conditionsurk} (with $2r+1$ in place of $2r$) hold for any $\pi  \in F_{2r+1,n}$. We first prove an analogue of Lemma~\ref{lemma sum over characters}. 

\begin{lemma}
Let  $r \in \mathbb N$, $n\geq 2$ and $q\geq 3$, and let $\pi \in F_{2r+1,n}$. We have the estimate
\label{lemma sum over characters odd}
\begin{equation}
C_{2r+1,n}(q;\pi):=\sum_{\substack{\bchi\in X_{2r+1,n}\\  \pi(\bchi ) =\overline{\bchi}\\  \chi_{\mu,j} =\overline{\chi_{\nu,i}} \Rightarrow (\mu,j)\in \{ (\nu,i),\pi(\nu,i)\}    } } \!\!\!\!\!\! 1 =  \phi(q)^{(n-1)\frac{2r+1}2 -\frac 12} \Big(1+ O\Big(   \frac {n^2r^2}{\phi(q)}\Big) \Big).
\label{equation lemma sum characters odd}
\end{equation}
If moreover $q$ is large enough in terms of $\delta$ and $\eta$, then in the ranges $n\leq \log q/\log_2q $ and $r\leq \phi(q)^{\frac 12}/n$ we have the estimate
\begin{align*}
\sum_{\substack{\bchi \in X_{2r+1,n} \\ \pi(\bchi)=\bchi \\ \chi_{\mu,j} = \overline{\chi_{\nu,i}} \Rightarrow (\mu,j) \in \{ (\nu,i),\pi(\nu,i) \} }} \prod_{\substack{ 1\leq \mu \leq 2r+1 \\ 1\leq j \leq n }}b(\chi_{\mu,j};\widehat \eta^2)^{\frac 12}  = \phi(q)^{(n-1)\frac{2r+1}2-\frac 12}&(\alpha(\widehat \eta^2)\log q)^{n\frac{2r+1}2} \times \\  &\Big( 1+O \Big( n \frac {\log_2 q}{ \log q } \Big) \Big)^{r+1}
\end{align*}
Finally, the analogue of~\eqref{equation average b(chi) with one zero removed} holds with $\phi(q)^{(n-1)\frac{2r+1}2-\frac 12}(\alpha(\widehat \eta^2)\log q)^{n\frac{2r+1}2-1}$ in place of $\phi(q)^{(n-1)r}(\alpha(\widehat \eta^2)\log q)^{nr-1}$.
\end{lemma}

\begin{proof}
Fix  $\pi \in F_{2r+1,n}$. In the sum defining $C_{2r+1,n}(q;\pi)$ and by definition of $F_{2r+1,n}$, we can change the indexing in the variable $\mu$ in such a way that 
$ J_{\mu,\nu}(\pi) 
\neq \varnothing $ if either  $\{ \mu,\nu\} \subset \{ 1,2,3 \}$ or $\{ \mu,\nu\} = \{ 2u,2u+1 \}$ for some $2\leq u\leq r$, and $J_{\mu,\nu}(\pi)=\varnothing $ otherwise.
For $\bchi\in X_{2r+1,n}$, let us write $\bchi=(\bchi_1,\bchi_2)$, where
$ 
\bchi_1=(\chi_{\mu,j})_{\substack{1\leq \mu \leq 3\\ 1\leq j\leq n}}\in X_{3,n} $ and $\bchi_2 =(\chi_{\mu,j})_{\substack{4\leq \mu \leq 2r+1\\ 1\leq j\leq n}}\in X_{2r-2,n} .$ With this notation, we can define $\pi_1$ and $\pi_2$ in such a way that 
$\pi(\bchi)=(\pi_1(\bchi_1),\pi_2(\bchi_2)).$ Then the condition $ \pi(\bchi)=\overline\bchi$ is equivalent to 
$ \pi_1(\bchi_1)=\overline{\bchi_1}$ and $ \pi_2(\bchi_2)=\overline{\bchi_2}$.
This is possible since $\pi_1(\bchi_1)$ does not depend on $\bchi_2$ and $\pi_2(\bchi_2)$ does not depend on $\bchi_1.$
Consequently, we may write
$$ C_{2r+1,n}(q;\pi) =  \sum_{\substack{ 
\bchi_1 \in X_{3,n} 
\\  \pi_1(\bchi_1 ) =\overline{\bchi_1}\\ \chi_{\mu,j} =\overline{\chi_{\nu,i}} \Rightarrow (\mu,j)\in \{ (\nu,i),\pi(\nu,i)\}  
} } 
\sum_{\substack{ 
\bchi_2 \in X_{2r-2,n} \\   \pi_2(\bchi_2 ) =\overline{\bchi_2}
\\  \chi_{\mu,j} \notin \{ \overline{\chi_{\nu,i}} |\nu \leq 2r+1, j\leq n, (\nu,i)\notin \{ (\mu,j),\pi(\mu,j)\}\}    } }  1. $$
The innermost sum can be computed using the same arguments as in Lemma~\ref{lemma sum over characters}, and we deduce that it is sufficient to prove~\eqref{equation lemma sum characters odd} for $r=1$. Recalling Definition~\ref{definition F odd s}, this means that $N_\pi=\{ 1,2,3\}$ and $S_\pi = \varnothing$. 

We first treat the case where 
$|\{  (\mu,\nu) : \mu<\nu,\, J_{\mu,\nu}(\pi) \neq \varnothing\} |=2; $ by reordering $\mu$ in the array of characters $\bchi=(\chi_{\mu,j})_{\substack{1\leq \mu\leq 3\\ 1\leq j\leq n}}$, we may assume that $ J_{1,2}(\pi),J_{1,3}(\pi) \neq \varnothing$, and $J_{2,3}(\pi)=\varnothing$. Under the condition $\pi(\bchi )= \overline{\bchi }$,
the relations $\prod_{j =1}^{n}\chi_{\mu,j} = \chi_{0,q}$ $ (\mu=1,2,3) $ are equivalent to
$$ \Theta_{\mu,1}(\pi):=\prod_{j \in J_{\mu,1}(\pi)}\chi_{\mu,j} = \chi_{0,q}\qquad  (\mu=2,3) $$
(since the remaining relation is automatically satisfied). 
We reformulate the relations for $j\in J_{\mu,\mu}$ which follow from the equality $\pi (\bchi)=\overline{\bchi}$, with $\bchi\in X_{3,n}$. To do so, we denote $\bchi_{\mu}:=(\chi_{\mu,j})_{  j\in J_{\mu,\mu}}$ and $\pi_{\mu} (\bchi_{\mu}):=(\chi_{\pi(\mu,j)})_{  j\in J_{\mu,\mu}}$.
Then, the condition
$\pi (\bchi)=\overline{\bchi}$ implies that 
$\pi_{\mu} (\bchi_{\mu})=\overline{\bchi_{\mu}}$ for  $1\leq \mu\leq 3$. This condition also implies that the values of $\chi_{1,j}$ with $j\notin J_{1,1} $ are determined by those of $\chi_{\mu,j}$ with $\mu\in \{ 2,3\}$. 
We deduce that 
\begin{align*}
 C_{3,n}(q;\pi)&= \sum_{\substack{ \chi_{2,1}, \dots ,\chi_{2,n} ,\chi_{3,1} ,\dots ,\chi_{3,n} \neq \chi_{0,q}\\  \Theta_{2,1}(\pi)=\Theta_{3,1}(\pi)=\chi_{0,q} \\ \pi_{\mu} (\bchi_{\mu})=\overline{\bchi_{\mu}}
 \,(\mu=2,3)
 \\ \chi_{\mu,j} =\overline{\chi_{\nu,i}} \Rightarrow (\mu,j)\in \{ (\nu,i),\pi(\nu,i)\} } } \sum_{\substack{ \chi_{1,j} \neq \chi_{0,q}  \,(j\in J_{1,1})\\  
 \pi_1 (\bchi_{1})=\overline{\bchi_{1}}
 \\ \chi_{1,j}  \notin \{\overline{\chi_{\nu,i}} |  i\leq n, 1\leq  \nu\leq 3, (\nu,i)\notin\{ (1,j),\pi(1,j) \} \} }  }1
 \\ &= \phi(q)^{\frac{|J_{2,2}|}2 + |J_{2,1}|-1} \cdot \phi(q)^{\frac{|J_{3,3}|}2 + |J_{3,1}|-1} \cdot \phi(q)^{\frac{|J_{1,1}|}{2}} \Big(1+O\Big( \frac {n^2}{\phi(q)} \Big)\Big).
\end{align*}
Applying~\eqref{equation conditionsurk}, we deduce that the exponent of $\phi(q)$ is equal to
$$ \frac{|J_{2,2}|}2 + |J_{2,1}|-1+\frac{|J_{3,3}|}2 + |J_{3,1}|-1+\frac{|J_{1,1}|}{2}=n+\frac{|J_{2,2}|+|J_{3,3}|-|J_{1,1}|}2-2. $$
By our hypothesis on $\pi$ we have that $ n- |J_{1,1}| = n-|J_{2,2}|+n-|J_{3,3}|$, which concludes the proof in this case.

We now move to the case where 
$|\{  (\mu,\nu) :  \mu<\nu,\,   J_{\mu,\nu}(\pi) \neq \varnothing\} |=3.$ Given an array of characters $\bchi=( \chi_{\mu,j})_{\mu,j}$ such that $\pi(\bchi)= \overline{\bchi}$  and keeping the notation
$$\Theta_{\mu,\nu}(\pi):= \prod_{j\in J_{\mu,\nu}} \chi_{\mu,j} ,$$ we can rewrite the conditions $ \prod_{j=1}^n \chi_{\mu,j} =\chi_{0,q} $ $ (1\leq \mu \leq 3)$ as
$$  \Theta_{1,2}(\pi)= \overline{\Theta_{1,3}(\pi)} ;\qquad \Theta_{1,3}(\pi)= \overline{\Theta_{2,3}(\pi)}.  $$
 This new set of conditions is now minimal, since $J_{1,3}\cap J_{2,3}=\varnothing$. We can now write 
 \begin{align*}
 C_{3,n}(q;\pi)&= \sum_{\substack{ \chi_{1,1}, \dots, \chi_{1,n} \neq \chi_{0,q} \\  \Theta_{1,2}(\pi)=\overline{\Theta_{1,3}(\pi)} \\  \pi_{1} (\bchi_{1})=\overline{\bchi_{1}} \\ \chi_{\mu,j} =\overline{\chi_{\nu,i}} \Rightarrow (\mu,j)\in \{ (\nu,i),\pi(\nu,i)\} } }\sum_{\substack{ \chi_{2,j}  \neq \chi_{0,q} \, (j\in J_{2,2}\cup J_{2,3})\\  \Theta_{2,3}(\pi)=\overline{\Theta_{1,3}(\pi)} \\  \pi_{2} (\bchi_{2})=\overline{\bchi_{2}}  \\ \chi_{\mu,j} =\overline{\chi_{\nu,i}} \Rightarrow (\mu,j)\in \{ (\nu,i),\pi(\nu,i)\} } }\sum_{\substack{  \chi_{3,j} \neq \chi_{0,q}\, (j\in  J_{3,3} ) \\  \pi_3 (\bchi_3)=\overline{\bchi_3} \\ \chi_{\mu,j} =\overline{\chi_{\nu,i}} \Rightarrow (\mu,j)\in \{ (\nu,i),\pi(\nu,i)\} } }1
 \\ &= \phi(q)^{\frac{|J_{1,1}|}2 + |J_{1,2}|+|J_{1,3}|-1} \cdot \phi(q)^{\frac{|J_{2,2}|}2 + |J_{2,3}|-1} \cdot \phi(q)^{\frac{|J_{3,3}|}{2}} \Big(1+O\Big( \frac {n^2}{\phi(q)} \Big)\Big).
\end{align*}
The exponent of $\phi(q)$ is now 
$$ \frac{|J_{1,1}|+|J_{2,2}|+|J_{3,3}|}{2}+ |J_{1,2}|+|J_{1,3}|+|J_{2,3}|-2, $$
which is seen to be equal to $\frac 32 n-2 $ by adding the equations $n=|J_{\mu,1}|+|J_{\mu,2}|+|J_{\mu,3}|$ for $1\leq \mu \leq 3$ and using the relation $|J_{\mu,\nu}|=|J_{\nu,\mu}|$.

The second and third claimed estimates follow as in the proof of Lemma~\ref{lemma sum over characters}.
 \end{proof}

\begin{lemma}
\label{lemma lower bound centered moments odd s}
Assume GRH$_{\widehat \eta}$, and let  $\delta>0$, $\eta \in \S_\delta$, $\Phi\in \U,$ and  $r \in \mathbb N$, $n\geq 2$, with $n$ even. For $T\geq 1 $, $q$ large enough in terms of $\delta$ and $\eta$ and in the ranges $n\leq \log q/\log_2 q$ and $r\leq \log q$, we have the lower bound
\begin{multline}
  \cV_{2r+1,n}(T,q;\eta,\Phi) \geq   \frac {1 }{\phi(q)^{ \frac 12} } 
  \Big(\frac{V_n(q;\eta)}{\nu_n}\Big)^{\frac{2r+1}2 }
  |F_{2r+1,n}|  \Big( \Big( 1+O_{\delta,\eta}\Big( n \frac{\log_2 q}{\log q}\Big)\Big)^r+O_{\delta,\eta}\Big( \frac r{\log q}\Big)\Big)
\\ +O_\Phi\Big( \frac{(K_{\delta,\eta}\log q)^{(2r+1)n}}{T\phi(q)^{2r+1}} \Big),
\label{equation lemma lower bound centered moments odd s}
\end{multline}
Moreover, for all $q\geq 3$ and all $n\geq 1$ we have the weaker bound
 $$ (-1)^n\cV_{2r+1,n}(T,q;\eta,\Phi) \geq O_\Phi\Big( \frac{(K_{\delta,\eta}\log q)^{(2r+1)n}}{T\phi(q)^{2r+1}} \Big).$$
Here, $K_{\delta,\eta}>0$ is a constant.
\end{lemma}

\begin{proof}
The estimate for odd values of $n$ is obtained by applying Lemma~\ref{lemma explicit formula higher moments of moments} and discarding the sum over zeros using positivity.
For even values of $n$, we bound the sum over zeros in the same way as in Lemmas~\ref{lemma sum over characters} and~\ref{lemma lower bound centered moments even s}, using this time Lemma~\ref{lemma sum over characters odd}. 
\end{proof}

In order to state the last combinatorial lemma of this section, we define the constants
 \begin{equation}\nu_n':=\sum_{\substack{ ({\bm \ell},{\bm k}) \in  \mathbb Z_{\geq 0}^3\times \mathbb Z_{\geq 2}^2\\  n=
k_{1}  +2\ell_1 
\\ n=
k_{2}  +2\ell_2
\\
n=
k_{1}+k_{2} +2\ell_3
}} \frac{n!^3 }{k_{1}!k_{2}!2^{\ell_{1 }+\ell_{2 }+\ell_{3}} \ell_{1}!\ell_{2}!\ell_{3}!};
\label{defKn'}
 \end{equation}
 \begin{equation}\label{defKn''}\nu_n'':=
\frac 13\sum_{\substack{ ({\bm \ell},{\bm k}) \in  \mathbb Z_{\geq 0}^3\times \mathbb Z_{\geq 1}^3\\  n=
k_{i}+k_{j} +2\ell_k
\\ \{ i,j,k\}=\{ 1,2,3\}}}\frac{n!^3 }{k_{1}!k_{2}!k_{3}!2^{\ell_{1 }+\ell_{2 }+\ell_{3}} \ell_{1}!\ell_{2}!\ell_{3}!}.
 \end{equation}

\begin{lemma}
\label{lemma combinatorics odd s}
For  $r \in \mathbb N$ and $n\in 2\mathbb N$, we have the formula
$$
 |F_{2r+1,n}| =
  \frac{(2r+1)!}{2^r(r-1)!}  \nu_n^{r-1} (\nu_n'+\nu_n'').
$$
\end{lemma}

\begin{proof}
Clearly, 
$$ |F_{2r+1,n}| = \frac{(2r+1)!}{2^{r}(r-1)!} \Big(|A_{2r+1,n}|+\frac 13|B_{2r+1,n}|\Big), $$ where $A_{2r+1,n},B_{2r+1,n}$  are sets of involutions $\pi: \{ 1,\dots,2r+1\} \times \{ 1,\dots , n\} \rightarrow \{ 1,\dots,2r+1\} \times \{ 1,\dots , n\}  $ having no fixed point, for which for each $1\leq \mu \leq 2r+1$,
\begin{equation}
 \sum_{\substack{1\leq \nu\leq 2r+1\\ \nu\neq \mu}}|J_{\mu,\nu}(\pi)| \geq 2,
 \label{equation condition geq 2 Jmunu}
\end{equation}
such that for $\pi\in A_{2r+1,n}$,
$$
 \big\{  \{\mu,\nu\}  : 
  J_{\mu,\nu}(\pi) \neq \varnothing \big\}  = \big\{\{ 1,2\} ,\{1,3\} \big\}\cup\big\{ \{2u,2u+1\}: 2\leq u \leq r \big\},
 $$
 and for $\pi \in B_{2r+1,n}$,
$$
 \big\{  \{\mu,\nu\}  : 
  J_{\mu,\nu}(\pi) \neq \varnothing \big\}  = \big\{\{ 1,2\} ,\{1,3\},\{2,3\} \big\}\cup\big\{ \{2u,2u+1\}: 2\leq u \leq r \big\}.
 $$
We let 
$$  k_u(\pi):=\begin{cases} 
| J_{u,u}(\pi) |&\text{ if } 1\leq u\leq 3,\\
 | J_{2u-4,2u-4}(\pi) | &\text{ if } 4\leq u \leq r+2,
 \end{cases}$$
and ${\bm k}(\pi) = (k_u(\pi))_{1\leq u\leq r+2}$. Note that $\bm \ell(\pi):= {\bm k}(\pi)/2 \in \mathbb (\mathbb Z_{\geq 0})^{r+2} $. 
 
For $\pi \in A_{2r+1,n}$ and $1\leq u \leq r+1$, 
we observe that 
$$n-2\ell_1(\pi)=|J_{1,2}(\pi)|+|J_{1,3}(\pi)|,\quad
n-2\ell_2(\pi)=|J_{1,2}(\pi)| ,\quad  n-2\ell_3(\pi)= |J_{1,3}(\pi)|, $$ and hence $n+2\ell_1(\pi)=2\ell_2(\pi)+2\ell_3(\pi)$.
For any given $\bm \ell\in \{0,\dots,  \frac {n}2-1  \}^{r+2}$ such that $n+2\ell_1=2\ell_2+2\ell_3$, there are exactly 
$$\binom{n-2\ell_1}{n-2\ell_2} (n-2\ell_2)!(n-2\ell_3)! \prod_{1\leq u\leq 3} \binom{n}{2 \ell_u} \frac{(2\ell_u)!}{2^{\ell_u}\ell_u!} \prod_{4\leq u\leq r+2}\binom{n}{2 \ell_u}^2 (n-2 \ell_u)!\Big( \frac{(2\ell_u)!}{2^{\ell_u}\ell_u!} \Big)^2 $$ involutions $\pi\in A_{2r+1,n}$ for which $\bm \ell(\pi)=\bm \ell$. 

As for $B_{2r+1,n}$, we let
$  m_1(\pi):= | J_{1,2}(\pi) |, m_2(\pi):= | J_{2,3}(\pi) |$ and $m_3(\pi):= | J_{3,1}(\pi) |$. For any given $\bm \ell\in \{0,\dots, \frac {n}2-1 \}^{r+2}$ and ${\bm m} \in \{1,\dots ,n-1 \}^{3}$ such that $ m_u+m_{ u+2 }+2\ell_u=n$ for $1\leq u\leq 3$ (with the convention that $m_4=m_1, m_5=m_2 $), there are  
$$\prod_{u=1}^3 \binom{n-2\ell_u}{m_u} m_u! \binom{n}{2\ell_u} \frac{(2\ell_u)!}{2^{\ell_u}\ell_u!}   \prod_{4\leq u\leq r+2}\binom{n}{2 \ell_u}^2 (n-2 \ell_u)!\Big( \frac{(2\ell_u)!}{2^{\ell_u}\ell_u!} \Big)^2 $$
involutions $\pi\in B_{2r+1,n}$ (note that $ m_u(\pi) \geq 1 $ implies the condition~\eqref{equation condition geq 2 Jmunu}) for which $\bm \ell(\pi)=\bm \ell$ and ${\bm m}(\pi)=\bm m$. The claim follows from adding all the different possible outcomes. 
 \end{proof}
\goodbreak

We are now ready to prove the main results of this section. 
\begin{proof}[Proof of Theorem~\ref{thmomentscentres2rT}]
 The lower bounds for~$\cV_{s,n}(T,q;\eta,\Phi)$ follow from Lemmas~\ref{lemma lower bound centered moments even s} and~\ref{lemma lower bound centered moments odd s}, combined with Lemmas~\ref{lemma combinatorics even s} and~\ref{lemma combinatorics odd s}. 

\end{proof}
We also establish the lower bounds in Theorem~\ref{thmomentscentres2r}.
\begin{proof}[Proof of Theorem~\ref{thmomentscentres2r}, first part]
 First note that under GRH, Lemma~\ref{lemma explicit formula higher moments of moments} implies that
\begin{align*}
    \V_{s,n}(q;\eta) & = \lim_{T\rightarrow \infty}   \cV_{s,n}(T,q;\eta,1_{[-1,1]}) \\
    &= \frac {(-1)^{sn}}{\phi(q)^{sn}} 
\sum_{\substack{\bchi\in X_{s,n}}}  \sum_{\substack{ \bm \gamma \in \Gamma(\bchi) }} \widehat \eta(\bm \gamma) \Delta_{s}(\bm\sigma_{\bm\gamma}),  
\end{align*}
by dominated convergence. For $s=2r$, we argue as in the proof of Lemma~\ref{lemma lower bound centered moments even s} and use positivity of $\widehat \eta$ to deduce that 
$$  \V_{2r,n}(q;\eta) \geq  \Big( \Big( 1+O_{\delta,\eta}\Big( n \frac{\log_2 q}{\log q}\Big)\Big)^r+O_{\delta,\eta}\Big( \frac r{\log q}\Big)\Big)\frac{V_n(q;\eta)^r}{\nu_n^r}  |F_{2r,n}|. $$
Thanks to Lemma~\ref{lemma combinatorics even s}, this implies~\eqref{inegmoment2r}. The proof of~\eqref{inegmoment2r+1} is similar, following this time Lemmas~\ref{lemma lower bound centered moments odd s} and~\ref{lemma combinatorics odd s}.
\end{proof}

\section{Unconditional applications of moments of moments}
\label{section unconditional applications}

The goal of this section is to prove Theorems~\ref{thunconditional n=2} and~\ref{thunconditional n>2}. We will seperate the proof into two parts. On one hand, we will show that it follows from GRH$_{\widehat \eta}$, and, on the other, we will show that it also follows from its negation. 

We first need to adapt a result of Hooley~\cite[Theorem 1]{H76}. 
\begin{lemma} Let  $\delta>0$ and $\eta\in\S_\delta$, and assume GRH$_{\widehat\eta}$. Then, in the range $q\geq 3$ and $\log q\leq S$, we have the bound
$$ \frac 1S \int_{0}^S M_{2}(\e^t,q;\eta) \d t \ll_{\delta,\eta} \frac{\log q}{\phi(q)}. $$
\label{lemma hooley upper bound}
\end{lemma}
\begin{proof}
We will show that in the range $ 2\log q \leq T$ we have the bound
$$  \cV_{1,2}(T,q;\eta,\Phi_0) = \frac 2T \int_{0}^\infty \Phi_0\Big(\frac t{T}\Big) M_{2}(\e^t,q;\eta) \d t \ll_{\delta,\eta} \frac{\log q}{\phi(q)}, $$
where \begin{equation}
\Phi_0(x):= 
\begin{cases}
1-|x| & \text{ if } |x|\leq 1, \\
0 &\text{ otherwise,}
\end{cases}
\label{equation definition triangle}
\end{equation} 
and $\widehat \Phi_0(\xi) = (\sin(\pi \xi)/\pi \xi)^2\geq 0$.
This clearly implies our claim. 

By Lemma~\ref{lemma explicit formula higher moments of moments} we have the estimate
 \begin{multline*}
      \cV_{1,2}(T,q;\eta,\Phi_0) = \frac {2}{\phi(q)^{2}} 
\sum_{\substack{ \chi \neq \chi_{0,q} }} \sum_{\substack{ \gamma_{\chi},\lambda_\chi }} \widehat \eta\Big(\frac{\gamma_\chi}{2\pi}\Big)\widehat \eta\Big( \frac{\lambda_\chi}{2\pi}\Big) \widehat \Phi_0\Big(\frac T{2\pi}(\gamma_\chi-\lambda_\chi)\Big)+O\Big( \frac{( K_{\delta,\eta}\log q)^{2}}{T\phi(q)} \Big),
\end{multline*} 
where $(\gamma_\chi,\lambda_\chi)$ runs over the pairs of imaginary parts of zeros of $L(s,\chi)$ on the critical line.
We will consider three distinct ranges of $ |\gamma_\chi-\lambda_\chi|$. 

Firstly, the bound~\eqref{equation pointwise b(chi) } implies that
$$\sum_{\substack{\gamma_\chi, \lambda_\chi\\ |\gamma_{\chi}-\lambda_\chi|\geq 1}} \widehat \eta\Big(\frac{\gamma_\chi}{2\pi}\Big)\widehat \eta\Big( \frac{\lambda_\chi}{2\pi}\Big) \widehat \Phi_0\Big(\frac T{2\pi} (\gamma_\chi-\lambda_\chi)\Big)\ll_\eta \frac {(\log q)^2}{T^2}, $$
whose contribution to  $\cV_{1,2}(T,q;\eta,\Phi_0)$ is $\ll (\log q)^2/(T^2\phi(q))$.

Secondly, in the range $ w\leq  |\gamma_\chi-\lambda_\chi| \leq 2w$  (which we denote by $|\gamma_\chi-\lambda_\chi|\sim w$) with  $ \frac 1{3\log q}\leq w \leq 1$,
we will use the fact that for $ |t|\in [0,\frac 9{10}]$, $ \widehat \Phi_{0}(t) \geq \frac 1{10}$. We deduce that
\begin{multline*}
\sum_{\chi\neq\chi_{0,q}}\sum_{\substack{\gamma_\chi,\lambda_{\chi} \\ \notag |\gamma_\chi-\lambda_{\chi}| \sim  w} }\widehat \eta\Big( \frac{\gamma_\chi}{2\pi}\Big)\widehat \eta\Big( \frac{\lambda_\chi}{2\pi}\Big) \widehat \Phi_0 \Big(\frac T{2\pi} (\gamma_\chi-\lambda_\chi)\Big)   \ll  \frac 1{(wT)^2}\sum_{\chi\neq\chi_{0,q}}\sum_{\substack{\gamma_\chi,\lambda_{\chi} \\ |\gamma_\chi-\lambda_{\chi}| \sim  w} }\widehat \eta\Big( \frac{\gamma_\chi}{2\pi}\Big) \widehat \eta\Big( \frac{\lambda_\chi}{2\pi}\Big) \\
\ll \frac{1}{(wT)^2} \sum_{\chi\neq\chi_{0,q}}\sum_{\substack{\gamma_\chi,\lambda_{\chi}}}\widehat \eta\Big( \frac{\gamma_\chi}{2\pi}\Big)\widehat \eta\Big( \frac{\lambda_\chi}{2\pi}\Big) \widehat \Phi_0\Big(\frac{\gamma_\chi-\lambda_\chi }{ 5\pi w}\Big)\notag 
\ll \frac{ \phi(q) (\log q)^2}{wT^2},
\label{equation bound with w}
\end{multline*}
by Proposition~\ref{proposition pair correlation}. As a result, decomposing the sum over $\gamma_\chi$ into dyadic intervals we obtain the bound
$$\frac 2{\phi(q)^2}\sum_{\chi\neq\chi_{0,q}}\sum_{\substack{\gamma_\chi,\lambda_{\chi} \\ \frac 1{3\log q} < |\gamma_\chi-\lambda_{\chi}| \leq 1} }\widehat \eta\Big( \frac{\gamma_\chi}{2\pi}\Big) \widehat \eta\Big( \frac{\lambda_\chi}{2\pi}\Big) \widehat \Phi_0\Big (\frac T{2\pi}(\gamma_\chi-\gamma_\psi)\Big) \ll \frac{(\log q)^2}{\phi(q)T^2}.  $$

Finally, we are left to bound
\begin{align*}
\frac 2{\phi(q)^2}\sum_{\chi\neq\chi_{0,q}}\sum_{\substack{\gamma_\chi,\lambda_{\chi} \\ |\gamma_\chi-\lambda_{\chi}|  \leq \frac 1{3\log q}} } & \widehat \eta\Big( \frac{\gamma_\chi}{2\pi}\Big) \widehat \eta\Big( \frac{\lambda_\chi}{2\pi}\Big)\widehat \Phi_0\Big(\frac T{2\pi} (\gamma_\chi-\gamma_\psi)\Big) \\
&\leq \frac 2{\phi(q)^2}\sum_{\chi\neq\chi_{0,q}}\sum_{\substack{\gamma_\chi,\lambda_{\chi} \\ |\gamma_\chi-\lambda_{\chi}|  \leq \frac 1{3\log q}} }\widehat \eta\Big( \frac{\gamma_\chi}{2\pi}\Big) \widehat \eta\Big( \frac{\lambda_\chi}{2\pi}\Big)
\\
&\ll \frac {1}{\phi(q)^2}\sum_{\chi\neq\chi_{0,q}}\sum_{\substack{\gamma_\chi,\lambda_{\chi}}}\widehat \eta\Big( \frac{\gamma_\chi}{2\pi}\Big) \widehat \eta\Big( \frac{\lambda_\chi}{2\pi}\Big)\widehat \Phi_0\Big(\frac{\log q}{4\pi }\cdot(\gamma_\chi-\lambda_{\chi})\Big).
\end{align*}
Applying Proposition~\ref{proposition pair correlation}, we see that this term is $\ll (\log q)/\phi(q)$.
The proof is finished. 
\end{proof}

We now show that $ M_n(\e^t,q;\eta )$ takes large values by using the bounds obtained in Section~\ref{section mean of first moment}.
\begin{lemma} Let $m\geq 1,$  $\delta>0$, $\eta\in \S_\delta$, and assume that GRH$_{\widehat \eta}$ holds. Fix $\eps>0$ and let $C_{\eps,\delta,\eta}$ be a large enough constant, and let $f: \mathbb R_{\geq 0}\mathbb \rightarrow \mathbb R_{\geq 0}$ be any function such that for all $t\geq 3$, $f(t)\geq 1 $. Then, for all $q \geq (C_{\eps,\delta,\eta})^m$, there exists an associated value $x_m(q)$ for which $\log(x_m(q))\in [f(q),f(q) \phi(q)^{m-1} (C_{\eps,\delta,\eta}\log q)^m \mu_{2m}^{-1}]$ and such that
\begin{equation}\label{OmegaM2m}
M_{2m}(x_{m}(q),q;\eta) \geq ( 1-\eps) \mu_{2m}\frac{(\alpha(\widehat \eta^2)\log q)^m}{\phi(q)^m}.  
\end{equation}

\label{lemma raw oscillations}
\end{lemma}
\begin{proof}
We can clearly assume that $\alpha(\widehat \eta^2)=1.$
We apply Proposition~\ref{proposition first moment of moments} with $\Phi=\Phi_0$ where $\Phi_0$ is defined in \eqref{equation definition triangle}.
We deduce that for $C_{\eps,\delta,\eta}$ large enough, $\phi(q)\geq (C_{\eps,\delta,\eta})^{m}$ and $ T\geq  \frac{(C_{\eps,\delta,\eta}\log q)^m  }{\mu_{2m}}  \phi(q)^{m-1}$,

\begin{equation}
 \frac 1T\int_{0}^{\infty} \Phi_0\Big(\frac tT\Big) M_{2m}(\e^{t},q;\eta)\d t \geq \Big(\frac 12-\frac \eps3\Big) \mu_{2m}\frac{(\log q)^m}{\phi(q)^m}.  
 \label{equation lower bound first moment}
\end{equation}
We claim that there exists $ t_m(q) \in [ \frac {\mu_{2m}T}{ \phi(q)^{m-1}( C_{\eps,\delta,\eta} \log q)^m} ,T]$ (where we might need to enlarge the value of the constant $C_{\eps,\delta,\eta}>0$)
such that 
$$\Phi_0\Big(\frac {t_{m}(q)}T\Big)M_{2m}(\e^{t_{m}(q)},q;\eta) >(1- \eps )\mu_{2m}\frac{( \log q)^m}{\phi(q)^m}.$$ 
To show this, assume otherwise that for all $t  \in [S ,T] $, where we used the shorthand $S= T \mu_{2m} / \phi(q)^{m-1}(C_{\eps,\delta,\eta} \log q)^m$,
$$\Phi_0\Big(\frac {t}T\Big)M_{2m}(\e^{t},q;\eta) \leq(1- \eps ) \mu_{2m}\frac{( \log q)^m}{\phi(q)^m}.$$ 
Applying Corollary~\ref{corollary sup bound}, we deduce that
\begin{equation}
    \begin{split} 
\int_0^{\infty}&\!\!\!\Phi_0\Big(\frac {t}T\Big)M_{2m}(\e^{t},q;\eta) \d t  =\int_{S}^T\!\!\! \Phi_0\Big(\frac {t}T\Big)M_{2m}(\e^{t},q;\eta) \d t+ \int_{0}^S \!\!\!\Phi_0\Big(\frac {t}T\Big)M_{2m}(\e^{t},q;\eta) \d t\\
& \leq  \Big(\sup_{S\leq t\leq T} M_{2m}(\e^{t},q;\eta) \Phi_0\Big(\frac {t}T\Big) \Big) \int_{S}^T  \d t +O\Big(S\frac{(K_{\delta,\eta}\log q)^{2m}}{\phi(q)}\Big) \\
& \leq (1- \eps )  \mu_{2m}\frac{( \log q)^m}{\phi(q)^m}\frac T2 +O\Big(S\frac{(K_{\delta,\eta}\log q)^{2m}}{\phi(q)}\Big) .\end{split}
\label{equation here we put hooley}
\end{equation}
When $C_{\eps,\delta,\eta}$ is large enough, this contradicts~\eqref{equation lower bound first moment}. In summary, for all $q\geq 3 $ such that $\phi(q)\geq (C_{\eps,\delta,\eta})^m$ (by enlarging once more the constant $C_{\eps,\delta,\eta}$, we may replace $\phi(q)$ with $q$) and for any $U\geq 1$, there exists $x_m(q)$ such that $\log (x_m(q)) \in [U,U \phi(q)^{m-1} (C_{\eps,\delta,\eta}\log q)^m \mu_{2m}^{-1}]$ and such that 
$$M_{2m}(x_m(q),q;\eta) > (1- \eps ) \mu_{2m}\frac{(\log q)^m}{\phi(q)^m}.$$ 
The proof is finished.
\end{proof}
 In the particular case $n=2$, we can localize the large values of $M_{n}(\e^t,q;\eta)$ more precisely thanks to Lemma~\ref{lemma hooley upper bound}. 
 
 \begin{lemma} Let  $\delta>0$ and $\eta\in \S_\delta$, and assume that GRH$_{\widehat \eta}$ holds. Fix $\eps>0$, and let $f:\R_{\geq 0} \rightarrow \R_{\geq 0}$ be any function such that for all $x\geq 3$, $f(x)\geq \log x$. Then, for all $q$ large enough in terms of $\delta, \eta$ and $\eps$, there exists an associated value $x(q)$ for which $f(q)  \asymp_{\delta,\eta,\eps} \log(x(q))$ and such that
\begin{equation}
M_{2}(x(q),q;\eta) \geq (1-\eps) \frac{\alpha(\widehat \eta^2)\log q}{\phi(q)}.  
\end{equation}

\label{lemma raw oscillations variance}
\end{lemma}

\goodbreak
\begin{proof}
 The proof is almost identical to that of Lemma~\ref{lemma raw oscillations} except that one can apply Lemma~\ref{lemma hooley upper bound} and deduce that the left hand side of~\eqref{equation here we put hooley} is 
 $$\leq (1-\eps) T \frac{\log q}{\phi(q)} +O_{\delta,\eta}\Big( S \frac{\log q}{\phi(q)}\Big). $$
 \end{proof}

In the next lemma we apply our results on higher moments of $M_{2m}(x,q;\eta)$ and deduce a stronger $\Omega$-result. However, our localization of the large values of $M_{2m}(x,q;\eta)$ is less precise than in Lemma~\ref{lemma raw oscillations}.
\begin{lemma} Let $m,r\geq 1,$  $\delta>0$, $\eta\in \S_\delta$, and assume that GRH$_{\widehat \eta}$ holds. Fix $\eps>0$ and let $C_{\eps,\delta,\eta,m,r}$ be a large enough constant, and let $f: \mathbb R_{\geq 0}\mathbb \rightarrow \mathbb R_{\geq 0}$ be any function such that for all $t\geq 3$, $f(t)\geq 1$. Then, for all $q \geq C_{\eps,\delta,\eta,m,r}$, there exists an associated value $x_m(q)$ for which $\log(x_m(q))\in [f(q),f(q)C_{\eps,\delta,\eta,m,r} \phi(q)^{2mr+m-r}(\log q)^{m(2r+1)}]$ and such that
\begin{equation}\label{OmegaM2m precise}
M_{2m}(x_q,q;\eta) >  \mu_{2m}\Big( \frac{\alpha(\widehat \eta^2)\log q +\beta_q(\widehat \eta^2)}{\phi(q)}\Big)^m+ (1- \eps )   E_{r,m}
\frac{\nu_{2m}^{\frac12} (\alpha(\widehat \eta^2)\log q)^{ m}}{\phi(q)^{  m+\frac 12+\frac 1{2(2r+1)}}}, 
\end{equation}
where 
$$ E_{r,m}:= \Big( \frac{(2r+1)!}{2^{r+1} (r-1)!}\vartheta_{2m,2r+1}\Big)^{\frac 1{2r+1}}   .$$ 
\label{lemma raw oscillations higher moments}
\end{lemma}

\begin{proof} 
The proof starts as that of Lemma~\ref{lemma raw oscillations}, with the test function $\Phi_0$ defined in~\eqref{equation definition triangle}. From Theorem~\ref{thmomentscentres2rT} (in which GRH$_{\widehat \eta}$ is clearly sufficient), we deduce that for $\phi(q)\geq C_{\eps,\eta,n,r}$ and $ T\geq  C_{\eps,\delta,\eta,m,r}(\log q)^{m(2r+1)}  \phi(q)^{2mr+m-r}$,
\begin{align*}
\frac {1}T \int_{\mathbb R} \Phi_0\Big(& \frac tT \Big) \Big(M_{2m}(\e^t,q;\eta)-m_{2m}(q;\eta)\Big)^{2r+1}\d t \\ &\geq \Big(1-\frac{\eps}3\Big)^{2r+1}
 \frac{(2r+1)!}{2^{r+1} (r-1)!}\frac{\vartheta_{2m,2r+1}}{\phi(q)^{\frac 12}}
 V_{2m}(q;\eta)^{\frac{2r+1}2}.
\end{align*}
Similarly as in the proof of Lemma~\ref{lemma raw oscillations} (note that the fact that $2r+1$ is odd is crucial here), we deduce that there exists $ t_m(q) \in [D_{\eps,\delta,\eta,m,r}\frac {T}{\phi(q)^{m(2r+1)-r}(\log q)^{m(2r+1)}} ,T]$, where $D_{\eps,\delta,\eta,m,r}>0$ is a small enough constant, such that 
$$M_{2m}(\e^{t},q;\eta) -m_{2m}(q;\eta) > \Big(1- \frac \eps2 \Big) \Big(\frac{(2r+1)!}{2^{r+1} (r-1)!}\Big)^{\frac 1{2r+1}}\frac{\vartheta_{2m,2r+1}^{\frac 1{2r+1}}}{\phi(q)^{\frac 1{2(2r+1)}}} 
 V_{2m}(q;\eta)^{\frac{1}2}.
$$
Finally, we apply the lower bound on $m_{2m}(q;\eta)$ given in Theorem~\ref{thmomentscentres2rT}.
\end{proof}

\begin{cor} Let  $\delta>0$ and $\eta\in \S_\delta$, and assume that GRH$_{\widehat \eta}$ holds. 
Fix $\eps>0$ small enough, and let $g:\mathbb R_{\geq 1} \rightarrow \R_{\geq 3}$ be an increasing function tending to infinity such that $g(u) \leq \e^{u}$. For each modulus $q\gg_{\eps,\delta, \eta} 1$, there exists an associated value $x_q$ such that $  g(c_1(\eps,\delta,\eta)\log x_q) \leq q \leq  g(c_2(\eps,\delta,\eta)\log x_q) $, where the $c_j(\eps,\delta,\eta)>0$ are constants, and
\begin{equation}
\label{OmegaM2}
M_{2}(x_q,q;\eta) \geq (1-\eps) \frac{\alpha(\widehat \eta^2)\log q}{\phi(q)}.  
\end{equation}
If moreover $g(u) \leq  u^{\frac 1M}$ for some $M\geq 4\eps^{-1}$ and for all $u$, then we can find $x_q$ such that $g((\log x_q)^{1-\frac{3}{M\eps}}) \leq q \leq  g(\log x_q) $ and
\begin{equation}
 M_{2}(x_q,q;\eta) \geq   \frac{\alpha(\widehat \eta^2)\log q +\beta_{q}(\widehat \eta^2)}{\phi(q)}+ 
\frac{ 1}{\phi(q)^{   \frac 32 -\eps}}. 
\label{equation lower bound lemma large values variance}    
\end{equation}
\label{corollary m=1 omega result GRH}
\end{cor}
\begin{proof}
For the first statement we apply Lemma~\ref{lemma raw oscillations variance} with $f=g^{-1}$, which implies that $f(q)\leq \log q$.

For the second statement, we apply Lemma~\ref{lemma raw oscillations higher moments} with $f=g^{-1}$, which now implies that $f(q) \geq q^M$. Hence, taking $r = \lceil\frac{\eps^{-1}}4 \rceil$ we have that~\eqref{equation lower bound lemma large values variance} holds for some $x_q$ in the range $f(q) \leq \log(x_q)\leq f(q)C_{\eps,\eta,1,r} \phi(q)^{r+1}(\log q)^{2r+1} \leq f(q)^{1+\frac {\eps^{-1}/4+2}{M}}$, and the proof follows. Here, we have used the bound $f(q) \geq q^M$ which implies that $\log q < M \log f(q) < f(q)^\eps$.   
\end{proof}

\begin{cor} 
\label{corollary m>=2 omega result GRH}
Let $\delta,\eps>0$ and $\eta\in \S_\delta$, assume that GRH$_{\widehat \eta}$ holds, and fix $m\geq 2$. 
For any $B\in \mathbb R_{ > 0}$, there exists a real number $\lambda \in [\frac 1{m-1+B},\frac 1B]$ with the following property. There exists a sequence of moduli $\{q_i\}_{i\geq 1}$ and associated values $\{x_i\}_{i\geq 1}$ such that 
$q_i= (\log x_i)^{\lambda+o(1)} $ and
\begin{equation}\label{OmegaM2m higher moments}
M_{2m}(x_i,q_i;\eta) \geq (1-\eps)\mu_{2m}\frac{(\alpha(\widehat \eta^2)\log q_i)^m}{\phi(q_i)^m}.  
\end{equation}

Let moreover $g:\mathbb R_{\geq 1} \rightarrow \R_{\geq 3}$ be an increasing function such that as $x\rightarrow\infty$, $ g(x) \leq x^{\frac 1{mM}}$, with $M\geq 2\eps^{-1}$. Then, for each modulus $q\geq C_{\eps,\delta,\eta,m}$ there exists an associated value $x_q$ such that $g((\log x_q)^{1-\frac{1}{\eps M}}) \leq q \leq  g(\log x_q) $ and~\eqref{OmegaM2m precise} holds.

\end{cor}
\begin{proof}
The proof of the second part is similar to that of Corollary~\ref{corollary m=1 omega result GRH}. In the first part we pick $f(q)=(q\log q)^{B}$, and apply Lemma~\ref{OmegaM2m}. The resulting sequence $x_m(q)$ is such that
$$ B+o(1) \leq  \frac{\log\log x_m(q)}{\log q}  \leq B+m-1+o(1), $$
and we can extract a subsequence having the desired property. 
\end{proof}

We now work under the negation of GRH$_{\widehat \eta}$. It follows from the work of Kaczorowski and Pintz~\cite{KP86} that if $L(s,\chi)$ does not have real non-trivial zeros, then $\psi_{\eta}(x,\chi)$ has $\gg \log X$ large oscillations in the interval $[1,X]$. Moreover, Pintz~\cite{P83} and Schlage-Puchta~\cite{SP18} have obtained stronger results for $\psi(x)-x$. A weaker but unconditional result will be sufficient for our purposes. 
\goodbreak

\begin{lemma}
\label{lemma omega result psi(x,chi)}
Let  $\delta>0$ and $\eta\in\S_\delta$, and fix $\eps >0$. For $q\geq 3$, let $\chi$ be a non-principal character $\bmod q$, and let $\Theta_{\chi,\eta}$ be the supremum of the real parts of the zeros $\rho_\chi$ of $L(s,\chi)$ such that $\widehat \eta\big( \frac{\rho_\chi-\frac 12}{2\pi i}\big) \neq  0$. Then, for every large enough $X\in\mathbb R$ (in terms of $\chi,\eps$ and $\eta$), there exists $x\in [X^{1-\eps},X]$ such that 
$$ |\psi_\eta(x,\chi)|>x^{\Theta_{\chi,\eta}-\frac 12 -\eps}. $$
\end{lemma}

\begin{proof}

We distinguish two cases, depending on the value of 
$$ \gamma_{\chi,\eta}:= \inf\{  |\gamma| :\gamma\in \mathbb R, L(\Theta_{\chi,\eta} +i\gamma,\chi)=0,\widehat \eta(\tfrac{\Theta_{\chi,\eta} -\frac 12 +i\gamma}{ 2\pi i})\neq 0 \}\in \mathbb R_{\geq 0}\cup \{\infty\}. $$
If $\gamma_{\chi,\eta}>0 $, then we apply~\cite[Theorem 1]{KP86} with $f(x):=2\Re e(\psi_{\eta}(x,\chi))=\psi_{\eta}(x,\chi)+\psi_{\eta}(x,\overline{\chi})$. We have that
 \begin{equation}
  \int_0^{\infty} f(x)  x^{-s-1} \d x=-\widehat \eta \Big(\frac{s}{2\pi i}\Big) \Big(\frac{L'(s+\frac 12,\chi)}{L(s+\frac 12,\chi)}+\frac{L'(s+\frac 12,\overline \chi)}{L(s+\frac 12,\overline \chi)}\Big).
\label{eq:Mellintransomega}
\end{equation}
By definition of $\Theta_{\chi,\eta}$, the function $L'(s+\frac 12,\chi)/L(s+\frac 12,\chi)$ has a pole $s_0$ of residue $\geq 1$ in the half plane $\Re e(s_0)> \Theta_{\chi,\eta}-\frac 12-\eps$, and moreover $\widehat \eta (\frac{s_0}{2\pi i})\neq 0$. Since $L(s+\frac 12,\overline{\chi})$ is entire, it follows that the poles of $L'(s+\frac 12,\overline{\chi})/L(s+\frac 12,\overline{\chi})$ have positive residues, and we conclude that the right hand side of \eqref{eq:Mellintransomega} is meromorphic on $\mathbb C$ and has a pole $s_0$ in the half plane $\Re e(s)> \Theta_{\chi,\eta}-\frac 12-\eps$. Hence, the conditions of~\cite[Theorem 1]{KP86} are satisfied and the conclusion follows.

We now assume that $\gamma_{\chi,\eta}=0$, which implies that $ L(\Theta_{\chi,\eta},\chi)=0$ and $\widehat \eta(\frac{\Theta_{\chi,\eta}-\frac 12}{2\pi i})\neq 0$. In this case, for $1<T_1<T_2$ we consider the average
$$ \frac 1{T_2-T_1}\int_{T_1}^{T_2} \e^{- t(\Theta_{\chi,\eta}-\frac 12)}\Re e(\psi_\eta(\e^t,\chi)) \d t,$$
which by Lemma~\ref{lemma:explicit formula} is 
\begin{equation}
    \begin{split}
    =& - \Re e\sum_{\rho_{\chi}\neq \Theta_{\chi,\eta}} \frac{\e^{ T_2(\rho_{\chi}-\Theta_{\chi,\eta}) }-\e^{ T_1(\rho_{\chi}-\Theta_{\chi,\eta})}}{ (T_2-T_1)(\rho_\chi - \Theta_{\chi,\eta})}\widehat \eta\Big( \frac{\rho_\chi-\frac 12}{2\pi i}\Big) \\ & - \Re e\Big( \widehat \eta\Big( \frac{\Theta_{\chi,\eta}-\frac 12}{2\pi i}\Big)\Big)\ord_{s=\Theta_{\chi,\eta}}L(s,\chi)+O_{\chi}\Big(\frac{\e^{-\Theta_{\chi,\eta} T_1}}{T_2-T_1} \Big). 
    \end{split}
    \label{equation average for RH false lemma}
\end{equation}
We have a similar identity for the average of $\Im m(\psi_\eta(\e^t,\chi))$. 
Then, taking $T_1=(1-\eps) \log X$ and $T_2=\log X$, we claim that~\eqref{equation average for RH false lemma}
 takes the form 
\begin{equation}
    \begin{split} \frac 1{\eps\log X}\int_{(1-\eps)\log X}^{\log X} &\e^{- t(\Theta_{\chi,\eta}-\frac 12)}\Re e\big(\psi_\eta(\e^t,\chi)\big) \d t\\ & = - \Re e\Big( \widehat \eta\Big( \frac{\Theta_{\chi,\eta}-\frac 12}{2\pi i}\Big)\Big)\ord_{s=\Theta_{\chi,\eta}}L(s,\chi) +O_\chi\Big ( \frac 1{\eps \log X}\Big). \end{split}
    \label{equation average for RH false lemma 2}
\end{equation} 
 Indeed, we have that $|\e^{ T_j(\rho_{\chi}-\Theta_{\chi,\eta}) }|\leq 1$. Moreover, we have the bound 
$$\widehat \eta(s) = \frac 1{-2\pi i s} \int_{\mathbb R} \e^{-2\pi  i s t} \eta'(t)\d t\ll \frac 1{|s|}  \int_{\mathbb R} \e^{ 2\pi |\Im m(s)| t} |\eta'(t)|\d t,   $$
which is $\ll 1/|s|$ whenever $ \Im m(s)\leq \frac 1{4\pi}$ by definition of $\S_\eta$. This is the case with $s= (\rho_\chi-\frac 12)/(2\pi i)$, and hence the infinite sum over zeros in~\eqref{equation average for RH false lemma} converges absolutely.

Coming back to~\eqref{equation average for RH false lemma 2} and its analogue for $\Im m(\psi_\eta(\e^t,\chi))$, we recall that $\widehat \eta(\frac{\Theta_{\chi,\eta}-\frac 12}{2\pi i})\neq 0$, and deduce that for all large enough $X$, there exists $t\in [(1-\eps)\log X,\log X]$ such that either $|\e^{- t(\Theta_{\chi,\eta}-\frac 12)}\Re e (\psi_\eta(\e^t,\chi) )| \gg_{\chi,\eta} 1,$ or $|\e^{ 
{-}
t(\Theta_{\chi,\eta}-\frac 12)}\Im m (\psi_\eta(\e^t,\chi) )| \gg_{\chi,\eta} 1$. 
This implies the claimed assertion.
\end{proof}

\begin{cor}
\label{corollary GRH false}
Let  $\delta>0$ and $\eta\in \S_\delta$, and assume that GRH$_{\widehat \eta}$ does not hold. Then, there exists an absolute (ineffective) number $\kappa>0$ such that the following holds. Let $f:\mathbb R_{\geq 0} \rightarrow \R_{\geq 3}$ be a increasing function tending to infinity such that $f(x) \leq x^{\kappa}$. 
For a positive (ineffective) proportion of moduli $q$, there exists an associated value $x_q$ such that $f(x_q^{1-\kappa})\leq  q \leq f(x_q) $ and for each $m\geq 1$,
$$ M_{2m}(x_q,q;\eta)  \gg \frac{ x_q^{m\kappa}}{\phi(q)^m}.$$ 
The implied constant in this bound is also ineffective.
\end{cor}

\begin{proof}
We first treat the case $m=1$. Let $\chi_e$ be a primitive character of least modulus $q_e$ for which there exists a zero $\rho_e=\Theta_e+i\gamma_e$ of $L(x,\chi_e)$ with $\Theta_e>\frac 12$, and such that $\widehat\eta\big(\frac{\rho_{e}-\frac 12}{2\pi}\big)\neq 0$. Fix $\eps< \Theta_e -\frac 12$, and pick any $\kappa<\Theta_e-\frac 12-\eps <\frac 12$. The set of moduli we will consider is the set of large enough multiples of $q_e$; those form a positive proportion of all moduli. If $q$ is such a multiple, we apply Lemma~\ref{lemma omega result psi(x,chi)}, and find that when $f^{-1}(q)$ is large enough in terms of $\kappa,\eps,\eta$  and $\chi_e$, there exists $x_q\in [f^{-1}(q),(f^{-1}(q))^{1+\kappa}]$ such that $|\psi_\eta(x_q,\chi_e)| > x^{\Theta_e-\frac 12-\eps}$. This implies that $f(x_q^{1-\kappa})\leq q\leq f(x_q)$. Denoting by $\chi_q$ the character modulo $q$ induced by $\chi_e$, by Lemma~\ref{lemma ramified primes} we have that 
$$ |\psi_\eta(x_q,\chi_q)| = |\psi_\eta(x_q,\chi_e)| +O(x_q^{-\frac 12} \log q)\gg_{\eps}  x_q^{\Theta_e-\frac 12-\eps}.$$ 
We deduce that
$$ M_2(x_q,q;\eta) = \frac 1{\phi(q)^2} \sum_{\substack{ \chi \bmod q\\ \chi \neq \chi_{0,q}}}|\psi_\eta(x_q,\chi)|^2  \geq \frac 1{\phi(q)^2} |\psi_\eta(x_q,\chi_q)|^2 \gg_{\eps}  \frac{x_q^{2\Theta_e-1-\eps}}{\phi(q)^2},$$ which is the required estimate. 

For $m \geq 2$ we apply H\"older's inequality and deduce that for all $x\geq 0$,
$$  M_{2m}(x,q;\eta) \geq M_2(x,q;\eta)^m.$$
The result follows.
\end{proof}

We are now ready to prove our main unconditional results.

\begin{proof}[Proof of Theorems~\ref{thunconditional n=2} and ~\ref{thunconditional n>2}]
If GRH$_{\widehat\eta}$ is false, then the claimed $\Omega$-results follow from Corollary~\ref{corollary GRH false}. Otherwise, they follow from Corollaries~\ref{corollary m=1 omega result GRH} and~\ref{corollary m>=2 omega result GRH}. 
\end{proof}

\begin{proof}[Proof of Corollary~\ref{corollary montgome}]
We argue by contradition. 
Assume that for some $q$ and  $K>0$ and for all $a\bmod q$ and $y\in [ \exp(c_2^{-1} g^{-1}(q)) ,\exp(c_1^{-1} g^{-1}(q))]$ we have the bound
$$ \Big|  \sum_{\substack{ n\leq y \\ n\equiv a\bmod {q}}} \Lambda(n)-\frac{1}{\phi(q)} \sum_{\substack{ n\leq y \\ (n,q)=1}} \Lambda(n) \Big| \leq K \frac{(y\log q)^{\frac 12}}{\phi(q)^{\frac 12}}.   $$
We let $\eta=\eta_0\star \eta_0$, where $\eta_0$ is any non-trivial smooth even function supported in $[-1,1]$. Applying summation by parts, we deduce that in the range $x\in [ \exp(c_2^{-1} g^{-1}(q)+2) ,\exp(c_1^{-1} g^{-1}(q)-2)]$,
\begin{align*}
    \psi_{\eta}( x;q,a) - \frac{\psi_{\eta} (x,\chi_{0,q})}{\phi(q)} &= \int_0^{\infty} \frac{\eta(\log(\frac yx))}{y^{\frac 12}} \d\Big(\psi(y;q,a) - \frac{\psi(y,\chi_{0,q})}{\phi(q)}\Big) \\
    &= -\int_{\e^{-2} x }^{\e^2 x} \frac{\eta'(\log(\frac yx))-\frac 12 \eta(\log(\frac yx))}{y^{\frac 32}}\Big(\psi(y;q,a) - \frac{\psi(y,\chi_{0,q})}{\phi(q)}\Big) \d y
    \\ &\ll K \frac{(\log q)^{\frac 12}}{\phi(q)^{\frac 12}}.
\end{align*}
This yields the upper bound
$  M_2(x, q;\eta) \ll  K^2  \log q / \phi(q).$
Under GRH$_{\widehat \eta}$, when $K$ is small enough and for some choice of $0<c_1,c_2\leq \frac 12$, this contradicts Lemma~\ref{lemma raw oscillations variance} as soon as $q$ is large enough. Under the negation of GRH$_{\widehat\eta}$, this contradicts Lemma~\ref{corollary GRH false} as soon as $q$ is in a well-chosen set (given by Lemma~\ref{corollary GRH false}) of positive (ineffective) density.
\end{proof}

Finally, we end this section by proving Corollary~\ref{corollary a=1}. This will be achieved through the following proposition, which we believe is of independent interest.

\begin{proposition}
\label{proposition moments a=1}
Assume GRH, and let  $\delta>0$, $\eta\in \S_\delta$, $T\geq 1$ and $\Phi\in \U$. For $m\in \mathbb N$, $q\geq 3$, and in the range $m\leq \log q/(2\log_2 q)$, we have the bound
\begin{multline*} \frac 1{T\int_0^{\infty} \Phi}\int_{\mathbb R}  \Phi\Big( \frac tT\Big) \Big( \psi_\eta(\e^t;q,1) - \frac{\psi_\eta(\e^t,\chi_{0,q})}{\phi(q)}\Big)^{2m} \d t \\ \geq \mu_{2m}\Big( \frac{ \alpha(\widehat \eta^2)\log q+\beta_q(\widehat \eta^2) }{\phi(q)}\Big)^m +O_\Phi\Big(\mu_{2m} \frac{ (C_{\delta,\eta}\log q)^{m}}{\phi(q)^{m+1 }}+ \frac{(K_{\delta,\eta}\log q)^{2m}}{T}\Big),
\end{multline*}
where $C_{\delta,\eta},K_{\delta,\eta}>0$ are constants.
 
\end{proposition}
\begin{proof}

The proof follows the lines of that of Lemma~\ref{lemma explicit formula higher moments of moments}.
Applying the orthogonality relations and the explicit formula in Lemma~\ref{lemma:explicit formula}, we see that
\begin{align*}&
\frac 1T\int_{0}^{\infty}   \Phi\Big( \frac tT\Big) \Big( \psi_\eta(\e^t;q,1) - \frac{\psi_\eta(\e^t,\chi_{0,q})}{\phi(q)}\Big)^{2m} \d t \\
&= \frac 1T\int_{0}^{\infty}   \Phi\Big( \frac tT\Big) \frac 1{\phi(q)^{2m}}\sum_{\substack{ \chi_1,\dots,\chi_{2m} \neq \chi_{0,q}}}\psi_\eta(\e^t,\chi_1)\cdots \psi_\eta(\e^t,\chi_{2m})\d t \\
&= \frac 1T\int_{0}^{\infty}   \Phi\Big( \frac tT\Big) \frac 1{\phi(q)^{2m}}\sum_{\substack{ \chi_1,\dots,\chi_{2m} \neq \chi_{0,q}}}\prod_{j=1}^{2m}\sum_{\gamma_{\chi_j}}\widehat\eta\Big(\frac{\gamma_{\chi_j}}{2\pi}\Big)\e^{- i t \gamma_{\chi_j}}  \d t+O_\Phi\Big( \frac{(K_{\delta,\eta}\log q)^{2m}}{T}\Big)\\
&= \frac 1{2\phi(q)^{2m}}\!\!\!\!  \sum_{\substack{ \chi_1,\dots,\chi_{2m} \neq \chi_{0,q}}}\sum_{\gamma_{\chi_1},\dots ,\gamma_{\chi_{2m}}}\!\!\!\widehat \Phi\Big(\frac T{2\pi}(\gamma_{\chi_1} +\dots +\gamma_{\chi_{2m}})\Big)\prod_{j=1}^{2m}\widehat\eta\Big(\frac{\gamma_{\chi_j}}{2\pi}\Big)+O_\Phi\Big( \frac{(K_{\delta,\eta}\log q)^{2m}}{T}\Big).
\end{align*}
At this point we apply positivity. Note that it is crucial here that the residue class we are working with is $a=1$, since otherwise we would not be able to handle the signs of $\overline{\chi}(a)$ in the explicit formula. The claimed result follows from Lemma~\ref{lemma lower bound sum over zeros}.
\end{proof}

\begin{proof}[Proof of Corollary~\ref{corollary a=1}]
After applying Proposition~\ref{proposition moments a=1}, the rest of the proof is similar to that of Corollary~\ref{corollary montgome}, the only major difference being that we apply the uniform bound 
$$\psi_\eta(\e^t;q,1) - \frac{\psi_\eta(\e^t,\chi_{0,q})}{\phi(q)}  \ll \log q $$
rather than Lemma~\ref{lemma hooley upper bound}. This is responsible for our weaker control on the size of $x_q$ in terms of $q$.
\end{proof}

\section{Applications of LI}

\label{section probabilistic model}
The goal of this section is to study the probabilistic model $H_n(q;\eta)$ defined in~Lemma \ref{lemma moments of limiting distribution as limit of V(T)}, for some $\eta \in \S_\delta$, under GRH and LI. 
To begin, we will see that the relation $\gamma_{1} +\dots +\gamma_{n}=0$ in the definition of $\Gamma_0(\bchi)$ (see Section~\ref{section limiting distributions}) has a very simple set of solutions.
\begin{lemma}\label{avecLI}
Let $q\geq 3$, and let $\gamma_1,\ldots,\gamma_n$ be ordinates of non-trivial zeros of the functions $L(s,\chi_1),\ldots, L(s,\chi_n)$, respectively, where $\chi_1,\dots \chi_n$ are non-principal characters $\bmod q$. Under GRH and LI, the relation $\gamma_1+\ldots+\gamma_n=0$ implies that $n=2m$ for some $m\in \mathbb N$, and moreover there exists $i_1,\ldots, i_m$ and $j_1,\ldots, j_m$ such that $$\{ i_1,\ldots, i_m, j_1,\ldots, j_m\}=\{ 1,\ldots, 2m\}$$ and  
$$\gamma_{i_k}+\gamma_{j_k}=0,\qquad \overline\chi_{j_k}=\chi_{i_k}.$$

\end{lemma}

\begin{proof}
Clearly, LI implies that $L(\frac 12,\chi)\neq 0$, that is $\gamma_\chi\neq 0$.
 We argue by induction over $n$. 
The statement is clear for $n=1$. For $n=2$, note that $\gamma_1+\gamma_2=0$ implies that 
$\gamma_1\gamma_2<0$, say $\gamma_1>0$ and $\gamma_2<0$. Also, $\frac 12-i\gamma_2$ is a zero of $L(s,\overline{\chi_2})$, and $\gamma_1=-\gamma_2$ implies by LI that $\chi_1=\overline{\chi_2}$.

Let $n\geq 3$
 and assume that the statement holds for all $1\leq r<n$. 
We can reorder the $\gamma_i$ in such a way that $ \gamma_1, \dots \gamma_{\ell} >0$, and $\gamma_{\ell+1},\dots ,\gamma_n<0$, for some $1\leq \ell\leq n-1$ (clearly, the $\gamma_i$ are not all of the same sign). Then, 
\begin{equation}
    \label{formula before LI}
  \gamma_1+\dots +\gamma_\ell = (-\gamma_{\ell+1}) +\dots+(-\gamma_n),\end{equation}
and $\frac 12-i\gamma_{\ell+1},\dots \frac 12-i\gamma_n$ are zeros of $L(s,\overline{\chi_j})$ with $\gamma_j>0$. 
We claim that  
$$ \{ \gamma_1,\dots ,\gamma_\ell\} = \{-\gamma_{\ell+1}  ,\dots, -\gamma_n \} .$$ Indeed, if for example  $\gamma_1\notin \{-\gamma_{\ell+1}  ,\dots, -\gamma_n \}$, then the formula
\eqref{formula before LI} implies that $\gamma_1$ is a $\mathbb Q$-linear combination of  positive imaginary parts of other zeros, which contradicts LI.
We conclude that there exists $j_\ell \geq \ell+1$ such that $-\gamma_{j_\ell}=\gamma_\ell$; by LI, we also need to have $\chi_\ell=\overline{ \chi_{j_\ell}}.$
Hence, our induction hypothesis for $r=n-2$ implies that $r$ is even, $\ell-1=n-\ell-1$, and moreover we can find distinct indices $j_1,\ldots, j_{\ell-1}\in\{ \ell+1,\ldots, n\}\smallsetminus \{ j_\ell\}$ such that for $1\leq k\leq\ell-1$, $-\gamma_{j_k}=\gamma_k $ and $\chi_k=\overline{ \chi_{j_k}}$.
\end{proof}

A direct consequence of Lemma~\ref{lemma explicit formula for moments of nu n} is the following.

\begin{lemma} 
\label{momentimpair} Let $\delta>0$, $\eta\in \S_\delta$, $r, m\in \mathbb N$ and $q\geq 3.$ Under GRH and LI, we have that 
$$\mathbb E[H_{2m+1}(q;\eta)^{2r+1}]=0  .$$
\end{lemma}
\begin{proof}
 By Lemma~\ref{avecLI}, if $\gamma_{\mu,j}$ are imaginary parts of zeros of $L(s,\chi_{\mu,j})$, then  $\gamma_{\mu,1}+\dots + \gamma_{\mu,2m+1}\neq 0$. Hence, Lemma~\ref{lemma explicit formula for moments of nu n} implies that $\mathbb E[H_{2m+1}(q;\eta)^{2r+1}]$ is equal to an empty sum.
\end{proof}
 
 Taking $\Phi=1_{[-1,1]}$ and $T\rightarrow \infty$ in Lemma~\ref{lemma explicit formula higher moments of moments} and applying Lemma~\ref{lemma formula for moments of limiting distribution}, we obtain the following corollary of Lemma~\ref{avecLI}.
 \begin{cor}
Let $\delta>0$ and $\eta\in \S_\delta$, and assume GRH and LI. Let $\bm \chi= ( \chi_{\mu,j})_{\substack{ 1\leq \mu\leq s \\ 1\leq j \leq n}}$ be an array of characters, and for $1\leq \mu\leq s$ and $1\leq j \leq n$, denote by $\gamma_{\mu,j} $ the imaginary part of a generic non-trivial zero of $ L(\,\cdot\,,\chi_{\mu,j})$.
Then, there exists an involution $\pi$ of the set $ \{ 1,\dots,s\} \times \{1,\dots ,n\}$ having no fixed points such that  $\pi(\bchi ) = \overline{\bchi }$ and, for $1\leq \mu\leq s$ and $1\leq j \leq n$, $\gamma_{\mu,j}+ \gamma_{\pi(\mu,j)}=0 $.   Moreover, if the product $sn$ is even, then we have the upper bound
$$ \mathbb V_{s,n}(q;\eta)  \leq    \frac {1}{\phi(q)^{sn}} 
\sum_{\pi \in I_{s,n}} \sum_{\substack{ \bchi\in X_{s,n}\\   \pi(\bchi ) =\overline{ \bchi } }}  \sum_{\substack{ \bm \gamma \in \Gamma(\bchi) \\  \forall (\mu,j),\,  \gamma_{\mu,j}+\gamma_{\pi(\mu,j)}=0 }} \widehat \eta(\bm \gamma) \Delta_{s}(\bm\sigma_{\bm\gamma}), $$
where $I_{s,n}$ is the set of involutions $\pi$ of the set $ \{ 1,\dots,s\} \times \{1,\dots ,n\}$ having no fixed points and such that\footnote{Actually, the first condition in~\eqref{equation Jmunu >2} is automatic since $\pi$ is an involution having no fixed points. } for each $1\leq \mu\leq s$ \begin{equation}
   |J_{\mu,\mu}(\pi)| \equiv 0 \bmod 2;\qquad  \sum_{\substack{1\leq \nu \leq s \\ \nu \neq \mu}} |J_{\mu,\nu}(\pi)| \geq 2,  
     \label{equation Jmunu >2}
\end{equation}
 where the set $J_{\mu,\nu}(\pi)$ is defined in~\eqref{equation definition Jmunu}. 
 \end{cor}
 Note that this upper bound on $\mathbb V_{s,n}(q;\eta)$ is not an equality, since there are arrays of characters associated to more than one involution $\pi\in I_{s,n}$. However, we will see in Lemma~\ref{lemma sum over characters LI} that the overcount is negligible. The following definition will be useful.
\begin{definition} For $r\geq 1$ and $n\geq 2$, let
$ G_{2r,n} $ be the set of involutions $\pi: \{1,\dots,2r \}\times \{ 1,\dots,n\} \rightarrow \{1,\dots,2r \}\times \{ 1,\dots,n\} $ having no fixed points for which~\eqref{equation Jmunu >2}
holds, and
$$\big| \big\{ (\mu,\nu) :1\leq  \mu < \nu \leq 2r, J_{\mu,\nu}(\pi)\neq \varnothing \big\}\big|\geq r+1. $$
In particular, recalling Definition~\ref{definition F} we have that
$$ I_{2r,n}=  F_{2r,n}\cup G_{2r,n}.$$
\end{definition}

\begin{lemma}\label{Lemma8.5}
Let $ r\in \mathbb N$, $n\geq 2,$ $q\geq 3$, and let $\pi \in G_{2r,n}$. We have the upper bound
\label{lemma sum over characters LI}
$$C_{2r,n}(q;\pi):=\sum_{\substack{\bchi\in X_{2r,n}
\\  \pi(\bchi ) =\overline{ \bchi }}}  1 \leq   \phi(q)^{(n-1)r-1} .$$
\end{lemma}

\begin{proof}
For an array of characters  $\bm \chi=( \chi_{\mu,j})_{\substack{ 1\leq \mu\leq s \\ 1\leq j \leq n}}$ such that $\pi(\bchi ) =\overline{ \bchi }$, we will use the notation 
$$ \Theta_{\mu,\nu}(\bm \chi,\pi) := \prod_{(\mu,j) \in   J_{\mu,\nu}  } \chi_{\mu,j}. $$
Note that $\Theta_{\mu,\mu}(\bm \chi,\pi)=\chi_{0,q}$ and $\Theta_{\mu,\nu}(\bm \chi,\pi) = \overline{\Theta_{\nu,\mu}(\bm \chi,\pi) }$. The relations
$$\chi_{\mu,1}\cdots \chi_{\mu,n} =\chi_{0,q}\qquad (1\leq \mu\leq 2r)$$
are equivalent to
\begin{equation}
\label{equation relations characters}
\prod_{\nu \neq \mu } \Theta_{\mu,\nu}(\bm \chi,\pi) = \chi_{0,q}\qquad (1\leq \mu\leq 2r).
\end{equation}
We consider
$$E_1(\pi):=\big\{ (\mu,\nu) :1\leq  \mu < \nu \leq 2r, J_{\mu,\nu}(\pi)\neq \varnothing \big\}.$$
We let $E(\pi)\subset E_1(\pi)$ be a set of least cardinality for which 
\begin{equation}
 \bigcup_{(\mu,\nu) \in E(\pi)} \{ \mu,\nu\} =\{ 1,\dots ,2r\}, 
 \label{equation property E(pi)}
\end{equation}
that is every $1\leq \mu \leq 2r$ is paired with some $1\leq \nu \leq 2r$ through $E(\pi)$. The existence of $E(\pi)$ follows from the assumption~\eqref{equation Jmunu >2}.
Clearly, $ |E(\pi)|\geq r$. 

We consider two distinct cases. First, we assume that $|E(\pi)|=r$. Then, by reordering the array $\bm \chi$, we may assume that $E(\pi)=\{(1,2),(3,4),\dots,(2r-1,2r)\}$ and $(1,3) \in E_1(\pi)$. For a given value of $\Theta_{\mu,\nu}(\bm \chi,\pi) $ for all $(\mu,\nu)\in \{(\mu,\nu) : 1\leq \mu<\nu \leq 2r  \}\smallsetminus (E(\pi)\cup\{(1,3)\})$, we claim that the relations~\eqref{equation relations characters} have a unique solution, that is the $\Theta_{\mu,\nu}(\bm \chi,\pi) $ with $(\mu,\nu) \in \E(\pi)\cup \{ (1,3)\} $ are determined. This is clearly the case for $\mu_0=5,7,\dots,2r-1$ since exactly one of the $\Theta_{\mu,\nu}(\bm \chi,\pi) $ with $(\mu,\nu)\in E(\pi)\cup \{(1,3)\}$ appears in the relations~\eqref{equation relations characters} with $\mu=\mu_0$. Now, for $(\mu_0,\nu_0) \in \{(1,2),(1,3),(3,4) \} $, we have the relations 
$$ 
\begin{cases}
\overline{\Theta_{1,2}(\bm \chi,\pi)}\Theta_{2,3}(\bm \chi,\pi)\Theta_{2,4}(\bm \chi,\pi) &=   \prod_{ 5\leq \nu \leq 2r } \overline{\Theta_{2,\nu}(\bm \chi,\pi)} \\
\overline{\Theta_{1,3}(\bm \chi,\pi)}\overline{\Theta_{2,3}(\bm \chi,\pi)}\Theta_{3,4}(\bm \chi,\pi) &=   \prod_{ 5\leq \nu \leq 2r } \overline{\Theta_{3,\nu}(\bm \chi,\pi)}\\
\overline{\Theta_{1,4}(\bm \chi,\pi)}\overline{\Theta_{2,4}(\bm \chi,\pi)}\overline{\Theta_{3,4}(\bm \chi,\pi)}&=   \prod_{ 5\leq \nu \leq 2r } \overline{\Theta_{4,\nu}(\bm \chi,\pi)},
\end{cases}
$$
which imply once more a unique choice for $\Theta_{\mu,\nu}(\bm \chi,\pi) $ with $(\mu,\nu) \in  \{(1,2),(1,3),(3,4)\} $. Now, $\Theta_{\mu,\nu}$ is a product of $|J_{\mu,\nu}|$ characters, hence the total number of summands in $C_{2r,n}(q;\pi)$ is at most
$$ \exp\Big(\log (\phi(q))\Big(\sum_{\mu=1}^{2r} \frac{|J_{\mu,\mu}|}2 + \sum_{\substack{1\leq \mu<\nu \leq 2r  \\ (\mu,\nu) \notin E(\pi) \cup \{ 1,3\}}}|J_{\mu,\nu}|\Big) \Big)\leq \phi(q)^{(n-1)r-1}.$$

We now consider the second case where $|E(\pi)|\geq r+1$. We claim that for every $(\mu_0,\nu_0)\in E(\pi)$, we either have that $ \mu_0 \notin \{ \mu,\nu\} $ for all $(\mu,\nu)\in E(\pi)\smallsetminus \{(\mu_0,\nu_0)\}$, or $ \nu_0 \notin \{ \mu,\nu\} $ for all $(\mu,\nu)\in E(\pi)\smallsetminus \{(\mu_0,\nu_0)\}$. Indeed, the contrary would contradict minimality of $E(\pi)$, since $E(\pi) \smallsetminus \{(\mu_0,\nu_0)\}$ would have the property~\eqref{equation property E(pi)}. Moreover, for a given value of $\Theta_{\mu,\nu}(\bm \chi,\pi) $ for all $(\mu,\nu)\in \{(\mu,\nu) : 1\leq \mu<\nu \leq 2r  \}\smallsetminus E(\pi)$, we claim that the relations~\eqref{equation relations characters} have a unique solution, that is the $\Theta_{\mu,\nu}(\bm \chi,\pi) $ with $(\mu,\nu) \in E(\pi) $ are determined. To show this, we fix $(\mu_0,\nu_0)\in E(\pi)$, and recall that either $ \mu_0 \notin \{ \mu,\nu\} $ for all $(\mu,\nu)\in E(\pi)\smallsetminus \{(\mu_0,\nu_0)\}$, or $ \nu_0 \notin \{ \mu,\nu\}\smallsetminus \{(\mu_0,\nu_0)\} $ for all $(\mu,\nu)\in E(\pi)$. In the first case, taking $\mu=\mu_0$ in~\eqref{equation relations characters} shows that  $\Theta_{\mu_0,\nu_0}(\bm \chi,\pi) $ is uniquely determined. In the second, this follows from taking $\mu=\nu_0$. As before, we conclude that 
$$ C_{2r,n}(q;\pi) \leq \phi(q)^{(n-1)r-1}.$$
 \end{proof}

We now move to the case where $s=2r+1$ is odd. As in the even case, 
we use the notation
$$E_1(\pi ):= \big\{ ( \mu,\nu) :1\leq  \mu < \nu \leq 2r+1, J_{\mu,\nu}(\pi)\neq \varnothing \big\}.$$

\begin{definition} Let $ r\in \mathbb N$, $n\geq 2.$ 
We define
$ G_{2r+1,n} $ to be the set of involutions $\pi:\{1,\dots,2r+1 \}\times \{ 1,\dots,n\}\rightarrow \{1,\dots,2r+1 \}\times \{ 1,\dots,n\} $ having no fixed points such that~\eqref{equation Jmunu >2} holds
for $1\leq \mu \leq 2r+1$, and for which either
$| E_1(\pi)|\geq r+3, $
or $| E_1(\pi)| =  r+2 $ and there does not exist $1\leq \mu_1<\mu_2<\mu_3 \leq 2r+1$ such that $\{(\mu_1,\mu_2),(\mu_2,\mu_3),(\mu_1,\mu_3)\} \subset E_1(\pi)$.
\end{definition}

We remark that~\eqref{equation Jmunu >2} implies that $|E_1(\pi)|\geq r+1$, and moreover we have that $$I_{2r+1,n}= F_{2r+1,n}\cup G_{2r+1,n}.$$

\begin{lemma}
Let $ r\in \mathbb N$, $n\in 2\mathbb N,$ $q\geq 3$, and let $\pi \in G_{2r,n}$. We have the bound
\label{lemma sum over characters odd LI}
$$C_{2r+1,n}(q;\pi):=\sum_{\substack{ \bchi\in X_{2r+1,n}\\  \pi(\bchi ) =\overline{ \bchi }}}  1 \leq   \phi(q)^{(n-1)\frac{2r+1}2-\frac 32} .$$
\end{lemma}

\begin{proof}

We will use the same notation as in the proof of Lemma~\ref{lemma sum over characters LI}.
We seperate the proof into two distinct cases. 
Firstly, we assume that $|E(\pi)| \geq r+2$. As in the proof of Lemma~{\ref{Lemma8.5}}, 
for a given value of $\Theta_{\mu,\nu}(\bm \chi,\pi) $ for all $(\mu,\nu)\in \{(\mu,\nu) : 1\leq \mu<\nu \leq 2r +1 \}\smallsetminus E(\pi)$, the relations~\eqref{equation relations characters} have a unique solution. We conclude that
$$ C_{2r+1,n}(q;\pi) \leq \phi(q)^{(n-1)\frac{2r+1}2-\frac 32}.$$

Secondly, we assume that $E(\pi) = r+1$. After reordering $\bm \chi$ we may also assume that 
$$E(\pi)=\{ (1,2),(1,3),(4,5),(6,7),\dots,(2r,2r+1)\}.$$
Now, since $E_1(\pi) \geq r+2$, we may assume that one of $(1,4),(2,4),$ or $(4,6) $, which we will denote by $(\mu_e,\nu_e)$, is an element of $ E_1(\pi) $ (recall the definition of $G_{2r+1,n}$). Arguing as before, we see that for a given value of $\Theta_{\mu,\nu}(\bm \chi,\pi) $ for all $(\mu,\nu)\in \{(\mu,\nu) : 1\leq \mu<\nu \leq 2r +1 \}\smallsetminus (E(\pi)\cup \{ (\mu_e,\nu_e)\})$, the relations
\begin{equation}
\label{equation relations characters odd}
\prod_{\nu \neq \mu } \Theta_{\mu,\nu}(\bm \chi,\pi) = \chi_{0,q}\qquad (1\leq \mu\leq 2r+1)
\end{equation}
determine $\Theta_{\mu,\nu}(\bm \chi,\pi)$ for $ (\mu,\nu)\in E(\pi)\cup \{ (\mu_e,\nu_e)\}$, $\mu\geq 8$. 
In the case where $(\mu_e,\nu_e)=(1,4)$, $\Theta_{6,7}(\bm \chi,\pi)$ is also determined, and moreover we have that 
$$ 
\begin{cases}
\overline{\Theta_{1,2}(\bm \chi,\pi)}\Theta_{2,3}(\bm \chi,\pi)\Theta_{2,4}(\bm \chi,\pi)\Theta_{2,5}(\bm \chi,\pi) &=   \prod_{ 6\leq \nu \leq 2r +1} \overline{\Theta_{2,\nu}(\bm \chi,\pi)} \\
\overline{\Theta_{1,3}(\bm \chi,\pi)}\overline{\Theta_{2,3}(\bm \chi,\pi)}\Theta_{3,4}(\bm \chi,\pi)\Theta_{3,5}(\bm \chi,\pi) &=   \prod_{ 6\leq \nu \leq 2r+1 } \overline{\Theta_{3,\nu}(\bm \chi,\pi)}\\
\overline{\Theta_{1,4}(\bm \chi,\pi)}\overline{\Theta_{2,4}(\bm \chi,\pi)}\overline{\Theta_{3,4}(\bm \chi,\pi)}\Theta_{4,5}(\bm \chi,\pi) &=   \prod_{ 6\leq \nu \leq 2r+1 } \overline{\Theta_{4,\nu}(\bm \chi,\pi)}\\
\overline{\Theta_{1,5}(\bm \chi,\pi)}\overline{\Theta_{2,5}(\bm \chi,\pi)}\overline{\Theta_{3,5}(\bm \chi,\pi)}\overline{\Theta_{4,5}(\bm \chi,\pi)} &=   \prod_{ 6\leq \nu \leq 2r+1 } \overline{\Theta_{5,\nu}(\bm \chi,\pi)}
\end{cases}
$$
which shows that $\Theta_{1,2}(\bm \chi,\pi),$ $\Theta_{1,3}(\bm \chi,\pi),$ $\Theta_{4,5}(\bm \chi,\pi)$ and $\Theta_{1,4}(\bm \chi,\pi)$ are also determined.  The same argument works in the case $(\mu_e,\nu_e)=(2,4)$. Finally, for $(\mu_e,\nu_e)=(4,6)$ we have that 
$$ 
\begin{cases}
\overline{\Theta_{1,2}(\bm \chi,\pi)}\Theta_{2,3}(\bm \chi,\pi) &=   \prod_{ 4\leq \nu \leq 2r +1} \overline{\Theta_{2,\nu}(\bm \chi,\pi)} \\
\overline{\Theta_{1,3}(\bm \chi,\pi)}\overline{\Theta_{2,3}(\bm \chi,\pi)} &=   \prod_{ 4\leq \nu \leq 2r +1} \overline{\Theta_{3,\nu}(\bm \chi,\pi)}\\
\end{cases}
$$
and
$$ 
\begin{cases}
\overline{\Theta_{4,5}(\bm \chi,\pi)}\Theta_{5,6}(\bm \chi,\pi)\Theta_{5,7}(\bm \chi,\pi) &=   \prod_{ \nu \notin [4,7] } \overline{\Theta_{5,\nu}(\bm \chi,\pi)} \\
\overline{\Theta_{4,6}(\bm \chi,\pi)}\overline{\Theta_{5,6}(\bm \chi,\pi)}\Theta_{6,7}(\bm \chi,\pi) &=   \prod_{ \nu \notin [4,7] } \overline{\Theta_{6,\nu}(\bm \chi,\pi)}\\
\overline{\Theta_{4,7}(\bm \chi,\pi)}\overline{\Theta_{5,7}(\bm \chi,\pi)}\overline{\Theta_{6,7}(\bm \chi,\pi)}&=   \prod_{  \nu \notin [4,7] } \overline{\Theta_{7,\nu}(\bm \chi,\pi)},
\end{cases}
$$
which once more show that $\Theta_{1,2}(\bm \chi,\pi),$ $\Theta_{1,3}(\bm \chi,\pi),$ $\Theta_{4,5}(\bm \chi,\pi),$  $\Theta_{4,6}(\bm \chi,\pi)$ and $\Theta_{6,7}(\bm \chi,\pi)$ are determined.

We conclude that 
$$ C_{2r+1,n}(q;\pi) \leq \phi(q)^{(n-1)\frac{2r+1}2-\frac 32}.$$
 \end{proof}
 \goodbreak
 
We are now ready to finish the proof of Theorem~\ref{thmomentscentres2r}.

\begin{proof}[Proof of Theorem~\ref{thmomentscentres2r}, second part]

We first prove the estimate for $\mathbb E[H_{2m}(q;\eta) ]$.  
 Lemma~\ref{lemma explicit formula for moments of nu n} implies that
\begin{equation}\label{calculEHnq}
\mathbb E[H_{2m}(q;\eta)]  =  \frac{ 1}{\phi(q)^{2m}}  \sum_{\substack{\bchi\in X_{1,n}}}\sum_{\substack{\gamma_{\chi_1},\ldots , \gamma_{\chi_{2m}}\\ \gamma_{\chi_1}+\ldots+ \gamma_{\chi_{2m}}=0}}\widehat\eta\Big(\frac{\gamma_{\chi_1}}{2\pi}\Big) \cdots \widehat\eta\Big(\frac{\gamma_{\chi_{2m}}}{2\pi}\Big),
\end{equation} 
which combined with Lemma~\ref{avecLI} takes the form (recall~\eqref{equation definition b(chi,h)}, and note that under LI, the zeros of $L(s,\chi)$ are simple)
\begin{align*}
 \mathbb E[H_{2m}(q;\eta)]& =  \frac{ 1}{\phi(q)^{2m}}\!\!\!\!\!\! \sum_{\substack{\{ i_1,j_1 \},\dotsm, \{ i_m ,j_m\} \\ \{i_1,\dots,i_m,j_1,\dots,j_m\}=\{1,\dots ,2m\} \\ i_k = \min\{ i_\ell,j_\ell \; (\ell \geq k) \} } } \sum_{\substack{ \chi_{i_1},\dots,\chi_{i_m} \neq \chi_{0,q}  \\ \forall k\neq \ell, \chi_{i_\ell}\neq   \overline{\chi_{i_{k}}}  }}\sum_{\substack{\gamma_{\chi_{i_1}},\ldots , \gamma_{\chi_{i_m}}}}\Big|\widehat\eta\Big(\frac{\gamma_{\chi_{i_1}}}{2\pi}\Big)\Big|^{ 2} \cdots \Big|\widehat\eta\Big(\frac{\gamma_{\chi_{i_m}}}{2\pi}\Big)\Big|^{ 2}\\
 &\quad+O\Big( \mu_{2m} \frac{m^2}{\phi(q)^{2m}}
\sum_{\chi\neq \chi_{0,q}}b (\chi;\widehat \eta^2)^2  
\Big(\sum_{\chi_{k}\neq \chi_{0,q} }  b (\chi_{k};\widehat \eta^2)\Big)^{m-2} \Big)
 \\ 
 &= \frac{  \mu_{2m}}{\phi(q)^{m}} \Big(\sum_{\substack{ \chi \bmod q \\ \chi \neq \chi_{0,q}}} b(\chi;\widehat \eta^2)\Big)^{m}+O\Big(  \mu_{2m} \frac{m^2(K_{\delta,\eta}\log q)^{m}}{\phi(q)^{m+1}} \Big),
\end{align*}
by Proposition~\ref{proposition first moment b(chi,h)}. 
The claimed estimate follows from a second application of Proposition~\ref{proposition first moment b(chi,h)}.

Moving on to higher moments, by taking $T\rightarrow \infty $ in Lemma~\ref{lemma explicit formula higher moments of moments} and applying Lemma~\ref{lemma moments of limiting distribution as limit of V(T)}, we obtain that
\begin{align*}
\V_{s,n}(q;\eta)&=  \frac {(-1)^{sn}}{\phi(q)^{sn}} \sum_{\substack{\bchi\in X_{s,n}}}  \sum_{\substack{ \bm \gamma\in \Gamma(\bchi)  }} \widehat \eta(\bm \gamma) \Delta_{s}(\bm\sigma_{\bm\gamma})  \\ 
&=\frac {(-1)^{sn}}{\phi(q)^{sn}}  \sum_{\substack{\bchi\in X_{s,n}\\  \forall (\mu,j), \exists (\nu,i)\neq (\mu,j) : \chi_{\mu,j} =\overline{ \chi_{\nu,i}}}}  \sum_{\substack{  \bm \gamma\in \Gamma(\bchi) }} \widehat \eta(\bm \gamma) \Delta_{s}(\bm\sigma_{\bm\gamma}),
\end{align*}
by Lemma~\ref{avecLI}. The inner sum may be further restricted by requiring that $\gamma_{\chi_{\mu,j}}+\gamma_{\chi_{\nu,i}} =0$, where $(\mu,j)$ and $(\nu,j)$ are such that  $ \chi_{\mu,j} =\overline{ \chi_{\nu,i}}$. We will first discard characters in the outer sum in such a way to change the condition $ \forall (\mu,j), \exists (\nu,i) : \chi_{\mu,j} =\overline{ \chi_{\nu,i}}$ into  $\forall (\mu,j), \exists! (\nu,i)\neq (\mu,j) : \chi_{\mu,j} =\overline{ \chi_{\nu,i}}$. The sum of the extra terms is
\begin{align*}
&     \ll \sum_{\substack{ \bchi\in X_{s,n}\\  \forall (\mu,j), \exists (\nu,i) \neq  (\mu,j)  : \chi_{\mu,j} =\overline{ \chi_{\nu,i}} \\ \exists (\mu,j) ,(\nu,i) ,(\nu',j') \text{ distinct } :  \chi_{\mu,j} =\overline{ \chi_{\nu,i}}=\overline{ \chi_{\nu',i'} } }}  \!\!\!\!\!\!\!\!\!\!\!(K_{\delta,\eta}\log q)^{\frac{sn}2}  \\& \leq (K_{\delta,\eta}\log q)^{\frac{sn}2}  \sum_{\pi \in I_{s,n}} \!\!\!\!\!\!\!\!\! \sum_{\substack{ \bchi\in X_{s,n}\\  \ \pi(\bchi ) =\overline{ \bchi } \\ \exists (\mu,j) ,(\nu,i) ,(\nu',j') \text{ distinct } :  \chi_{\mu,j} =\overline{ \chi_{\nu,i}}=\overline{ \chi_{\nu',i'} } }}\!\!\!\!\!\!\!\!\!\!\! 1, 
\end{align*}
where we have used the bound 
$$  \sum_{\gamma_{\chi}} \Big| \widehat \eta \Big(\frac {\gamma_\chi}{2\pi }\Big) \Big|^2 \ll_{\delta,\eta} \log q,  $$
which is a consequence of the decay of $\widehat \eta$ combined with the Riemann-von Mangoldt formula.
By Lemmas~\ref{lemma sum over characters LI} and~\ref{lemma sum over characters odd LI}, the total contribution of the $\pi\in G_{s,n}$ is 
$$ \ll |G_{s,n}| \phi(q)^{ \frac{(n-1)s}2 -1 -\frac{\delta_{2\nmid s}}2} (K_{\delta,\eta}\log q)^{\frac{sn}2}, $$
where $\delta_{2\nmid s}=1$ when $2\nmid s$, and is zero otherwise.
If $s=2r$ and $\pi \in F_{2r,n}$, then after reordering $\bm \chi$ in the inner sum we can assume that $ J_{2u-1,2u}(\pi) \neq \varnothing $ for $ 1\leq u\leq r$ and that $J_{\mu,\nu}(\pi)=\varnothing$ otherwise. There are two subcases. If there exists $$\chi_{2u-1,j} \in \big\{ \overline{\chi_{\nu,i} } : \nu \notin \{ 2u-1,2u\}, i\leq n\big\} ,$$ then arguing as in the proof of Lemma~\ref{lemma sum over characters}, one shows that the inner sum is $ \ll n^2r^2\phi(q)^{ (n-1)r -1} . $ Otherwise, there must exist $$\chi_{2u-1,j} \in \big\{ \overline{\chi_{\nu,i} } : \nu \in \{ 2u-1,2u\}, i\leq n\big\} \smallsetminus \big\{ \overline{\chi_{2u-1,j} },\overline{\chi_{\pi(2u-1,j)}} \big\}  .$$ Once more, the inner sum is $ \ll n^2 r^2\phi(q)^{ (n-1)r -1} . $ The calculation is similar in the odd case.

We conclude that 
\begin{multline*} \V_{s,n}(q;\eta) = \frac {(-1)^{sn}}{\phi(q)^{sn}}\sum_{\pi \in I_{s,n}} \sum_{\substack{\bchi\in X_{s,n}\\ \pi(\bchi ) =\overline{ \bchi }\\ \forall (\mu,j), \exists! (\nu,i)\neq (\mu,j) : \chi_{\mu,j} =\overline{ \chi_{\nu,i}}}}  \sum_{\substack{ \bm \gamma\in \Gamma(\bchi)  \\  \forall (\mu,j),\,  \gamma_{\mu,j}+\gamma_{\pi(\mu,j)}=0 }} \!\!\!\!\!\!\!\!\!\!\!\widehat \eta(\bm \gamma) \Delta_{s}(\bm\gamma) \\ +O\Big( (|G_{s,n}|+n^2r^2|F_{s,n}|)\frac{(K_{\delta,\eta}\log q)^{\frac{sn}2}}{\phi(q)^{ \frac{(n+1)s}2 +1 +\frac{\delta_{2\nmid s}}2}} \Big).
\end{multline*}
Now, by Lemmas~\ref{lemma sum over characters LI} and~\ref{lemma sum over characters odd LI}, the sum over $\pi \in G_{s,n}$ is absorbed in the error term. The remaining terms are
$$ = \frac {(-1)^{sn}}{\phi(q)^{sn}}\sum_{\pi \in F_{s,n}} \sum_{\substack{ \bchi\in X_{s,n}\\  \pi(\bchi ) =\overline{ \bchi }\\ \forall (\mu,j), \exists! (\nu,i)\neq (\mu,j) : \chi_{\mu,j} =\overline{ \chi_{\nu,i}}}}  \sum_{\substack{ \bm \gamma\in \Gamma(\bchi)  \\  \forall (\mu,j),\,  \gamma_{\mu,j}+\gamma_{\pi(\mu,j)}=0 }} \widehat \eta(\bm \gamma).  $$
Applying Lemma
~\ref{lemma sum over characters} and~\ref{lemma sum over characters odd} 
yields the claimed estimate on $\V_{s,n}(q;\eta)$.
\end{proof}

\end{document}